\documentclass[11pt]{article}
\usepackage{algorithm}
\usepackage{algpseudocode}
\usepackage{hyperref}
\usepackage{tikz}
\usetikzlibrary{patterns}
\usepackage{array}
\usepackage[stretch=10]{microtype}
\usetikzlibrary{shapes,arrows,positioning, fit}
\usepackage{caption}
\usepackage{subcaption}

\usepackage{amsmath, amssymb, amsthm}
\usepackage{mathtools}
\usepackage{enumerate}
\usepackage{dsfont} 
\usepackage{bm}

\newcommand{\reals}{\mathds{R}}
\newcommand{\naturals}{\mathds{N}}
\newcommand{\integers}{\mathds{Z}}

\newcommand{\setsep}{\;|\;}
\newcommand{\ddd}{\,\mathrm{d}}
\newcommand{\ddS}{\,\mathrm{d}S}

\newcommand{\dist}{\ensuremath{\operatorname{dist}}}
\newcommand{\sgn}{\ensuremath{\operatorname{sgn}}}
\newcommand{\ind}[1]{\mathds{1}_{#1}}
\DeclareMathOperator{\Div}{div}

\newcommand{\vx}{\mathbf{x}}
\newcommand{\vy}{\mathbf{y}}
\newcommand{\vn}{\mathbf{n}}

\newcommand{\vs}{\mathbf{s}}
\newcommand{\va}{\mathbf{a}}
\newcommand{\vb}{\mathbf{b}}
\newcommand{\veta}{\boldsymbol{\eta}}

\newcommand{\ve}[1]{\mathbf{#1}}

\newcommand{\mat}[1]{\mathbf{#1}}
\newcommand{\matA}{\mat{A}}
\newcommand{\matB}{\mat{B}}

\theoremstyle{definition}
\newtheorem{definition}{Definition}[section]
\theoremstyle{plain}
\newtheorem{theorem}[definition]{Theorem}
\newtheorem{proposition}[definition]{Proposition}
\newtheorem{question}[definition]{Question}
\newtheorem{corollary}[definition]{Corollary}
\newtheorem{lemma}[definition]{Lemma}
\newtheorem{example}[definition]{Example}
\theoremstyle{remark}
\newtheorem{remark}[definition]{Remark}
\definecolor{nonnegWeights}{RGB}{190,125,195}
\definecolor{linLay}{RGB}{100,125,210}

\DeclareMathOperator{\loss}{loss}
\DeclareMathOperator{\Loss}{Loss}
\DeclareMathOperator{\Var}{Var}
\DeclareMathOperator{\pVar}{pVar}
\DeclareMathOperator{\MVar}{MVar}
\DeclareMathOperator{\EMVar}{EMVar}
\DeclareMathOperator{\Lip}{Lip}
\DeclareMathOperator{\modi}{mod}

\DeclareMathOperator*{\esssup}{ess\,sup}

\definecolor{changedColor}{RGB}{25,220,160}

\usepackage[numbers]{natbib}
\def\softd{{\leavevmode\setbox1=\hbox{d}%
		\hbox to 1.05\wd1{d\kern-0.4ex{\char039}\hss}}}
	
\usepackage[left=2cm,right=2cm,top=2.0cm,bottom=2cm]{geometry}
\usepackage{tcolorbox}

\begin{document}

	\title{Estimatable variation neural networks and their application to scalar hyperbolic conservation laws}
		
	\author{\texorpdfstring{M\' aria Luk\' a\v cov\' a-Medvi{\leavevmode\setbox1=\hbox{d}%
				\hbox to 1.05\wd1{d\kern-0.3ex{\char039}\hss}}ov\' a}{M\' aria Luk\' a\v cov\' a-Medvid'ov\' a}$^{*}$, Simon Schneider$^{*}$
		}

\maketitle

\bigskip

\bigskip
\centerline{$^*$ Institute of Mathematics, Johannes Gutenberg-University Mainz}
\centerline{Staudingerweg 9, 55 128 Mainz, Germany}
\centerline{lukacova@uni-mainz.de, simon.schneider@uni-mainz.de}
	
	\begin{abstract}
		We introduce estimatable varation neural networks (EVNNs), a class of neural networks that allow a computationally cheap estimate on the $BV$ norm motivated by the space $BMV$ of functions with bounded M-variation. We prove a universal approximation theorem for EVNNs and discuss possible implementations. We construct sequences of loss functionals for scalar hyperbolic conservation laws for which a vanishing loss leads to convergence. Moreover, we show the existence of sequences of loss minimizing neural networks if the solution is an element of $BMV$.
		Several numerical test cases illustrate that it is possible to use standard techniques to minimize these loss functionals for EVNNs.
	\end{abstract}

{\bf Keywords:} neural network, total variation, convergence, universal approximation, scalar hyperbolic conservation laws

{\bf MSC codes:} 65M99, 35L65, 35L60

	\section{Introduction}\label{sec:introduction}
	
	While neural networks achieve impressive results over a wide variety of tasks, one major drawback of neural networks is their ``black box'' nature. Due to the complex construction based on compositions of linear and nonlinear mappings, there is a lack of control over the regularity of neural networks. 
	For arbitrary $k\in \naturals_0$,
 	upper bounds on $C^k$ norms of deep neural networks are easy to derive. 
 	But it is unclear if they are sharp enough to be of practical use. Additionally, 
 	if the $(k-1)$-th derivative of the target function is not Lipschitz continuous, the $C^k$ norm of the network is expected to blow-up during training. 
	
	In this work, we are concerned with using neural networks to approximate the solution $u$ of a scalar hyperbolic conservation law. It is well-known that solutions of scalar hyperbolic conservation laws can develop discontinuities even starting from smooth initial data. Therefore, $C^k$ based regularity cannot be expected for these functions.
	
	For training, it is natural to use sequences of loss functions $(\Loss_n)_{n\in \naturals}$ based on integration. Here, $n$ corresponds to the integration error tolerance and, e.g., the number of test functions used. Numerical integration with error estimates in terms of $C^k$ norms are unpractical in this setting. Therefore,
	integrals of neural networks can only be rigorously approximated with reasonable computational cost in an average sense, i.e., using Monte Carlo approximation.
	Moreover, possible oscillations of sequences of neural networks $(\eta_n)_{n\in\naturals}$ cannot be controlled. 
	This impacts the availabilty of theoretical results for neural network based approaches.
	Let $(\eta_n)_{n\in\naturals}$ be a sequence of neural networks and assume that $(\Loss_n)_{n\in\naturals}$ is any sequence of explicitly computable loss functions for which the computational cost does not depend on a $C^k$ norm of the network.
	For the reasons outlines above there do not exist results of the following kind for scalar hyperbolic conservation laws:
	\begin{theorem}[Informal]\label{thm:informal}
		Given a sequence of neural networks $(\eta_n)_{n\in\naturals}$ satisfying 
		\begin{equation*}
			\Loss_n(\eta_n)\xrightarrow{n\to\infty}0,
		\end{equation*}
		then $\eta_n \to u$ in a suitable sense.
	\end{theorem}
	
	To provide such theoretical guarantees, we propose to use a novel type of neural network termed estimatable variation neural network (EVNN). A regularity control in form of an upper bound on the total variation (or $W^{1,1}$ norm) is available for EVNNs. We present a straightforward to implement architecture to construct EVNNs for approximating functions depending (non-smoothly) on a small number of variables. We show a universal approximation result for networks with this architecture. Moreover, we are able to integrate these networks numerically up to arbitrary precision and control possible oscillations. These properties can be used to construct explicitly computable loss functions $(\Loss_n)_{n\in\naturals}$ for both ordinary differential equations (ODEs) as well as scalar hyperbolic conservation laws whose minimization leads to the convergence to the unique solution in the sense of Theorem~\ref{thm:informal}, see Theorem  \ref{thm:convergence} below.
	
	Additionally, we show that if the solution is a function of bounded M-variation, a notion we define below, then for each $n$ there always exists a network $\eta$ such that $\Loss_n(\eta)$ is arbitrarily close to $0$. In other words, if the solution is an element of a certain function space, the network has a sufficient number of parameters and the optimizing algorithm finds sufficiently good local minima, loss minimization leads to the unique solution. 

	As is typical for solution learning approaches, the proposed algorithms are not competitive with conventional numerical methods in regards to computational cost for solving forward problems. But they are convenient to modify for additional data fitting constraints or for parameter search. Moreover, the strong theoretical results show that they are a promising starting point for further research on the use of neural networks for (partial) differential equations.

	The structure of the present paper is as follows. In the rest of Section~\ref{sec:introduction} we give a brief overview of existing approaches which use neural networks as trial functions to solve hyperbolic conservation laws. In Section~\ref{sec:BV_funcs} we introduce the space $BMV$ of functions with bounded M-variation which motivates our architecture and provides a theoretical framework to prove error estimates. A simple architecture to construct EVNNs will be presented in Section~\ref{sec:EVNNs} along with theoretical results such as a universal approximation property in $BMV$, see Theorem~\ref{thm:UniBMVapprox}. In Section~\ref{sec:applications} we present our main results: the convergence analysis of EVNNs for scalar hyperbolic conservation laws. Our theoretical results are illustrated by numerical simulations in Section~\ref{sec:num_experiments}.
	
	\subsection{Solution learning for scalar hyperbolic conservation laws}
	
	In the context of hyperbolic partial differential equations several different approaches to approximate solutions by neural networks have been proposed. 
	
	The physics-informed neural network (PINN, see \cite{RaissiPINNs}) approach was used to approximate solutions of a hyperbolic two phase transport model in conservation form by Fuks and Tchelepi in \cite{fuks2020limitations}. They reported low accuracy if solutions exhibited shocks. Adding artificial viscosity is reported to improve the performance in this case if a large enough viscosity coefficient is chosen.

	In \cite{cPINNs} Jagtap et al.~proposed the so-called \textit{conservative physics-informed neural network} (\textit{cPINN})  for solving forward and inverse problems concerning nonlinear conservation laws. The main idea was to decompose the computational domain into subdomains and to employ individual neural networks on each subdomain. Each network was trained with the classical PINN approach with additional \textit{interface conditions} on the interfaces of subdomains.
	
	Patel et al.~proposed to use a loss function reminiscent of the finite volume method in \cite{cvPINNs}. They numerically investigated the effects of the addition of artificial viscosity and penalization of total variation and entropy inequality, respectively. Both forward and inverse problems were studied.

	De Ryck et al.~introduced \textit{weak physics informed neural networks (wPINNs)} for scalar hyperbolic conservation laws in one dimension in \cite{de2022wpinns}. These are neural networks which are trained using the min-max optimization problem for the Kruzhov entropy residual
	\begin{align}\label{eq:MishraEntroyResidual}
		\mathcal{R}(\eta_1,\eta_2,c) \coloneqq - \int_{[0,T]}\int_{[0,1]}  \bigg( &\left| \eta_1(t,x) - c \right| (\partial_t \eta_2)(t,x)
		\notag
		\\ & + \sgn (\eta_1(t,x)-c) (F(\eta_1(t,x))-F(c)) \bigg)(\partial_x \eta_2)(t,x) \ddd x \ddd t.
	\end{align}
	Here, the Kruzkov entropy is characterized by the constant $c\in \reals$ and the test function $\eta_2$ is taken from the set of neural networks with a fixed architecture. $\mathcal{R}$ is maximized with respect to $c$ and $\eta_2$ and then minimized with respect to a neural network $\eta_1$ approximating the weak admissible solution of the equation. In order to approximate the integral in (\ref{eq:MishraEntroyResidual}), the Monte Carlo quadrature method is applied. Moreover, the network $\eta_2$ is approximately normalized with respect to its $H^1$ seminorm which again is calculated using the Monte Carlo quadrature method. De Ryck et al.~derived an $L^1$-error estimate between $\eta_1$ and the unique weak admissible solution. The estimate is valid with high probability under the assumption that a sufficiently good solution $\eta_1$ of this min-max problem is found. More precisely, $\eta_1$ has to satisfy $\sup_{\eta_2,c} \mathcal{R}_h(\eta_1,\eta_2,c) < \epsilon$ for the estimate to hold, where $\epsilon$ denotes a tolerance depending on the upper bound on the error and $\mathcal{R}_h$ denotes the approximation of $\mathcal{R}$ in which integrals are replaced by the corresponding Monte Carlo approximations. While the approach is quite elegant, to the best of our knowledge there is no way of practically ensuring that the computed approximate solution of the min-max problem is sufficiently good in the sense described above.
	
	Besides giving theoretical insights into why the standard PINN approach might not work for nonlinear hyperbolic conservation laws, a slight modification of wPINNs was investigated by Chaumet and Giesslemann in \cite{chaumet2023}. They replaced the min-max optimization of the entropy residual (\ref{eq:MishraEntroyResidual}) with a min-max problem corresponding to the weak formulation of the conservation law and a min-max problem corresponding to the entropy inequality for a single fixed entropy-entropy flux pair. Moreover, they avoid the normalization of the test function by solving a suitable dual problem instead of maximizing the residuals with respect to the test function. The authors reported improved accuracy of their modified wPINNs for the numerical test cases investigated in \cite{de2022wpinns}. Finally they extended their proposed method to weak Dirichlet boundary conditions and system of conservations laws in one space dimension.
	
	Chen et al.\ introduced the TGPT-PINN in \cite{chen2024tgpt}. They use pretrained neural networks as activation functions to speed up the computation of new solutions by drastically reducing the number of trainable parameters.
	
	In \cite{RichterPowell2022} Richter-Powell et al. proposed two different network architectures to approximate smooth, divergence free functions. They showed that any smooth divergence free function can be approximated by both architectures. Further, they propose to use divergence free neural networks to approximate the solution-flux vector $(\eta_1, \veta_2) \approx (u,\ve{F}(u))$ consisting of both the solution $u$ and the flux $\ve{F}(u)$. They then train the network to satisfy initial and boundary conditions and $\ve{F}(\eta_1) \approx \veta_2$. 
	
	\section{Functions of bounded (M-)variation}\label{sec:BV_funcs}

	To motivate the network architecture that we will present in Section~\ref{sec:EVNNs}, we first introduce the function space of functions of bounded \textit{M-variation} on $U$, $BMV(U)$, for $U\subset \reals^d$. It is the linear hull of all monotone functions on $U$ with the notion of monotonicity defined below. 
	This function space is inspired by the classical space of functions of bounded pointwise variation in one variable, see e.g.\ \cite{ambrosio2000functions,leoni2017first}. For these functions a straightforward bound on the variation in the $BV$ sense is available. We will discuss the connection of the M-variation and the variation in the $BV$ sense for $d>1$ in Section~\ref{sec:VarEst}.
	Other notions of variation for multivariate functions which generalize the functions of bounded pointwise variation can be found in \cite{appell2013bounded, clarkson1933definitions}.
	
	For $\vx,\vy\in\reals^d $ we define the partial orders $<$ and $\leq$ by $\vx<\vy$ if $x_i<y_i$ for all $i=1,\dots, d$ and $\vx\leq \vy$ off $x_i\leq y_i$ for all $i=1,\dots, d$.
	Monotonicity of functions from $U \subset \reals^d$ to $\reals$ is to be understood with respect to $\leq$, i.e., we say a function $f\colon U \to \reals$ is nondecreasing if $f(\vx)\leq f(\vy)$ for each $\vx,\vy\in U$ with $\vx\leq \vy$. We say, $f$ is nondecreasing with respect to $\vs \in \{-1,1\}^d$ if $f(\vx)\leq f(\vy)$ for each $\vx,\vy\in U$ with $\vs \odot \vx\leq \vs \odot \vy$ where $\odot$ denotes elementwise multiplication.

	We define the set of \textit{directions of monotonicity} to be 
	\begin{equation*}
	DM \coloneqq \{-1,1\}^{d} \eqqcolon \{\vs_i \setsep i\in J\}, \qquad J\coloneqq \{1,\dots,2^d\}.
	\end{equation*}
	\begin{definition}[$BMV$]\label{bmv}
		We say that a function $f\colon U\subset\reals^d \to \reals$ is of bounded M-variation if there exist nondecreasing functions $f_i \colon \reals^d \to \reals$, $i\in J$, with
		\begin{equation}\label{eq:mono_decomp}
		f(\vx) = \sum_{i\in J} f_i(\mathbf{s}_i \odot \vx) \quad \text{for all }\vx\in U,
		\end{equation}
		where $\odot$ denotes the componentwise multiplication in $\reals^d$. We denote the set of all functions of bounded M-variation on $U$ by $BMV(U)$. Moreover, we set
		\begin{equation*}
		\MVar(f; U) \coloneqq \inf_{f_i} \sum_{i\in J} \left( \sup_{\vx \in U} f_i(\vs_i \odot \vx) - \inf_{\vx \in U} f_i(\vs_i \odot \vx) \right),
		\end{equation*}
		where the first infimum is taken over all decompositions of $f$ of the form (\ref{eq:mono_decomp}).
	\end{definition}

	It is no restriction to consider nondecreasing functions $f_i$ defined on all of $\reals^d$ in (\ref{eq:mono_decomp}). Indeed, for any nondecreasing function $f\colon U \to \reals$ with $\inf_U f>-\infty$ and $\sup_U f < \infty$, we define the nondecreasing extension
	\begin{align}\label{eq:definitionE+-}
		E[f;U](\vx) &=
		\sup \left( \{f(\vy) \setsep \vy\in U,\, \vy \leq \vx\} \cup \{\inf_{\vy\in U} f(\vy)\}\right)
		.
	\end{align}
	We sometimes write $E[f] = E[f;U]$ if there is no ambiguity regarding $U$.
	Note that $E[f](\vx) = f(\vx)$ for all $\vx\in U$ and 
	\begin{equation*}
		-\infty < \inf_{\vy\in U} f(\vy) \leq E[f](\vx) \leq \sup_{y\in U} f(y)<\infty
	\end{equation*}
	for all $\vx\in\reals^d$. 

	Clearly, $BMV(U)$ is a vector space. Moreover,
	every function of the following form is of bounded M-variation: 
	\begin{equation*}
		f(\vx) = 
		\begin{cases}
		r_1 &\text{for }(\vx-\mathbf{a})\cdot \vn>0\\
		r_2 &\text{otherwise,}
		\end{cases}
		\qquad \text{with }r_1,r_2\in\reals,\; \vn, \va\in \reals^d
		.
	\end{equation*}
	Indeed, choosing $i_0\in J$ such that $(r_1-r_2)\,\mathbf{s}_{i_0} \odot \vn \geq 0$, we find that
	$f(\mathbf{s}_{i_0} \odot \vx_2) - f(\mathbf{s}_{i_0} \odot \vx_1) \geq 0$
	for $\vx_2\geq \vx_1$.
	Thus, we can choose $f_{i_0}=f$ and $f_{i}=0$ for all $i\neq i_0$ in (\ref{eq:mono_decomp}).
	Additionally, we have the following lemma concerning the extension by zero of a function in $BMV$.
	\begin{lemma}\label{lem:extensionByZero}
	Given $H = \prod_{i=1}^d [a_i, b_i]\subset \reals^d$, let $f\in BMV(H)$. For $j\in \{1,\dots,d\}$ and $c\in [a_j,b_j]$ we define
	\begin{equation*}
		\tilde f \colon \tilde H \to \reals,\quad \tilde f(\vx) = \begin{cases}
			f(\vx) &\text{if }x_j\in I_1\\
			0 &\text{if }x_j \in I_2,
		\end{cases}
	\end{equation*}
	with 
	$
		(I_1, I_2) \in \Big\{ 
		\big((c,b_j], (-\infty, c]\big),\;\;
		\big([c,b_j], (-\infty, c)\big),\;\; 
		\big([a_j,c), [c, \infty)\big),\;\;  
		\big([a_j,c], (c, \infty)\big) \Big\}
	$
	and 
	\begin{equation*}
		\tilde H \coloneqq [a_1,b_1]\times \dots \times
		[a_{j-1},b_{j-1}]\times 
		\Big(I_1 \cup I_2\Big)
		\times [a_{j+1},b_{j+1}]
		\times \dots\times [a_d,b_d].
	\end{equation*}
	Then $\tilde f \in BMV(\tilde H)$ with 
	$
		\MVar(\tilde f; \tilde H)\leq \MVar(f) +  \max \{|f(\vx)| \setsep \vx \in H\}.
	$
	\end{lemma}
	\begin{proof}
		We define $\ve{a}=(a_1,\dots,a_d)$, $\ve{b}=(b_1,\dots,b_d)$.
		Note that for any $\vs\in \{-1,1\}^d$ and functions $f,g\colon H \to \reals$ which are either nondecreasing with respect to $\vs$ or $-\vs$, the product $[fg:\vx \mapsto f(\vx)g(\vx)]$ is a function of bounded M-variation with 
		\begin{equation*}
			\MVar(fg;H)\leq \MVar(f)\|g\|_\infty + \MVar(g)\|f\|_\infty.
		\end{equation*}
		The above can be shown analogously as in the case $d=1$ (see e.g.\ \cite[Proposition 1.10]{appell2013bounded}) by considering the functional
		\begin{equation*}
			\begin{aligned}
			\pVar_\vs[f] \coloneqq& \sup\Bigg\{ \sum_{i=1}^N |f(\vx_i) - f(\vx_{i-1})| \setsep N\in\naturals,\, \vx_0, \vx_1,\dots, \vx_N \in H \text{ with}
			\\
			&\vs\odot\left(\frac{\ve{a}+\ve{b}}{2} + \vs\odot \frac{\ve{a}-\ve{b}}{2}\right)
			\leq \vs\odot \vx_0 
			\leq \vs\odot\vx_1
			\leq \dots 
			\leq \vs\odot\vx_N 
			\leq
			\vs\odot\left(\frac{\ve{a}+\ve{b}}{2} - \vs\odot \frac{\ve{a}-\ve{b}}{2}\right)
			\Bigg\}.
			\end{aligned}
		\end{equation*}
		As $\ind{\{x_i \geq c\}}$ is either nondecreasing with respect to $\vs$ or $-\vs$ for any $\vs \in\{-1,1\}^d$, the proof is complete.
	\end{proof}

	In particular, since $BMV(U)$ is a linear space, it contains the space of all piecewise constant functions on rectilinear grids. The latter is typically used in the context of the finite volume method. 
	
	As another consequence of Lemma~\ref{lem:extensionByZero} it is also possible to extend $BMV$ functions periodically.
	\begin{corollary}\label{cor:periodic_extension}
		Given $H = \prod_{i=1}^d [a_i, b_i]\subset \reals^d$, let $f\in BMV(H)$. For $j\in \{1,\dots,d\}$ and $c\in [a_j,b_j]$ we define
		\begin{equation*}
			\tilde f \colon \tilde H \to \reals,\quad \tilde f(\vx) = \begin{cases}
				f(\vx) &\text{if }x_j\in [a_j,b_j]\\
				f(\vx - (b_j-a_j)\ve{e}_j) &\text{if }x_j \in (b_j,2b_j-a_j],
			\end{cases}
		\end{equation*}
		with 
		$
			\tilde H \coloneqq [a_1,b_1]\times \dots \times
			[a_{j-1},b_{j-1}]\times 
			[a_j, 2b_j-a_j]
			\times [a_{j+1},b_{j+1}]
			\times \dots\times [a_d,b_d].
		$
		Then $\tilde f \in BMV(\tilde H)$ with 
		$
			\MVar(\tilde f; \tilde H)\leq 2\Big(\MVar(f) +  \max \{f(\vx) \setsep \vx \in H\}\Big).
		$
	\end{corollary}

	\subsection{Integrating functions of bounded M-variation}\label{sec:IntegratingBMVFunctions}
	
	Monotone functions are measurable \cite{chabrillac1987continuity}, and thus $BMV$ functions are as well. Moreover, $BMV$ functions are bounded on bounded sets. Therefore, $BMV(U)$ is a subspace of $L^p(U)$ for all $p\in [1,\infty]$ for each bounded, measurable set $U$.
	
	In \cite{papageorgiou1993integration} Papageorgiou investigated optimal approximation of integrals of monotone functions in several variables. In this paper we will  use the quadrature scheme from \cite[Theorem 2]{papageorgiou1993integration} adopted to approximate integrals of the form 
	\begin{equation*}
	\int_U F \left(f^{(1)}(\vx),\dots, f^{(m)}(\vx)\right) \ddd \vx. 
	\end{equation*}
	Here, $F$ denotes a Lipschitz continuous function, $m\in \naturals$ is arbitrary and $f^{(k)}\in BMV(U)$ for each $k\in \{1,\dots,m\}$.
	For simplicity, we restrict ourselves to integration on hyperrectangles $U=\prod_i [a_i,b_i]$.
	Approximation of integrals on arbitrary measurable sets is discussed in Remark~\ref{rem:integration_on_arbitrary_measurable_sets}.
	
	Let $F\colon \reals^m \to \reals$ have the Lipschitz constant $L$ on $(f^{(1)}(\cdot),  \dots, f^{(m)}(\cdot))(U)$ with respect to the $\ell^1$-norm on $\reals^m$. Denote $V=\prod_i (b_i-a_i)$ and $h_i = (b_i-a_i)/k$. For $C_\mathbf{j} = \prod_{i=1}^d [a_i+j_i h_i, a_i+(j_i +1)h_i]$
	with $\ve{j} = (j_1,\dots, j_d) \in \{0,1,\dots, k-1\}^d$, we fix $\vx_\ve{j}\in C_\ve{j}$.
	Using similar arguments as in the proof of \cite[Theorem 2]{papageorgiou1993integration}, it is straightforward to derive the following theorem.
	\begin{theorem}\label{papageorgiou_integration_theorem}
	Let $k\in\naturals$ satisfy $k\geq {d \,L\, V \epsilon^{-1}\sum_{l=1,...,m} \MVar(f^{(l)})}$ for $\epsilon>0$. Then
	\begin{equation}\label{ineq:Integration_estimate}
		\left| \int_{H} F\big(f^{(1)}(\vy),...,f^{(m)}(\vy)\big) \ddd \vy - \frac{V}{k^d} \sum_{\mathbf{j}\in \{0,\dots, k-1\}^d} F\big(f^{(1)}(\vx_{\ve{j}}),...,f^{(m)}(\vx_{\ve{j}}) \big) \right|< \epsilon.
	\end{equation}
	\end{theorem}

	The error rate is asymptotically optimal for monotone functions \cite[Theorem 1]{papageorgiou1993integration} and thus also for $BMV$ functions. It is noteworthy that the concrete choice of the points $\vx_{\ve{j}} \in C_{\mathbf{j}}$ plays no role in the derivation of the error estimate (\ref{ineq:Integration_estimate}). In particular, it is possible to choose these points randomly. This yields a random quadrature formula with an average error which does not suffer from the curse of dimensionality, while still satisfying the accuracy derived above in the worst case scenario, cf.~\cite{papageorgiou1993integration}.
	
	In low dimensions, it seems to be more practical to choose quadrature points $\vx_{\ve{j}}$ in such a way, that they coincide with each other as much as possible. Therefore, the $\vx_{\ve{j}}$ will be chosen from the set of vertices of the hyperrectangles $C_{\mathbf{j}}$ in all numerical experiments presented in Section~\ref{sec:num_experiments}. If $k$ is even, this choice reduces the amount of sampling points by $2^d$.
	
	\begin{remark}
		Note that Theorem~\ref{papageorgiou_integration_theorem} implies that it is possible to approximate the $L^p$ norm of every $BMV$ function up to arbitrary accuracy for all $1\leq p < \infty$. In particular, the $L^p$ norm of differences of $BMV$ functions can be calculated up to a given accuracy.
	\end{remark}

\begin{remark}\label{rem:integration_on_arbitrary_measurable_sets}
	The restriction of considering only hypercubes can be relaxed easily in the following sense:
	Note that for every bounded domain $U \subset \reals^d$ there exists a hyperrectangle $H$ with $H \supset U$. Moreover, it is possible to approximate $\ind{U}$, the characteristic function of the set $U$, in $L^1(H)$ by a function $g\in BMV(H)$, see Appendix~\ref{App:approxProps}. Using  the extension operator $E$ defined in (\ref{eq:definitionE+-}), the functions $f_i$ can be extended to $H$. Consequently, we can apply Theorem~\ref{papageorgiou_integration_theorem} to numerically approximate $\int_U f \ddd \vx \approx \int_H fg \ddd \vx$.
	Note however that the necessary number of quadrature points to achieve a prescribed accuracy now also depends on the M-variation of $g$.
\end{remark}
	
	\subsection{Relation between total variation and M-variation}\label{sec:VarEst}
	In this section, we show how the M-variation can be used to bound the total variation. Moreover, we investigate how to sharpen this bound for monotone functions.
	First, we consider functions defined on the hyperrectangle $U=\prod_{j=1}^{d}[a_j,b_j]$, $a_j < b_j$ for all $j=1,\dots,d$.
	Given a function of bounded M-variation 
	\begin{equation*}
	f(\vx) = \sum_{i\in J} f_i(\mathbf{s}_i \odot \vx)
	\end{equation*}
	with $f_i \in C^1(\reals^d)$ for all $i\in J$ we find that 
	$
	\Big|\partial_{x_k} f(\vx)\Big| \leq \sum_{i\in J} (\vs_i)_k \partial_{x_k} f_i(\vs_i \odot \vx).
	$
	This allows us to invoke the fundamental theorem of calculus to derive the following estimate of the derivatives of $f$ in the $L^1$ norm:
\begin{equation}\label{eq:basicW11Estimate}
	\| \partial_{x_k} f \|_{L^1(U)}
	\leq \sum_{i\in J} (\vs_i)_k 
	\left(
	\int_{U_k^+} 
	\Big(   
	f_i(\vs_i \odot \vx )
	\Big)
	\ddd S \vx
	-
	\int_{U_k^-} 
	\Big(   
	f_i(\vs_i \odot \vx )
	\Big)
	\ddd S \vx
	\right).
\end{equation}
with $U_k^+ = \prod_{j=1}^{k-1}[a_j,b_j] \times \{b_k\} \times \prod_{j=k+1}^{d}[a_j,b_j]$ and $U_k^- = \prod_{j=1}^{k-1}[a_j,b_j] \times \{a_k\} \times \prod_{j=k+1}^{d}[a_j,b_j]$.
	Due to the previous results, we can numerically calculate the right hand side up to an arbitrary accuracy. This allows us to estimate the total variation of $f$ based on a finite, predeterminable number of sampling points.
	
	The following lemma shows that the M-variation already by itself provides a computationally cheap estimate on the total variation of a monotone function.	
	We define for any open set $U\subset \reals^d$
	\begin{equation*}
		\Var_\infty(f; U) \eqqcolon \sup \left\{ \int_U f \Div(\varphi) \ddd \vy \setsep \varphi \in C_c^1(U), \quad \|\varphi\|_\infty = \max_{i=1,...,d} \|\varphi_i\|_\infty \leq 1 \right\}.
	\end{equation*}

	\begin{lemma}\label{lem:UpperTVBound}
		Let $U$ be a bounded Lipschitz domain, $f:U\to\reals$ be such that $\vx \mapsto f(\vs \odot \vx)$ is nondecreasing for a fixed $\vs \in \{-1,1\}^d$ and all $\vx$ with $\vs\odot \vx \in U$. Then
		\begin{equation*}
		\frac{1}{\sqrt{d}}\Var(f; U) \leq	\Var_\infty(f; U) \leq  \frac{\sqrt{d} |\partial U|}{2} \MVar(f;U).
		\end{equation*}
	\end{lemma}

	\begin{proof}
		With $f_n = \nu_n * f$ we find that for any fixed $\varphi \in C^1_c(U;\reals^d)$ with $\sup_\vx\|\varphi(\vx)\|_{\infty}\leq 1$, we have
		\begin{align*}
			\int_U f \Div(\varphi) \ddd \vx &= \lim\limits_{n\to\infty}-\int_U \nabla f_n(\vx) \cdot  \varphi(\vx) \ddd \vx
			= 
			\liminf\limits_{n\to\infty}\int_{\partial U} f_n(\vx) \, \vn \cdot  \vs \ddS \vx
			\\
			&= 
			\liminf\limits_{n\to\infty} \left(
			\int_{\partial U} f_n(\vx) \, (\vn \cdot \vs)^+ \ddS \vx
			+
			\int_{\partial U} f_n(\vx) \, (\vn \cdot \vs)^- \ddS \vx
			\right)	
			\leq \frac{\sqrt{d} |\partial U|}{2} \MVar(f;U).
		\end{align*}
	\end{proof}
	A small total variation does not necessarily imply a small M-variation. For $a \in (0,1)$ consider  $f_a\colon (0,1)^2 \to \reals$, 
	\begin{equation*}
		f_a(x,y) = \begin{cases}
			1 \quad & \text{if }x,y<a,\\
			0 \quad & \text{otherwise.}
		\end{cases}
	\end{equation*}
	Clearly, $\MVar(f_a)=1$ but $\Var(f_a) = 2a$ illustrating that there cannot exist a positive constant $C$ with $\Var(f)\geq C \MVar(f)$ for all monotone functions. 
	Still, by decomposing the computational domain $U$ in hyperrectangles, it is possible to get a lower bound on the total variation of monotone continuous functions. 
	For this purpose we define the set of cubes
	\begin{equation*}
		\mathcal{C}_h \coloneqq \left\{ c = \prod_{i=1}^{d} [z_i,z_i+h] \setsep z=(z_1,\dots,z_d)\in h\mathds{Z}^d \right\}.
	\end{equation*}
	The following theorem can be seen as a generalization of the estimate provided by (\ref{eq:basicW11Estimate}).
	\begin{theorem}\label{thm:MVarRefinement}
		Let $U\subset\reals^d$ be a bounded Lipschitz domain. Let $f\in C(\overline{U})$ be monotone. Then there exist constants $C_1,C_2>0$ only depending on $U$ such that for all $h>0$
		\begin{equation*}
			\limsup\limits_{h \searrow 0}C_1 \sum\limits_{\substack{c \in \mathcal{C}_h\\
					c \cap \overline{U} \neq \emptyset}}
			\MVar(f;c) h^{d-1}
			\leq
			\Var(f;U) 
			\leq 
			C_2 \sum\limits_{\substack{c \in \mathcal{C}_h\\
			c \cap \overline{U} \neq \emptyset}}
			\MVar(f;c) h^{d-1}.
		\end{equation*}
	\end{theorem}
	
	\begin{proof}
		We begin with the first inequality. Let $\epsilon>0$ be arbitrary. 
		We extend $f$ continuously on $\reals^d$ and choose $h$ small enough such that
		$
			|f(\vx)-f(\vy)| < \epsilon
		$
		for all $\vx,\vy$ in some neighborhood of $\overline{U}$ with $\|\vx-\vy\|\leq \delta \coloneqq \sqrt{d} h$.
		We denote by $f^U \coloneqq  \int_U f(x) \ddd x$ the average of $f$ on $U$. Note that 
		\begin{equation*}
			\MVar(f;c) = \MVar(f-f^U;c).
		\end{equation*}
		Let $\vs\in \{-1,1\}^ d$ be such that $\vx\mapsto f(\vs\odot \vx)$ is nondecreasing on $\overline{U}$. We denote by $[c]$ the equivalence class of $c\in\mathcal{C}_h$, $c\subset \overline{U}$, with respect to the relation $\sim_\vs$,
		\begin{equation*}
			c \sim_\vs c' \quad \Leftrightarrow \quad 
			\text{there exists }m \in \mathds{Z} \text{ such that }c+mh \vs   = c'.
		\end{equation*}
		For $c\in [c']$ with $c\cap \partial U = \emptyset$ we find
		$
			\left( c \pm h \vs \right) \cap \overline{U} \neq \emptyset.
		$
		In this case, 
		\begin{equation*}
			\MVar(f;c-h\vs)+
			\MVar(f;c)+
			\MVar(f;c+h\vs)
			=
			\max_{\vx\in \overline{U} \cap (c+h\vs)} f(\vx)
			-
			\min_{\vx\in \overline{U} \cap (c-h\vs)} f(\vx).
		\end{equation*}
		Consequently, choosing $h$ sufficiently small we obtain with $I_\delta(\partial U) \coloneqq \{x \setsep \dist(x,\partial U)<\delta\}$
		\begin{align*}
			\sum\limits_{\substack{c \in \mathcal{C}_h\\
			c \cap \overline{U} \neq \emptyset}}
			\MVar(f;c) h^{d-1}
			&=
			\sum\limits_{[c']}\sum\limits_{c \in [c']}
			\MVar(f-f^U;c) h^{d-1}
			\leq
			4
			\sum\limits_{[c']}\sum\limits_{\substack{c \in [c']\\
			c \cap \partial U \neq \emptyset}}
			\max_{\vx \in c}|f(\vx)-f^U| \frac{|c|}{2h}
			\\
			&=
			4\sum\limits_{\substack{c \in \mathcal{C}_h\\
					c \cap \partial U \neq \emptyset}}
			\max_{\vx \in c\cap \overline{U}}|f(\vx)-f^U| \frac{|c|}{2h}
			\leq
			4 \int_{I_\delta(\partial U)} \frac{1}{2h} \left(|f(\vx)-f^U| +\epsilon\right) \ddd \vx.
		\end{align*}
		Therefore, as a consequence of \cite[Theorem 2.106]{ambrosio2000functions}, we find
		\begin{align*}
			\limsup_{h\searrow0} 
			\sum\limits_{\substack{c \in \mathcal{C}_h\\
					c \cap \overline{U} \neq \emptyset}}
			\MVar(f;c) h^{d-1}
			&\leq 
			C \left( \int_{\partial U} |f(\vx)-f^U| \ddd \vx + \epsilon \right).
		\end{align*}
		By \cite[Theorem 3.44]{ambrosio2000functions} and \cite[Theorem 3.87]{ambrosio2000functions} we recover the first inequality since $\epsilon$ was arbitrary.
		
		Next we prove the second inequality. Note that due to the continuity of $f$ we have
		\begin{equation*}
		\Var(f;U) = 
		\sum\limits_{\substack{c \in \mathcal{C}_h\\
				c \cap \overline{U} \neq \emptyset}}
			\Var(f;U\cap c).
		\end{equation*}
		Furthermore, we note that $\partial (U \cap c) \subset (\partial U \cap c) \cup \partial c$. Clearly, $|\partial U \cap c| \leq C h^{d-1}$ where $C>0$ only depends on $U$. Thus, we may apply Lemma~\ref{lem:UpperTVBound} to derive the second inequality.
	\end{proof}

	\section{Estimatable Variation Neural Networks (EVNNs)}\label{sec:EVNNs}
	
	In this section we propose a new class of neural networks with estimatable variation. Clearly, since every fully connected neural network with locally Lipschitz continuous activation function is Lipschitz continuous on $[0,1]^d$, most neural networks are $BMV$ functions in the sense of Section~\ref{sec:BV_funcs}. But it is highly nontrivial to compute the M-variation of a given network even on simple domains. To be able to efficiently estimate the M-variation of the network on hyperrectangles, we propose the following basic architecture.
	\begin{definition}
	We say that a network has an architecture of type (A) if it can be decomposed as
	\begin{equation}\label{naiveEVNN}
	\eta(\vx) = \sum_{i=1}^{K}\matA_{L+1,i}\cdot \hat\sigma \bigg(\matA_{L,i}\cdot \hat\sigma\Big(...\hat\sigma(\matA_{1,i}\cdot (\vs_i\odot \vx)+\vb_{1,i})\Big)+\vb_{L,i}\bigg) +{\vb_{L+1,i}}
	\eqqcolon
	\sum_{i=1}^{K} \eta_i(\vs_i \odot \vx),
	\end{equation}
	with a nondecreasing activation function $\sigma$, $\vs_1,\dots,\vs_K\in\{-1,1\}^d$, $\vb_{l,i}\in\reals^{m_l}$, $l=1,\dots,L+1$ , $i=1,\dots,K$, and matrices $\matA_{l,i}\in \reals_{\geq 0}^{m_l,m_{l-1}}$, $l=1,...,L+1$, $i=1,...,K$ with nonnegative elements. We refer to the networks $\eta_i$ as \textit{subnetworks}.
	\end{definition}
	The nonnegativity of the weight matrices $\matA_{l,i}$ ensures that each subnetwork $\eta_i$ is nondecreasing. In this way the above definition mimics the definition of $BMV$ functions. 
	
	The major drawback of this architecture is obvious. If the target function is an arbitrary element in $BMV$, up to $K=2^d$ summands are necessary in (\ref{naiveEVNN}). Consequently, the above architecture suffers from the curse of dimensionality even without considering the learning stage. If the target function can be decomposed into $K << d$ monotone functions, there is still the question of how to efficiently find the right \textit{directions of monotonicity} $\vs_1,\dots,\vs_K$. Roughly speaking, a monotone function with direction of monotonicity $\vs_i$ has shocks only along hypersurfaces whose normal vector is restricted to the space
	\begin{equation*}
		\left\{\vn = \beta \sum_{k=1}^d \alpha_k \big( \ve{e}_k \cdot \vs_i \big) \ve{e}_k \setsep \beta \in \{-1,1\},\, \alpha_k>0, \, k=1,\dots, d \right\}.
	\end{equation*}
	Thus, one class of functions for which architectures of type (A) are suitable are functions which are non-smooth only in certain variables in the following sense:
	$
		(x_1,\dots, x_l) \mapsto f(x_1,\dots,x_l,x_{l+1},\dots, x_d)
	$
	is smooth for all fixed $x_{l+1},\dots, x_d$.
	Even in the more general high dimensional setting,
	there are several ways to tackle this issue -- the most straightforward being to investigate learnable bases of monotonicity.
	We leave this issue for future work and focus on illustrating the advantages of this architecture for solving partial differential equations with non-smooth solutions in low dimensions. Currently, these are typically still challenging problems for neural networks.
	
	Formally, the number of summands in (\ref{naiveEVNN}) can be halfed by learning the sum of the functions nondecreasing in $\vs_i \odot \vx$ and $\vs_j \odot \vx$  respectively, with $i,j$ such that $\vs_i = -\vs_j$. This corresponds to allowing the elements of $\matA_{L+1,i}$ to take negative values, i.e., having the last layer of each subnetwork be a standard linear layer. This leads to the following definition.
	\begin{definition}
		We say a network has an architecture of type (A') 
		if it can be decomposed as
		\begin{equation*}
			\eta(\vx) = \sum_{i=1}^{K}\matA_{L+1,i}\cdot \hat\sigma \bigg(\matA_{L,i}\cdot \hat\sigma\Big(...\hat\sigma(\matA_{1,i}\cdot (\vs_i\odot \vx)+\vb_{1,i})\Big)+\vb_{L,i}\bigg) +{\vb_{L+1,i}}\,,
		\end{equation*}
		with a nondecreasing activation function $\sigma$, $\vs_1,\dots,\vs_K\in\{-1,1\}^d$, $\vb_{l,i}\in\reals^{m_l}$, $l=1,\dots,L+1$ , $i=1,\dots,K$, and matrices $\matA_{l,i}\in \reals_{\geq 0}^{m_l\times m_{l-1}}$, $l=0,...,L$, $i=1,...,K$ with nonnegative elements and $\matA_{L+1,i}\in \reals^{N\times m_L}$.
	\end{definition}
	Figure~\ref{fig:Implementation} illustrates this architecture.
	{%
	The definition does not specify restrictions on $K$. We are mostly interested in architecture (A') to reduce the computational cost. Hence, if we work with networks with architecture of type (A'), we will assume $K \leq 2^{d-1}$.
	}
	
	Weight matrices $\matA_{1,i}, \dots, \matA_{L,i}$ with positive elements are necessary for the proposed architectures. We ensure the positivity of the weights by applying a function $\modi\colon \reals \to \reals_{\geq 0}$ elementwise to the weight matrices. The subnetworks of EVNNs with an architecture of type (A') are implemented by setting
	$\matA_{l,i}=\modi(\matB_{l,i})$ 
	{%
	for some matrices $\matB_{l,i}$ for $l=1,\dots,L$. }
	Likewise, $\matA_{l,i}=\modi(\matB_{l,i})$ for $l=1,\dots,L+1$ for networks with an architecture of type (A).
	We refer to functions $\modi$ as weight modifiers.
	
	We define the estimated M-variation for networks with an architecture of type (A) and (A') on $U\subset\reals^d$ by
	\begin{equation*}
		\EMVar(\eta; U) \coloneqq \sum_{i=1}^K \left( \sup_{\vx \in U} \eta_i(\vx) - \inf_{\vx \in U} \eta_i(\vx) \right)
	\end{equation*}
	and
	\begin{align*}
		\EMVar(\eta; U) \coloneqq \sum_{i=1}^K |\matA_{L+1,i}| \cdot \Bigg( &\sup_{\vx \in U}
		 \hat\sigma \bigg(\matA_{L,i}\cdot \hat\sigma\Big(...\hat\sigma(\matA_{0,i}\cdot \vs_i\odot \vx+\ve{b}_{0,i})\dots\Big)+\vb_{L,i}\bigg)\\
		&- \inf_{\vx \in U} 
		\hat\sigma \bigg(A_{L,i}\cdot \hat\sigma\Big(...\hat\sigma(A_{0,i}\cdot \vs_i\odot \vx+\vb_{0,i})\dots\Big)+\vb_{L,i}\bigg)
		\Bigg),
	\end{align*}
	respectively. Here, the elements of $|\matA_{L+1,i}| \in \reals^{m_{L+1}\times m_L}$ are given by
	$
		\big(|\matA_{L+1,i}|\big)_{l,k} = \Big|\big(\matA_{L+1,i}\big)_{l,k} \Big|.
	$
	Clearly, the components of the estimated M-variation $\EMVar(\eta; U)$ provide an upper bound on the M-variation of the corresponding components of the network $\eta$. In what follows, we will write $\EMVar(\eta )$ instead of $\EMVar(\eta; U)$ if there is no ambiguity. To improve the estimate on the M-variation of $\eta$ provided by $\EMVar$, we can add a regularization term proportional to $\EMVar(\eta)$ to any loss function. As presented in Section~\ref{sec:VarEst}, this regularization also penalizes networks with a large variation. In fact, $\EMVar$ can be used to get an upper bound on the variation of the network.

\begin{figure}
	\centering
	\begin{subfigure}{0.49\textwidth}
		\centering
		\includegraphics{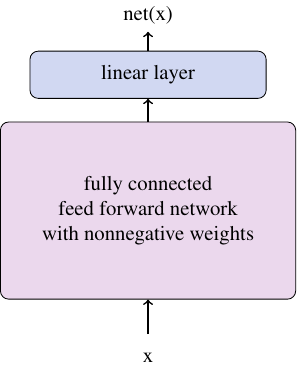}
		\caption{$d=1$.}
		\label{fig:three sin x}
	\end{subfigure}
	\hfill
		\begin{subfigure}{0.49\textwidth}
		\centering
		\includegraphics{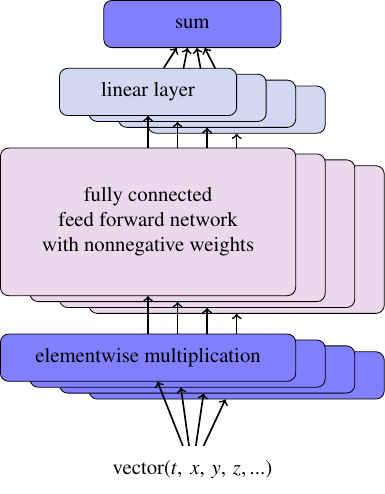}
		\caption{$d>1$.}
		\label{fig:y equals x}
	\end{subfigure}

	\caption{Possible implementations for the proposed architecture for EVNNs from $\reals^d$ to $\reals$. It is straightforward to generalize this architecture to mappings from $\reals^d$ to $\reals^N$.}
	\label{fig:Implementation}
\end{figure}

	\subsection{Choice of the activation function}\label{sec:choice_acti}
	Choosing a nondecreasing activation function is crucial for the architectures of type (A) and (A'). Note that most of the currently popular activation functions are nondecreasing. In order to approximate the variation of the target function as sharply as possible, we need to impose additional restrictions on the choice of activation function. In particular, convex activation functions (like the ReLU activation function) should be avoided. 
	{%
	Otherwise networks with architectures of type (A) are convex themselves. And for networks $\eta$ with architectures of type (A') and the proposed subnetworks with $K\leq 2^{d-1}$ the mapping
	\begin{equation*}
		h \mapsto \EMVar\Big(\eta; \, \{a_1\} \times \dots \times \{a_{j-1}\} \times [a_j+h, b_j+h] \times \{a_{j+1}\}\times \dots \times \{a_d\}\Big)
	\end{equation*}
	is convex for all $a_i\in\reals$, $b_j>a_j$ if there is no subnetwork with direction of monotonicity $\vs$ with $\vs\cdot \ve{e}_j <0$. 
	}
	Consequently, the network will systematically overestimate the M-variation of the approximated function. 
 
	To enable a sharp estimate on the M-variation, we will require the activation function $\sigma:\reals \to \reals$ to satisfy
	\begin{equation}\label{eq:actiProperty}
	\lim_{x\to-\infty}\sigma(x) = \sigma_{-}, \qquad \lim_{x\to \infty}\sigma(x) = \sigma_{+},
	\end{equation}
	for some constants $\sigma_{-}<\sigma_{+}$. For sufficiently deep networks this condition is sufficient for approximations with sharp estimate on the M-variation, see Theorem~\ref{thm:UniBMVapprox} and Corollary~\ref{cor:BMVApproxinL1}. For the numerical test cases presented in Section~\ref{sec:num_experiments} we restrict ourselves to the hyperbolic tangent function which satisfies (\ref{eq:actiProperty}).
	
	\begin{remark}
	For networks with an architecture of type (A) or (A'), the inputs of all activation functions are monotone by design. Assuming (\ref{eq:actiProperty}) to hold, every neuron is only able to \textit{fire} once, i.e., once the input has surpassed a certain threshold at a point $\vx\in\reals^d$, the output of this neuron will stay mostly the same for all $\vy\geq \vx$. This limits to which extend neurons can be \textit{reused} in different parts of the domain. Consequently we might expect a reduced expressiveness compared to standard feed-forward neural networks with the same number of parameters.
	
	There are several obvious approaches to fix the aforementioned problem. Maybe the simplest is to choose an unbounded monotone activation function with 
	steplike behavior, e.g.,
	\begin{equation*}
		\sigma(x) = \int_{0}^{x} |\cos(y)| \ddd y.
	\end{equation*} 
	We experimented with this activation function but for our numerical experiments, using the $\tanh$ activation function lead to sufficient expressivity of our networks. 
	\end{remark}
	
	\begin{remark}\label{rem:AvsAprimeForTanh}
	The choice of the activation function also influences the similarity of architectures of type (A) and (A'). If $\sigma$ satisfies
	\begin{equation}\label{eq:generalized_uneven_property}
		a\sigma(-bx-c)=-\sigma(x) \qquad \text{for all }x\in\reals
	\end{equation}
	for some suitable constants $a,b>0$, $c\in\reals$, it is easy to verify that every network with an architecture of type (A') is also a network with an architecture of type (A). 
	Obviously, $tanh$ satisfies (\ref{eq:generalized_uneven_property}).
	\end{remark}
	\subsection{Approximation properties}\label{approxProps}
	
	In this section, we briefly sketch the approximation properties of networks with the proposed architectures. A more detailed discussion is given in Appendix~\ref{App:approxProps}. There, all results only mentioned in this section are proven and a more detailed proof for Theorem~\ref{thm:UniBMVapprox} is given.
	
	First, we note that sufficiently large networks with architectures of either type (A) or of type (A') can essentially approximate the same functions with respect to the uniform norm on compact sets, even for activation functions which do not satisfy (\ref{eq:generalized_uneven_property}). Using the continuity of $\sigma$ it is easy to see that property (\ref{eq:actiProperty}) is sufficient.
	
	The standard uniform approximation results in $C^k$ norms and $L^p$ norms are also valid for shallow networks with architecture of type (A) or (A'). This is straightforward as shallow neural networks can easily be approximated by shallow networks with architecture of type (A) or (A'). 
	
	While it is important to know that the networks with architectures of type (A) are able to approximate functions up to arbitrary accuracy in various common metrics, for our purposes a different question is more natural: 
	\begin{question}\label{q:1}
		Can functions with bounded M-variation be approximated up to arbitrary accuracy with networks whose estimated M-variation remains bounded? 
	\end{question}
	This question is essential for the algorithms presented in Section~\ref{sec:applications}. For example, the M-variation of an EVNN determines how many sampling points are necessary to guarantee a given accuracy of the numerical integration scheme presented in Section~\ref{sec:IntegratingBMVFunctions}. 
	If the M-variation increases as we approach a target function, the numerical integration would get increasingly computationally expensive. 
	Clearly, without a bound on the computational cost an algorithm need not terminate within reasonable or even finite time.
	
	For simplicity, we restrict ourselves to functions defined on hyperrectangles $H = \prod_{i=1}^d [a_i, b_i]$ and networks with architecture of type (A).
	
	\begin{theorem}[Universal approximation of $BMV$ functions]\label{thm:UniBMVapprox}
		For any continuous $BMV$ function $f\colon H \to \reals$, any $\epsilon>0$, any nondecreasing activation function $\sigma$ satisfying (\ref{eq:actiProperty}) and any $L\geq 3$, there exists a EVNN with an architecture of type (A) with $L$ hidden layers and activation function $\sigma$ satisfying $\|f-\eta\|_\infty \leq \epsilon$ and $\EMVar(\eta;H) < \MVar(f;H)$.
	\end{theorem}

	\begin{proof}
		It is no restriction to assume $\sigma_- =0$ and $\sigma_+=1$.
		It is straightforward to approximate the identity function with a single hidden layer. Thus, we may restrict ourselves to the case $L=3$. By definition, for every $\epsilon'>0$ there exist bounded monotone functions $f_1, \dots, f_{2^d} \colon \reals^d \to \reals$ with 
		\begin{equation*}
			f(\vx) = \sum_{i=1}^{2^d} f_i(\vx), \qquad\MVar(f) + \epsilon' > \sum_{i=1}^{2^d} \left(\max_H f_i - \min_H f_i \right)
			=\sum_{i=1}^{2^d} \MVar(f_i).
		\end{equation*}
		We extend $f$ to $C(\reals^d)\cap BMV(\reals^d)$ without changing the M-variation by setting
		\begin{equation*}
			f(x_1,\dots, x_d) = f(\max\{\min\{x_1, a_1\}, b_1\}, \dots,\max\{\min\{x_d, a_d\}, b_d\}). 
		\end{equation*}
		By a standard mollification argument, we may assume all $f_i$ to be continuous as well. Clearly, it suffices to show the proposition for all $f_i$. Therefore, we may assume from now on, that $f$ is nondecreasing. Moreover, it is no restriction to assume $0 = f(a_1,\dots,a_d) \leq f \leq f(b_1,\dots,b_d)=1$. Clearly, we may approximate $f$ uniformly by the sum
		$
		\sum_{1 \leq k<n} \frac1n\ind{\{f\geq k/n\}}
		$. By continuity of $f$, there exists $\delta>0$ such that
		\begin{equation*}
			\dist\Big(\partial \{f\geq (2k+1)/(2n)\},\,\{f<k/n\}\cup \{f\geq (k+1)/n\} \Big) > \delta \qquad \text{for all }k = 1,\dots, n-1.
		\end{equation*}
		Therefore, it is sufficient to approximate $\ind{\{f\geq k/n\}}$ by nondecreasing functions $g_k$ with
		\begin{equation}\label{eq:newAuxEq}
			0\leq g_k \leq 1, 
			\qquad g_k<\epsilon'' \text{ on } \{f< k/n\}
			\qquad 
			g_k>1-\epsilon'' \text{ on } \{f\geq (k+1)/n\},
		\end{equation} 
	with $\epsilon''$ sufficiently small and $k=1,\dots, n-1$. We fix such a $k$. By compactness, there are a finite number of points  $\vx^{(l)}= (x^{(l)}_1, \dots,x^{(l)}_d)$, $l=1,\dots, K$, with $\dist(\vx^{(l)}, \partial \{f\geq (2k+1)/(2n)\}) < \delta/2$ and
		\begin{equation*}
			\{f\geq (k+1)/n\} \subset \bigcup_{l=1}^K \{\vx \geq \vx^{(l)}\} \eqqcolon A \subset \{f\geq k/n\}.
		\end{equation*}
		Instead of approximating $\ind{\{f\geq k/n\}}$ directly, we will approximate $\ind{A}$. Take $S=S(\epsilon'')$ large enough such that $\sigma(-S)< \epsilon''$ and $\sigma(S) > 1- \epsilon''$. Then we may approximate
		\begin{equation*}
			\ind{A}(\vx) = \ind{\cup_{l=1}^K \{\vy \geq \vx^{(l)}\}}(\vx) \approx \sigma\left(-S+2S\sum_{l=1}^K \ind{\{\vy \geq \vx^{(l)}\}}(\vx)\right).
		\end{equation*}
		Here, we use $\approx$ to indicate an approximation which allows the right hand side to satisfy (\ref{eq:newAuxEq}).
		Obviously, $\{\vy \geq \vx^{(l)}\} = \cap_{i=1}^d \{\vy=(y_1, \dots, y_d) \setsep y_i \geq x^{(l)}_i\}$. Therefore, possibly choosing a larger $S$,
		\begin{equation*}
		 	\ind{A}(\vx) \approx \sigma\left(-S+2S\sum_{l=1}^K\sigma\left(-(2d-1)S+2S\sum_{i=1}^d \ind{\{\vy=(y_1, \dots, y_d) \setsep y_i \geq x^{(l)}_i\}}(\vx)\right) \right).
		\end{equation*}
	Again after possibly taking a larger $S$, we may approximate $\ind{\{\vy=(y_1, \dots, y_d) \setsep y_i \geq x^{(l)}_i\}}(\vx)$ by $\sigma(2\sqrt{d}S(\vx-\vx^{(l)})\cdot \ve{e}_i/\delta)$ and $\ind{\{f\geq k/n\}}(\vx)$ by $\ind{A}(\vx)$ to find
	\begin{equation*}
		\ind{\{f\geq k/n\}}(\vx) \approx 
		\sigma\left(-S+2S\sum_{l=1}^K\sigma\left(-(d-1)S+2S\sum_{i=1}^d 
		\sigma(2\sqrt{d}S(\vx-\vx^{(l)})\cdot \ve{e}_i/\delta)
		\right) \right).
	\end{equation*}
This immediately leads to a network with the properties listed in the theorem.
	\end{proof}

	If we drop the assumption of $f$ being continuous, we have to restrict ourselves to convergence in $L^p(H)$, $1 \leq p < \infty$.

	\begin{corollary}\label{cor:BMVApproxinL1}
		Let $f \in BMV(H)$, $\epsilon>0$, $1\leq p< \infty$. For any nondecreasing activation function $\sigma$ satisfying (\ref{eq:actiProperty}) and any $L\geq 3$ there exists a EVNN with an architecture of type (A) with $L$ hidden layers and activation function $\sigma$ satisfying $\|f-\eta\|_{L^1(H)} \leq \epsilon$ and $\EMVar(\eta) < \MVar(f)$.
	\end{corollary}

	\begin{proof} We define the mollified extension 
		$
			E_h[f] =  \kappa_{h} * E[f],
		$
		with $E$ defined in (\ref{eq:definitionE+-}), $\kappa$ denoting the standard mollifier and $\kappa_h(\vx) = h^{-d} \kappa(h \vx)$.
		Choosing $h$ small enough, we have $\|E_h(f)-f\|_{L^p(H)} \leq \frac{\epsilon}{2}$ and $\MVar(E_h[f], H) \leq \MVar(f)$. Using Theorem~\ref{thm:UniBMVapprox} we derive the conclusion of the corollary.
	\end{proof}

	The restriction of needing at least 3 hidden layers might be surprising at first, given that the classical uniform approximation results are true for shallow neural networks. We provide an example for a function in $BMV$ which cannot be effectively approximated by shallow EVNNs in Appendix~\ref{App:approxProps}. In Appendix~\ref{App:sec:shockLimitation} we discuss a limitation of $BMV$ itself for finding local discontinuities.
	
	\subsection{Initialization}\label{sec:Initialization}
	In this section we discuss the weight and bias initialization procedure for networks with architectures of type (A) and (A'). We use a modified version of the Layer-sequential unit-variance (LSUV) initialization proposed in \cite{LSUVInit2015}. 
	The bias $\vb_l$ of layer $l$ was initialized with random numbers uniformly distributed in the interval $(-\sqrt{m_l}, \sqrt{m_l})$. Here $m_l$ denotes the width of layer $l$.
	
	Since the only orthonormal matrices with nonnegative weights are column permutations of the identity matrix, we did not initialize the weight matrices $\matA_l$ as orthonormal matrices. Instead, we just filled the weight matrices with the absolute values of Gaussian random numbers with mean 0 and variance 1. Then we propagated a batch of inputs from $(-1,1)^d$ layerwise through each subnetwork, rescaling the weights such that the output of each layer for the batch has variance $1$. The factor $1$ was found to be working well via a trial and error procedure. 
	Algorithm~\ref{initAlg} sums up the initialization procedure.
	
	\begin{algorithm}
		\caption{Initialization of a subnet.}
		\label{initAlg}
		\begin{algorithmic}
			\For{$l = 1,2,...,L+1$}
			\State Fill weight matrices $\matA_l \in \reals^{m_l\times m_{l-1}}$ with Gaussian random numbers with mean 0 and variance 1.
			\State Fill bias vectors $\vb_l \in \reals^{m_l}$ with uniform random numbers from the interval $(-\sqrt{m_l}, \sqrt{m_l})$.
			\EndFor
			\State Set $x=(\vx_i)_{i\in I}$ to be a batch of data from $(-1,1)^d$.\\
			(We chose the vertices of a regular tensor mesh.) 
			\For{l =$ 1,2,...,L+1$}
			\State $\vy_i \gets \matA_l \cdot \vx_i$ for $i\in I$
			\State $\text{var} \gets \frac{1}{|I|-1}\sum_{i\in I} \left( \vy_i^2 - \left(\frac{1}{|I|}\sum_{i\in I} \vy_i\right)^2 \right)$
			\State $\matA_l \gets \frac{\mathrm{LSUV Gain}}{
			 \sqrt{\text{var}}}\matA_l$ \qquad (The LSUV gain was set to 1.)
			\State $\vx_i \gets \sigma(\matA_l \cdot \vx_i + \vb_l)$
			\EndFor
		\end{algorithmic}
	\end{algorithm}
	
	\begin{remark}\label{rem:xavier}
	We have also experimented with a modification of the commonly used Xavier uniform initialization \cite[Section 4.2.1]{pmlr-v9-glorot10a} (also known as Glorot initialization, referred to as \textit{normalized initialization} in the original work). We have used the absolute values of the elements of the weight matrix after Xavier initialization and initialized the biases as described above. The results were generally of a similar quality. This is somewhat surprising given that the motivation of the Xavier initialization was based on each weight having mean zero and positive variance. For networks with architectures of type (A) or (A') both properties cannot be enforced simultaneously for the weight matrices.
	\end{remark}

\section{Application to scalar hyperbolic conservation laws}\label{sec:applications}
	In this section, we discuss the application of EVNNs to scalar hyperbolic conservation laws in one space dimension. Thus, we consider the following equation
	\begin{equation}\label{eq:conservationLaw}
		\partial_t u + \partial_x F(u) = 0
	\end{equation}
	Utilizing the additional regularity information provided by the estimated M-variation, we construct sequences of explicitly computable loss functions $(\Loss_n)_{n\in\naturals}$ such that any sequence of networks $(\eta_n)_{n\in\naturals}$ with architectures of type (A) or (A') satisfying 
	\begin{equation}\label{eq:convCon}
		\Loss_n(\eta_n) \xrightarrow{n\to\infty} 0
	\end{equation}
	converges to the appropriate solution.
	Further, we show that networks satisfying (\ref{eq:convCon}) exist if the solution is an element of $BMV$. Even stronger results are shown for ODEs in Appendix~\ref{App:sec:ApplicationToODEs}.
	
	Throughout this section, the flux function $F\in C^2(\reals)$ is assumed to be strictly convex and the initial data are given by $u_0=u(0,\cdot) \in L^\infty([0,L])$.
	We recall some classical results on scalar hyperbolic conservation laws in one space dimension. We restrict ourselves to periodic boundary conditions.
	It is well known that classical solutions of (\ref{eq:conservationLaw}) generally do not exist globally. On the other hand, weak solutions are generally not unique. Therefore, the accepted notion of solution is given by weak solutions satisfying an additional physically motivated admissibility criterion singling out a unique solution. We start by recalling the definition and classical existence and uniqueness result for admissible weak solutions.
	To this end, we define the set $PT$ of periodic test functions by
	\begin{equation*}
	PT \coloneqq \left\{\varphi\in C_c^\infty([0,T)\times[0,L]) \setsep D^\alpha\varphi(t,0)=D^\alpha\varphi(t,L) \text{ for all }\alpha\in \naturals_0^2, \,  t\in [0,T)  \right\}.
	\end{equation*}
	\begin{definition}[Admissible weak solution]
				\label{def:admissibleWeakSolution}
			Given initial data $u_0\in L^1([0,L])$ and $T>0$, we say $u\in L^1([0,T]\times[0,L])$ is a weak admissible solution of (\ref{eq:conservationLaw}) if $u$ satisfies
			\begin{equation}\label{perWeakEquation}
				\int_0^T \int_{[0,L]} u(t,x) \partial_t \varphi(t,x)+F(u)(t,x) \partial_x \varphi(t,x) \ddd x \ddd t
				= 
				-\int_{[0,L]} u_0(x) \varphi(0,x) \ddd x
			\end{equation}
			for all test functions $\varphi\in PT$ and
			if there exist
			locally Lipschitz continuous functions $S,Q\colon \reals \to \reals$, with $S$ strictly convex and $Q$ satisfying $Q' = F' S'$ almost everywhere,
			such that
			\begin{equation}\label{perWeakEntropy}
				\int_0^T \int_{[0,L]} S(u)(t,x) \partial_t \varphi(t,x)+Q(u)(t,x) \partial_x \varphi(t,x) \ddd x \ddd t
				\geq 
				-\int_{[0,L]} S(u_0)(x) \varphi(0,x) \ddd x
			\end{equation}
			for all nonnegative test functions $\varphi\in PT$, $\varphi\geq0$.
		\end{definition}
	\begin{theorem}[Existence and uniqueness]
		For each $u_0 \in L^\infty([0,L])$, there exists a unique admissible weak solution $u\in C([0,T];L^1([0,L]))\cap L^\infty([0,T]\times [0,L])$ of (\ref{eq:conservationLaw}).
	\end{theorem}

		\begin{proof}
		On the whole space,
		existence (and uniqueness) for admissible weak solutions which satisfy the entropy condition for all convex $\tilde S$ and their respective entropy flux $\tilde Q$ satisfying $\tilde{Q}' = F' \tilde{S}'$ almost everywhere can be found, e.g., in \cite[Theorem 6.2.2]{dafermos2005hyperbolic}. In particular there exists a weak solution satisfying the entropy inequality for the entropy-entropy flux pair $(S,Q)$. For the proof of uniqueness we refer the reader to \cite{panov1994uniqueness,de2004minimal,krupa2019uniqueness}. Modifying these results to the setting with periodic boundary conditions is straightforward.
	\end{proof}

	\subsection{Convergence of minimizers}

	To solve (\ref{eq:conservationLaw}) for given initial data $u_0\in L^\infty([0,L])$, we use EVNNs with architectures of type (A) or (A') and a loss function based on the finite volume method. The latter is a popular numerical method to approximate solutions of hyperbolic conservation laws.
	More precisely, we use the following sequence of loss functions for a given neural network $\eta$:
	
	\begin{align}\label{eq:convLawLoss}
	\loss_n[\eta] &= 
	\loss^\text{pde}\Big(\eta;\, (t^{(n)}_k)_{k=0}^{K_n},\, (x^{(n)}_m)_{m=0}^{M_n}\Big)
	+
	\loss^\text{ent}\Big(\eta;\, (t^{(n)}_k)_{k=0}^{K_n},\, (x^{(n)}_m)_{m=0}^{M_n}\Big)
	\notag
	\\
	&
	+
	\loss^\text{IC}\Big(\eta;\, (x^{(n)}_m)_{m=0}^{M_n}\Big)
	+
	\loss^\text{BC}\Big(\eta;\, (t^{(n)}_k)_{k=0}^{K_n}\Big)
	+
	\Loss^\text{reg}(\eta),
	\end{align}
	with
	\begin{align*}
	&\loss^\text{pde}\Big(\eta;\, (t^{(n)}_k)_{k=0}^{K_n},\, (x^{(n)}_m)_{m=0}^{M_n}\Big)
	=
	\sum_{k=1}^{K_n} \sum_{m=1}^{M_n} 
	\left|
	\int_{(0,T)\times(0,L)}  \ind{(t^{(n)}_{k-1},t^{(n)}_k)\times(x^{(n)}_{m-1},x^{(n)}_m)}
	\ddd 
	\Big(
	\partial_t \eta + \partial_x F(\eta) 
	\Big)
	(t,x) 
	\right|,
	\notag
	\\
	&\loss^\text{ent}\Big(\eta;\, (t^{(n)}_k)_{k=0}^{K_n},\, (x^{(n)}_m)_{m=0}^{M_n}\Big)
	=\\
	&\qquad \sum_{k=1}^{K_n} \sum_{m=1}^{M_n}
	\max\left\{
	\int_{(0,T)\times(0,L)}  \ind{(t^{(n)}_{k-1},t_k^{(n)})\times(x^{(n)}_{m-1},x^{(n)}_m)}
	\ddd 
	\Big(
	\partial_t S(\eta) + \partial_x q(\eta) 
	\Big)
	(t,x) - \epsilon_n,\, 0
	\right\},
	\notag
	\\
	&\loss^\text{IC}\Big(\eta\Big)
	=
	\int_{[0,L]} \left| u_0 - \eta(0,x) \right| \ddd x,  \qquad 
	\loss^\text{BC}\Big(\eta\Big)
	=
\int_{[0,T]} \left| \eta(t,0)- \eta(t,L) \right| \ddd t ,
	\notag
	\\
	&\Loss^\text{reg}\Big(\eta\Big)
	=
	\max\Big\{ \EMVar(\eta; \;[0,T]\times[0,L])|-R,\,0  \Big\}.
	\end{align*}
	Here, $\epsilon_n \to 0$ for $n\to\infty$.
	The loss function in (\ref{eq:convLawLoss}) consists of penalty terms corresponding to the PDE, the entropy inequality, the initial data, the boundary condition and regularization, respectively.
	For each $n\in\naturals$, $(t^{(n)}_k)_{k=0}^{K_n}$ and $(x^{(n)}_m)_{m=0}^{M_n}$ denote the coordinates of the corners of the test volumes $[t^{(n)}_{k-1}, t^{(n)}_k]\times[x^{(n)}_{m-1}, x^{(n)}_{m}]$. They are assumed to satisfy
	\begin{itemize}
		\item
		$t^{(n)}_0 = 0$, $x^{(n)}_0 = 0$, $t^{(n)}_{K_n} = T$, $x^{(n)}_{M_n} = L$,
		\item 
		$t^{(n)}_k>t^{(n)}_{k-1}$ and $x^{(n)}_m>x^{(n)}_{m-1}$ for all $m=1,\dots,M_n$, $k=1,\dots,K_n$,
		\item 
		$\max\{t^{(n)}_k-t^{(n)}_{k-1}\setsep k=1,\dots, K_n\} \xrightarrow{n\to\infty}0$
		and
		$\max\{x^{(n)}_m-x^{(n)}_{m-1}\setsep m=1,\dots, M_n\} \xrightarrow{n\to\infty}0$.
	\end{itemize}
	From now on we will omit the superscript $n$ of $t_k$ and $x_m$ for notational convenience.
	The threshold $R$ is chosen independent of $n$ and determines how large the estimated M-variation of $\eta$ can be before the network is being penalized for it.
 	$R$ should be an upper bound on the M-variation of the solution $u$. Obviously, getting such a bound in practice is difficult and probably not every solution of a scalar hyperbolic conservation law has bounded M-variation.
 	
 	Note that the loss in (\ref{eq:convLawLoss}) is not fully discretized. The integrals in $\loss^\text{pde}$, $\loss^\text{ent}$, $\loss^\text{IC}$ and $\loss^\text{BC}$ generally cannot be computed exactly. We can use the integration scheme from Section~\ref{sec:IntegratingBMVFunctions} to fully discretize the loss function. We denote by $I_\mathcal{E}[F(f), U]$ the corresponding numerical approximation of the integral of the composition $F\circ f$ over $U$. Here, $f$ is a $BMV$ function, $F$ is locally Lipschitz and $\mathcal{E}$ denotes the error tolerance.
 	The resulting fully discretized loss function is given by
 	\begin{align*}
 		\Loss_n[\eta] &= 
 		\Loss^\text{pde}\Big(\eta;\, (t_k)_{k=1}^{K_n},\, (x_m)_{m=1}^{M_n}\Big)
 			+
 		\Loss^\text{ent}\Big(\eta;\, (t_k)_{k=1}^{K_n},\, (x_m)_{m=1}^{M_n}\Big)
 			\notag
 			\\
 			&
 			+
 		\Loss^\text{IC}\Big(\eta\Big)
 			+
 		\Loss^\text{BC}\Big(\eta\Big)
 			+
 		\Loss^\text{reg}\Big(\eta\Big),
 	\end{align*}
 	with
 	\begin{align*}
 		&\Loss^\text{pde}\Big(\eta;\, (t_k)_{k=1}^{K_n},\, (x_m)_{m=1}^{M_n}\Big)
 		\\
 			&\qquad=
 			\sum_{k=1}^{K_n} \sum_{m=1}^{M_n} 
 			\Bigg|
 			I_{\mathcal{E}_\text{int}(n)}\Big(\eta(t_{k},\cdot), \,[x_{m-1},{x_m}]\Big)
 			-
 			I_{\mathcal{E}_\text{int}(n)}\Big(\eta(t_{k-1},\cdot), \,[x_{m-1},{x_m}]\Big)
 			\\
 			&\qquad\phantom{=}
 			+
 			I_{\mathcal{E}_\text{int}(n)}
 			\Big(
 			F(\eta(\cdot,x_m)),\,
 			[t_{k-1},t_k]
 			\Big)
 			-
 			I_{\mathcal{E}_\text{int}(n)}
 			\Big(
 			F(\eta(\cdot,x_{m-1})),\,
 			[t_{k-1},t_k]
 			\Big)
 			\Bigg|\,,
 			\notag
 			\\
 		&\Loss^\text{ent}\Big(\eta;\, (t_k)_{k=1}^{K_n},\, (x_m)_{m=1}^{M_n}\Big)
 		\\
 		&\qquad=\max\Bigg\{
 		\sum_{k=1}^{K_n} \sum_{m=1}^{M_n} 
 		I_{\mathcal{E}_\text{int}(n)}\Big(S\big(\eta(t_k,\cdot)\big), \,[x_{m-1},{x_m}]\Big)
 		-
 		I_{\mathcal{E}_\text{int}(n)}\Big(S\big(\eta(t_{k-1},\cdot)\big), \,[x_{m-1},{x_m}]\Big)
 		\\
 		&\qquad\phantom{=}
 		+
 		I_{\mathcal{E}_\text{int}(n)}
 		\Big(
 		Q(\eta(\cdot,x_m)),\,
 		[t_{k-1},t_k]
 		\Big)
 		-
 		I_{\mathcal{E}_\text{int}(n)}
 		\Big(
 		Q(\eta(\cdot,x_{m-1})),\,
 		[t_{k-1},t_k]
 		\Big)
 		-\epsilon_n, 0\Bigg\}\,,
 		\notag
 			\\
 		&\Loss^\text{IC}\Big(\eta\Big)
 			=
 			I_{\mathcal{E}_\text{int}(n)}
 			\Big(
 			\big|\eta(0,\cdot) -  u_0\big|
 			,\,
 			[0,L]
 			\Big)
 			\,.
 			\notag
 			\\
 			&\Loss^\text{BC}\Big(\eta\Big)
 			=
 			I_{\mathcal{E}_\text{int}(n)}
 			\Big(
 			\big|F(\eta(\cdot,0)) -  F(\eta(\cdot,L))\big|
 			,\,
 			[0,T]
 			\Big)
 			\,.
 			\notag
 		\end{align*}
	Moreover, the error tolerance $\mathcal{E}_\text{int}(n)$ is assumed to satisfy 
	$\mathcal{E}_\text{int}(n) K_n M_n \to 0$ for $n\to\infty$.
	This ensures that $\Loss_n$ is a good approximation of $\loss_n$ for sufficiently large $n$.
	\begin{lemma}\label{lem:LossDiscretization}
		For any network $\eta$ with an architecture of type (A) or (A') and continuous activation function,
		\begin{equation*}
			\big|\loss_n(\eta)-\Loss_N(\eta)\big| \leq \mathcal{E}_\text{int}(n)\big(2 K_n M_n + 2\big).
		\end{equation*}
	\end{lemma}
	\begin{proof}
		Let $G\colon \reals \to \reals$ be locally Lipschitz continuous
		and denote by $T\colon BV([t_1,t_2]\times[x_1,x_2]) \to L^1(\partial[t_1,t_2]\times[x_1,x_2] )$ the trace on the cube $[t_1,t_2]\times[x_1,x_2]\subset [0,T]\times [0,L]$, cf.\ \cite[Theorem 3.87]{ambrosio2000functions}.
		The continuity of $\sigma$ implies continuity of $\eta$. Consequentially, we obtain
		\begin{align*}
			\int_{[x_1,x_2]} G(\eta(t_i,x)) \ddd x &= 	\int_{\{t_i\} \times [x_1,x_2]} T[G(\eta(t,x))] \ddd S(t,x), \qquad \text{for }i=1,2,\\
			\int_{[t_1,t_2]} G(\eta(t,x_i)) \ddd t &= 	\int_{[t_1,t_2]\times \{x_i\} } T[G(\eta(t,x))] \ddd S(t,x)\qquad \text{for }i=1,2.
		\end{align*}
		The statement of the lemma now follows from the divergence theorem \cite[Theorem 3.87]{ambrosio2000functions}.
	\end{proof}
	
	Assuming a suitable $R>0$ is identified, we can prove convergence in $L^1$ in the following sense.
	
	\begin{theorem}[Convergence]\label{thm:convergence}
		Let $u$ be an admissible weak solution of (\ref{eq:conservationLaw}) in the sense of Definition~\ref{def:admissibleWeakSolution} with initial data $u_0 \in L^\infty$.
		Let the sequence of networks $(\eta_n)_{n\in\naturals}$ with architectures of type (A) or (A') and continuous activation function $\sigma$ satisfy
		\begin{equation}\label{eq:convergenceCondition}
		\Loss_n[\eta_n] \xrightarrow{n\to\infty}0.
		\end{equation}
		Then $ \eta_n \to u$ in $L^1([0,T]\times [0,L])$.
	\end{theorem}

	\begin{proof}
		Due to the construction of the loss functionals and Lemma~\ref{lem:LossDiscretization}, (\ref{eq:convLawLoss}) together with (\ref{eq:convergenceCondition}) implies that 
		\begin{equation}\label{eq:aux_eq_in_convergence}
			\sup_{n\in\naturals} \|\eta_n\|_{BV((0,T)\times(0,L))} < \infty.
		\end{equation}
		Indeed, the uniform bound on $\Var(\eta_n, [0,T]\times[0,L])$ is a consequence of the boundedness of $\EMVar(\eta_n)$. As $|\eta_n(t,x)-\eta_n(s,y)|\leq \EMVar(\eta_n)$, the boundedness in $L^1$ follows from
		\begin{align*}
			L|\eta_n(t,x)| &\leq \Big|\int_{[0,L]} \eta_n(t,x)-\eta_n(0,y) \ddd y \Big| + 
			\Big|\int_{[0,L]} \eta_n(0,y) \ddd y \Big|
			\\
			&\leq L \EMVar(\eta_n)
			+ \Big|\int_{[0,L]} u_0(y)-\eta_n(0,y) \ddd y \Big|
			+ \Big|\int_{[0,L]} u_0(y)\ddd y  \Big|.
		\end{align*}
		
		We obtain the relative compactness of $(\eta_n)_{n\in\naturals}$ in $L^1((0,T)\times(0,L))$ due to (\ref{eq:aux_eq_in_convergence}), see \cite[Corollary 3.49]{ambrosio2000functions}. 
		It remains to show that the limit of any converging subsequence of $(\eta_n)_{n\in\naturals}$ is the unique admissible weak solution of the scalar hyperbolic conservation law.
		
		For $f\in L^1([0,T] \times [0,L])$ we define 
		the projection $\Pi_n$ into the space of piecewise constant functions for almost all $(t,x)$ by
		\begin{equation*}
			\Pi_n[f](t,x) \coloneqq \sum_{k=1}^{K_n} \sum_{m=1}^{M_n} \ind{(t_{k-1}, t_k)\times(x_{m-1}, x_m)}(t,x) 
			\int_{(t_{k-1}, t_k)\times(x_{m-1}, x_m)} f(s,y) \ddd (s,y).
		\end{equation*}
		Note that $\varphi-\Pi_n[\varphi] \to 0$ uniformly on the complement of a set with Hausdorff dimension 1 for every continuous function $\varphi \in C([0,T]\times[0,L])$. Due to the continuity of $\eta$, $D\eta$ vanishes on all sets with Hausdorff dimension 1.
		Using the divergence theorem, the residual of $\eta_n$ with respect to (\ref{perWeakEquation}) is given by
		\begin{align*}
		&\left| 
		\int_{(0,T)\times(0,L)} 
		\eta_n(\partial_t \varphi) + F(\eta_n)(\partial_x \varphi) 
		\ddd (t,x)
		-
		\int_{[0,L]}
		\varphi(0, x) \eta_n(0,x)
		\ddd x
		\right|
		\\
		&\quad\leq
		\left|
		\int_{(0,T)\times(0,L)} 
		\varphi
		\ddd 
		\Big(\partial_t \eta_n + \partial_x F(\eta_n)\Big)
		(t,x)
		\right|
		+
		\left|
		\int_{0}^{T}
		\varphi(t,0) \Big(F(\eta_n(t,0)) - F(\eta_n(t,L))\Big)
		\ddd t
		\right|
		\\
		&\quad
		\leq
		\left|
		\int_{(0,T)\times(0,L)} 
		(\varphi-\Pi_n[\varphi])
		\ddd 
		\Big(\partial_t \eta_n + \partial_x F(\eta_n)\Big)
		(t,x)
		\right|
		\\
		&\quad\phantom{=}
		+
		\left|
		\int_{0}^{T}
		\varphi(t,0) \Big(F(\eta_n(t,0)) - F(\eta_n(t,L))\Big)
		\ddd t
		\right|
		+
		\left|
		\int_{(0,T)\times(0,L)} 
		\Pi_n[\varphi]
		\ddd 
		\Big(\partial_t \eta_n + \partial_x F(\eta_n)\Big)
		(t,x)
		\right|.
		\end{align*}
		Similarly, we find for a positive test function $\varphi\geq 0$
		\begin{align*}
		&
		\int_{(0,T)\times(0,L)} 
		S(\eta_n)(\partial_t \varphi) + Q(\eta_n)(\partial_x \varphi) 
		\ddd (t,x)
		-
		\int_{[0,L]}
		\varphi(0, x) S(\eta_n)(0,x)
		\ddd x
		\\
		&\quad=
		-\int_{(0,T)\times(0,L)} 
		\varphi
		\ddd 
		\Big(\partial_t S(\eta_n) + \partial_x Q(\eta_n)\Big)
		(t,x)
		+
		\int_{0}^{T}
		\varphi(t,0) \Big(Q(\eta_n(t,0)) - Q(\eta_n(t,L))\Big)
		\ddd t
		\\
		&\quad
		=
		-\int_{(0,T)\times(0,L)} 
		(\varphi-\Pi_n[\varphi])
		\ddd 
		\Big(\partial_t S(\eta_n) + \partial_x Q(\eta_n)\Big)
		(t,x)
		\\
		&\quad \phantom{=}
		+
		\int_{0}^{T}
		\varphi(t,0) \Big(Q(\eta_n(t,0)) - Q(\eta_n(t,L))\Big)
		\ddd t
		-
		\int_{(0,T)\times(0,L)} 
		\Pi_n[\varphi]
		\ddd
		\Big(\partial_t S(\eta_n) + \partial_x Q(\eta_n)\Big) 
		(t,x).
		\end{align*} 
		As demonstrated, e.g., in the proof of \cite[Theorem 3.96]{ambrosio2000functions}, the composition $f\circ g$ of a Lipschitz continuous function $f$ and $BV$ function $g$ has bounded variation with
		$\Var(f\circ g) \leq \Lip(f) \Var(g)$.
		 Thus, as $S$, $Q$ and $F$ are Lipschitz continuous and $(\eta_n)_{n\in\naturals}$ is bounded in $BV((0,T)\times(0,L))$, both 
		$\partial_t \eta_n + \partial_x F(\eta_n)$
		and
		$\partial_t S(\eta_n) + \partial_x Q(\eta_n)$
		are finite signed Borel measures with uniformly bounded total variation in $n$.
		Using the obvious estimates implied by (\ref{eq:convergenceCondition}), we conclude the statement of the theorem completing the proof.
	\end{proof}

	\begin{remark}
		Similar arguments can be applied to systems of conservation laws in several dimensions in a straightforward way. Since for multidimensional systems of hyperbolic conservation laws weak solutions satisfying an entropy inequality are not unique in general, the resulting theorem only guarantees that each subsequence of $(\eta_n)_{n\in\mathbb{N}}$ contains a subsequence converging to a weak solution satisfying a weak entropy inequality under suitable assumptions. Note however, that we generally do not expect solutions of system of hyperbolic conservations laws to be in $BMV$. Therefore, approximation by networks with architectures of type (A) or (A') are expected to be either significantly overestimating the variation of the solution or to be of low accuracy.
	\end{remark}

	\subsection{Existence of a minimizing sequence}
	
	We continue by showing that sequences of neural networks satisfying the assumption of Theorem~\ref{thm:convergence} exist if the weak admissible solution of (\ref{eq:conservationLaw}) has bounded M-variation.
	\begin{theorem}[Existence of minimizing sequence]\label{thm:existencePDE}
		Let $u\colon [0,T]\times [0,L] \to \reals$ be a weak admissible solution of (\ref{eq:conservationLaw}) with periodic boundary conditions such that $u\in BMV([0,T]\times [0,L])$. Further, let $N_x,N_t\in\naturals$, $\Delta t \coloneqq T/N_t$, $h \coloneqq L/N_x$, $x_m \coloneqq mh$ and $t_k \coloneqq k \Delta t$.
		Then for any $\epsilon>0$, $L\in\naturals_{\geq 3}$ and any activation function $\sigma$ satisfying (\ref{eq:actiProperty}) there exists
		a network $\eta=\eta(\epsilon)$ with an architecture of type (A) with $L$ hidden layers and activation function $\sigma$ such that
		\begin{enumerate}[(i)]
			\item \label{BurgersExistence1}
			for all $k\in \{1,\dots, N_t\}$, $m\in \{1,\dots,N_x\}$
			\begin{equation*}
			\bigg|
			\int_{x_{m-1}}^{x_m} \eta(t_k, x)-\eta(t_{k-1}, x) \ddd x
			+
			\int_{t_{k-1}}^{t_k} F(\eta)(t, x_m)-F(\eta)(t, x_{m-1}) \ddd t
			\bigg| \leq \epsilon,
			\end{equation*}
			\item \label{BurgersExistence2}
			for all $k\in \{1,\dots, N_t\}$, $m\in \{1,\dots,N_x\}$
			\begin{equation*}
			\int_{x_{m-1}}^{x_m} S(\eta)(t_k, x)-S(\eta)(t_{k-1}, x) \ddd x
			+
			\int_{t_{k-1}}^{t_k} Q(\eta)(t, x_m)-Q(\eta)(t, x_{m-1}) \ddd t
			\leq \epsilon,
			\end{equation*}
		\item \label{BurgersExistence3}
		\begin{minipage}{\textwidth-28.1pt}
		\centering
		$
		\displaystyle
		\EMVar(\eta; [0,T]\times[0,L]) \leq 2 
		\left(
		\MVar(u; [0,T]\times[0,L])
		+
		\|u\|_{L^\infty([0,T]\times[0,L])}
		\right),
		$
		\end{minipage}
		\item \label{BurgersExistence4}
		\begin{minipage}{\textwidth-28.1pt}
		\centering
		$\displaystyle
			\|\eta(0,\cdot)-u(0,\cdot)\|_{L^1([0,L])} \leq \epsilon,
		$
		\end{minipage}
		\item \label{BurgersExistence5}
		\begin{minipage}{\textwidth-28.1pt}
			\centering
			$\displaystyle	
			\bigg| \int_{0}^{T} F(\eta(t,0))-F(\eta(t,L)) \ddd t \bigg| \leq \epsilon.
			$
		\end{minipage}
		\end{enumerate}
	\end{theorem}
\begin{remark}
	In particular, if $u\in BMV$ and $R$ is chosen large enough, there exist a sequence of networks $\eta_n$ with $\Loss_n(\eta_n) \to 0$. 
\end{remark}

		\begin{proof}
			We fix a representation of $u$ and do not consider functions being equal almost everywhere to be identical. We denote by $\tilde u$ the following extension of $u$ to $\reals^2$,
			\begin{equation*}
				\tilde u(t,x) = u\big(\min\{\max\{t,T\},0\}, x-kL \big).
			\end{equation*}
			Here, $k=k(x)\in\integers$ is chosen such that either $x \in \big(kL, (k+1)L\big]$ and $k>0$,
			$x \in \big[0,L]$ and $k=0$, or
			$x \in \big[kL, (k+1)L\big)$ and $k<0$.
			Note that $\tilde u$ is a weak admissible solution of equation (\ref{eq:conservationLaw}) on $[0,T]\times \reals$. Moreover, Corollary~\ref{cor:periodic_extension} implies that $\tilde u \in BMV([0,T]\times [0,2 L])$ with
			$
				\MVar(\tilde u; [0,T]\times [0,2 L]) \leq  2 
				\left(
				\MVar(u; [0,T]\times[0,L])
				+
				\|u\|_{L^\infty([0,T]\times[0,L])}
				\right).
			$
			Thus, for any $\delta>0$ there exist a finite set $J\subset\naturals$ and monotone functions $u_{j,\delta}$ for $j\in J$ such that 
			\begin{align*}
				\tilde u(t,x) &= \sum_{j\in J} u_{j,\delta}(t,x), 
				\qquad \text{for all }(t,x)\in H,\\
				\MVar(\tilde u; H) + \delta &\geq \sum_{j\in J} \sup\{u_{j,\delta}(x) \setsep x\in H\} - \inf\{u_{j,\delta}(x) \setsep x\in H\},
			\end{align*}
			where $H = [0,T]\times [0,2 L]$. Using a mollification argument we derive a sequence $(v_n)_{n\in\naturals}\subset C([0,T]\times[0,2 L]) \cap BMV([0,T]\times[0,2L])$ such that $v_n \to \tilde u$ in $C([0,T];L^1([0,2 L]))$ and such that 
			\begin{equation*}
				\MVar(v_n; H) \leq \MVar(\tilde u; H)
				\qquad \text{and} \qquad
				\|v_n\|_\infty \leq \|\tilde u\|_\infty.
			\end{equation*}
			In particular, we derive for any $t\in [0,T]$ and any locally Lipschitz continuous function $G\colon \reals \to \reals$
			\begin{equation*}
				\int_{0}^{2L} \big|G(v_n)(t,x)-G(\tilde u)(t,x) \big|\ddd x \xrightarrow{n\to\infty} 
				0.
			\end{equation*}
			The convergence of $v_n$ to $\tilde u$ in $C([0,T]; L^1([0,2L]))$ implies convergence in $L^1(H)$. Thus, there exists a subsequence $(v_{n_l})_{l\in \naturals}$ such that for all 
			$m\in \{0,1,\dots, N_x\}$
			\begin{equation}\label{eq:ExistenceEq3}
				\int_{0}^{T} \Big| G(\tilde{u})(t, x_m + \gamma) - G(v_{n_l})(t, x_m + \gamma) \Big| \ddd t
				\xrightarrow{n\to\infty} 
				0,
			\end{equation}
			for any locally Lipschitz continuous function $G\colon \reals \to \reals$ and almost all $\gamma \in [0,L]$. 
			
			Next, we show that for almost all shifts $\gamma$ in $x$, the sum over the boundary integrals vanish for all test volumes. This is a consequence of $\tilde u$ being an admissible weak solution of (\ref{eq:conservationLaw}). Note that we do not work with the trace of $\tilde{u}$ as a function of bounded variation.
			
			Denote by $\psi \in C^\infty(\reals)$ a nondecreasing function with $\psi(x)=0$ for all $x\leq 0$ and $\psi(x)=1$ for all $x\geq 1$. We define for $\alpha>0$, $k\in \{1,\dots, N_t\}$ and $m\in \{1,\dots, N_x\}$
			\begin{equation*}
				\varphi_{\alpha,\gamma,k,m}(t,x) 
				\coloneqq
				\psi\left( \frac{x-x_{m-1}-\gamma}{\alpha} \right)
				\psi\left( \frac{x_{m}+\gamma-x}{\alpha} \right)
				\psi\left( \frac{t-t_{k-1}}{\alpha} \right)
				\psi\left( \frac{t_k-t}{\alpha} \right).
			\end{equation*}
		Clearly, for any $\alpha > 0$
		\begin{equation*}
			\ind{(a+\alpha,b-\alpha)}(x)
			\leq
			\psi\left( \frac{x-a}{\alpha} \right)\!
			\psi\left( \frac{b-x}{\alpha} \right) 
			\leq  
			\ind{(a,b)}(x).
		\end{equation*}
		Thus, since $\tilde u \in C([0,T]; L^1([0,2L]))$
		\begin{equation*}
			\int_{H} G(\tilde u)(t,x) \partial t \varphi_{\alpha,\gamma, k,m}(t,x) \ddd (t,x)
			\xrightarrow{\alpha \searrow 0} 
			\int_{x_{m-1}+\gamma}^{x_m+\gamma} G(\tilde u)(t_{k-1},x) - G(\tilde u)(t_k,x) \ddd x
		\end{equation*}
		for any locally Lipschitz continuous function $G$ and any $\gamma \in [0,L]$. Similarly, for any Lebesgue point $x\in [0,2L]$ of $\int_{t_{k-1}}^{t-k}\tilde u(t, x-h) - \tilde u(t, x) \ddd t$ and $m, \gamma$ with $x=x_m+\gamma$ we find
		\begin{equation}\label{eq:ExistenceEq1}
			\int_{H} G(\tilde u)(t,x) \partial x \varphi_{\alpha,\gamma, k,m}(t,x) \ddd (t,x)
			\xrightarrow{\alpha \searrow 0} 
			\int_{t_{k-1}}^{t-k} G(\tilde u)(t,x_{m-1}+\gamma) - G(\tilde u)(t,x_m+\gamma) \ddd t.
		\end{equation}
		Taking $\varphi_{\alpha,\gamma,k,m}$ as the test function in (\ref{perWeakEquation}) and (\ref{perWeakEntropy}), we conclude that for almost all $\gamma\in[0,L]$,
		\begin{align}\label{eq:ExistenceEq2}
		\notag\bigg|
		\int_{x_{m-1}+\gamma}^{x_m+\gamma} \tilde u(t_k, x)-\tilde u(t_{k-1}, x) \ddd x
		+
		\int_{t_{k-1}}^{t-k} F(\tilde u)(t, x_m+\gamma)-F(\tilde u)(t, x_{m-1}+\gamma) \ddd t
		\bigg| &= 0,\\
		\int_{x_{m-1}+\gamma}^{x_m+\gamma} S(\tilde u)(t_k, x)-S(\tilde u)(t_{k-1}, x) \ddd x
		+
		\int_{t_{k-1}}^{t-k} Q(\tilde u)(t, x_m+\gamma)-Q(\tilde u)(t, x_{m-1}+\gamma) \ddd t
		&\leq 0.
		\end{align}
		Note that $\tilde{u}(t,x_0+\gamma) = \tilde{u}(t,x_{M_{n}}+\gamma)$ for all $\gamma\in(0,L)$ and $t\in[0,T]$ by construction.
		Choosing $\gamma\in(0,L)$ such that
		(\ref{eq:ExistenceEq3}), (\ref{eq:ExistenceEq1}) and (\ref{eq:ExistenceEq2}) are satisfied for all $k=1,\dots,K_n$, $m=1,\dots,M_n$ and small enough such that
		$
			\|\tilde u(0,\cdot+\gamma) - \tilde u(0,\cdot) \| < \frac{\epsilon}{3}\,,
		$
		we can take $l$ large enough such that $(t,x) \mapsto v_{n_l}(t,x+\gamma)$ satisfies properties (\ref{BurgersExistence1})-(\ref{BurgersExistence5}) with $\epsilon$ replaced by $2\epsilon/3$. Invoking Theorem~\ref{thm:UniBMVapprox} completes the proof.
	\end{proof}

	\section{Numerical Experiments}\label{sec:num_experiments}
	
	We have implemented a basic version of EVNNs with architectures of type (A) or (A') in PyTorch \cite{pytorch} as a proof of concept. In this section we present numerical results of our implementation of the algorithm presented in Section~\ref{sec:applications} on several benchmark problems.
	
	We always used the $\tanh$ activation function and networks with architectures of type (A'). Note that these can be rewritten as networks with architectures of type (A), see Remark~\ref{rem:AvsAprimeForTanh}.
	To guide the choice of the hyperparameters for the numerical tests, we performed a rough hyperparameter study
	for an unrelated regression problem in 2D.
	Unless specified otherwise, we use the hyperparameters indicated in Table~\ref{tab:numExperiments_hyperparametersettings} for each subnetwork. We use networks with architecture (A') with two subnetworks. To simplify the parallelization of our implementation, integrals are possibly approximated with higher precision then indicated by the integration tolerance.

	\begin{table}
	\centering
	\begin{tabular}{|c|c|c|c|c|c|}
			\hline 
		 LSUV gain & depth	& width &activation & weight modifier  	& epochs \\
		\hline 
		 	1		& 3		& 50	&$\tanh$	& $[x\mapsto |x|]$	& 8192 \\
		\hline 
	\end{tabular} 
	\caption{Default hyperparameter settings.}
	\label{tab:numExperiments_hyperparametersettings}
\end{table}

	All experiments were repeated several times and mean error, maximum error, minimum error as well as the standard deviation of errors are reported as indicators of the stability with respect to network initialization.
	
	All computations were performed with double precision.

	\subsection{Burgers Equation}
	We apply EVNNs to the classical benchmark problems for the inviscid Burgers equation,
	\begin{equation*}
		\partial_t u + \partial_x \frac{u^2}{2} = 0,
	\end{equation*}
	which were considered in \cite{de2022wpinns, chaumet2023}. These are a standing shock, a moving shock, a rarefaction wave and a test case with sine (smooth) initial data for which the solution develops a shock within finite time. For comparability, we use the same computational domains as in \cite{de2022wpinns, chaumet2023}. The reference solution for the standing and moving shock and the rarefaction wave is the respective exact solution of the problem. For the sine data the reference solution was generated using the local Lax-Friedrich method with a resolution of ($4096 \times 4096$ cells).
	
	For all test cases we used the entropy-entropy flux pair $(S,Q)$ given by $S(u) = \frac{1}{2}u^2$ and $Q(u) = \frac{1}{3}u^3$.
	We used uniform Cartesian mesh for the test volumes with $t_k-t_{k-1} = 2^{-7} = x_m-x_{m-1}$ for all $k=1,\dots,K$, $m=1,\dots,M$. Instead of calculating integrals over $[t_{k-1},t_k] \times \{x_m\}$ or $\{t_k\}\times [x_{m-1},x_m]$, we computed the respective average over the set. The integration tolerance was set to $\mathcal{E}_\text{int}=5e-3$. This corresponds to an integration tolerance of $5/2^{7}/10^{3}\approx 4e-5$ for the integrals over faces of the grid.
	
	We used a modified version of the loss function to simplify the optimization problem. In particular, we employ causal training \cite{wang2022respecting} to incentivize the network to learn earlier values before later ones. More precisely, we change $\Loss^\text{pde}$ to
	\begin{align*}
		&\Loss^\text{pde}\Big(\eta;\, (t_k)_{k=1}^{K_n},\, (x_m)_{m=1}^{M_n}\Big)
		\\
		&\qquad=
		\sum_{k=1}^{K_n} w_k \sum_{m=1}^{M_n} 
		\Bigg|
		I_{\mathcal{E}_\text{int}(n)}\Big(\eta_n(t_k,\cdot), \,[x_{m-1},{x_m}]\Big)
		-
		I_{\mathcal{E}_\text{int}(n)}\Big(\eta_n(t_{k-1},\cdot), \,[x_{m-1},{x_m}]\Big)
		\\
		&\qquad\phantom{=}
		+
		I_{\mathcal{E}_\text{int}(n)}
		\Big(
		F(\eta_n(\cdot,x_m)),\,
		[t_{k-1},t_k]
		\Big)
		-
		I_{\mathcal{E}_\text{int}(n)}
		\Big(
		F(\eta_n(\cdot,x_{m-1})),\,
		[t_{k-1},t_k]
		\Big)
		\Bigg|\,,
		\notag
	\end{align*}
	with $w_1=1$ and for $k=2,\dots, K_n$
	\begin{align*}
		w_k =&
		\exp\Bigg(-
		\zeta_{k} \sum_{l=1}^{k-1} \sum_{m=1}^{M_n} 
		\Bigg|
		I_{\mathcal{E}_\text{int}(n)}\Big(\eta_n(t_k,\cdot), \,[x_{m-1},{x_m}]\Big)
		-
		I_{\mathcal{E}_\text{int}(n)}\Big(\eta_n(t_{k-1},\cdot), \,[x_{m-1},{x_m}]\Big)
		\\
		&\qquad\phantom{=}
		+
		I_{\mathcal{E}_\text{int}(n)}
		\Big(
		F(\eta_n(\cdot,x_m)),\,
		[t_{l-1},t_l]
		\Big)
		-
		I_{\mathcal{E}_\text{int}(n)}
		\Big(
		F(\eta_n(\cdot,x_{m-1})),\,
		[t_{l-1},t_l]
		\Big)
		\Bigg|\Bigg)\,.
		\notag
	\end{align*}
Here, we chose
\begin{align}\label{eq:causalTrainigEps}
	\zeta_{k} &= \left( \frac{L}{M_n}\sum_{m=1}^{M_n}\left| I_{\mathcal{E}_\text{int}(n)}\Big(\eta_n(t_{k-1},\cdot), \,[x_{m-1},{x_m}]\Big)
	\right|+\left|
	I_{\mathcal{E}_\text{int}(n)}\Big(\eta_n(t_{k},\cdot), \,[x_{m-1},{x_m}]\Big)
	\right| \right)^{-1} 
\notag
\\&
\approx
\frac{1}{t_k-t_{k-1}}\int_{[t_{k-1}, t_k]\times[-L,L]} |\eta_n(t,x)| \ddd (t,x).
\end{align}

	The same modification was applied to the entropy inequality. In this case, $\veta$ is replaced by $S(\veta)$ in (\ref{eq:causalTrainigEps}).
	
	Moreover, as in \cite{cvPINNs} the initial data on the temporal boundary and fluxes on the spatial boundary
	${0}\times[-L,L]$ and $\left(\{-L\}\cup\{L\}\right)\times [0,T]$
	are calculated with the respective prescribed function instead of with $\eta$, e.g.\ we replace
	\begin{equation*}
		\small
		I_{\mathcal{E}_\text{int}(n)}\Big(\eta(0,\cdot), \,[x_{m-1},{x_m}]\Big) 
		\text{ by }
		\int_{[x_{m-1},{x_m}]}\! u_0(x) \ddd x
		 \text{ and }
		I_{\mathcal{E}_\text{int}(n)}
		\Big(
		F(\eta(\cdot,L)),\,
		[t_{k-1},t_k]
		\Big)
		\text{ by }
		\int_{[t_{k-1},{t_k}]} \!u(t,L) \ddd x.
	\end{equation*}
	
	This is done both in $\Loss^\text{pde}$ and $\Loss^\text{ent}$. For periodic boundary conditions the fluxes on $\{-L\}\times [0,T]$ and $\{L\}\times [0,T]$ are replaced by the arithmetic mean of both.
	These changes ensure that the loss on boundary cells is meaningful even if the boundary condition is not learned yet. 
	Our experiment regarding ODEs in Appendix~\ref{App:sec:numExperiments_ODE} show that the network can learn more accurate solutions than the chosen (mesh) hyperparameter suggests.
	Therefore we resolve initial and boundary terms to higher accuracy than the uniform Cartesian grid allows by including $\Loss^\text{IC}$ and $\Loss^\text{BC}$.
	
	We weighted $\Loss^\text{Ic}$ and $\Loss^\text{BC}$ by a factor of 100. This corresponds roughly to the ratio of numbers of cells to the number of cells on the boundary. The entropy loss was weighted by a factor of 0.1 as the entropy condition is only enforced to single out non-physical solutions.
	
	We use the popular Adam optimizer \cite{kingma2017adam} with a learning rate of 0.01. We do not employ weight decay as we already penalize the variation of the network for regularization. The history of the Adam optimizer is reset every 1000 epochs to avoid local minima.
	
	It should be noted, that the experiments in Section~\ref{sec:NumExperiments_shocks} and \ref{sec:NumExperiments_rarefaction} were performed with Dirichlet boundary conditions. Here, we prescribed values on the whole spatial boundary corresponding to reference solution of the respective Riemann problem. The numerical results indicate that the algorithm is not limited to periodic boundary conditions.
	
	Each experiment was repeated 16 times to investigate the influence of the initialization of the network parameters. The hyperparameter $R$ was chosen to be twice the expected $M$-variation of the admissible weak solution for all test cases. The parameter $\epsilon_n$ in the entropy loss was set to 0.01 for all experiments. Table~\ref{tab:BurgersErrors} contains the average error, standard deviation of the errors, maximal error and minimal error over all network initializations. Additional, we computed the error of the average over all networks. This last is the error that was given in \cite{de2022wpinns, chaumet2023}.
	
	\begin{table}
		\centering
		\begin{tabular}{|c|c|c|c|c|c|c|}
			\hline       experiment& note&  mean error & error std. & max error & min error & av. error
			\\
			\hline
			\ref{sec:NumExperiments_shocks} & standing shock &2.41e-03&1.04e-03&5.19e-03&4.04e-04&1.64e-03
			\\ \hline
			\ref{sec:NumExperiments_shocks} & moving shock &8.66e-03&2.37e-02&1.00e-01&7.98e-04&7.23e-03
			\\ \hline
			
			\ref{sec:NumExperiments_rarefaction} &
			&9.24e-03&2.94e-03&1.43e-02&5.51e-03&3.04e-03
			\\ \hline
			
			\ref{sec:NumExperiments_Sine}& periodic BC
			&6.04e-02&2.76e-02&1.32e-01&1.66e-02&4.81e-02
			\\ \hline
			
			\ref{sec:NumExperiments_Sine}& Dirichlet BC
			&3.99e-02&2.56e-02&9.21e-02&1.16e-02&3.01e-02
			\\ \hline
		\end{tabular}
		\caption{Errors for the test cases for the Burgers equation.
		All errors are with respect to the relative $L^1$ norm. We give the mean, standard derivation, maximum and minimum of the errors with respect to the 16 independently initialized training runs. The last column contains the error of the function defined as the average of all network runs for the specified test case. This is the error reported by \cite{chaumet2023, de2022wpinns} after different numbers of epochs.}
		\label{tab:BurgersErrors}
	\end{table}
	
	\subsubsection{Standing and moving shock}\label{sec:NumExperiments_shocks}
	For the first two experiments, we set $T=0.5$, $L=1$. We first consider the initial data $u(0,x) = 1$ for $x<0$ and $u(0,x)=-1$ otherwise, leading to a standing shock at $x=0$. Next, for the moving shock, the initial data are given by $u(0,x) = 1$ for $x<0$ and $u(0,x)=0$ otherwise. The exact solution consists of a shock at $x=0$ which moves with speed $1/2$ to the right. In both cases Dirichlet boundary conditions were enforced at $x=-1$ and $x=1$.
	We set $R=4$ for the moving shock and $R=2$ for the standing shock.
	
	Numerical results for the first random network initialization of 16 are presented in Appendix~\ref{App:sec:AdditionalNumericalBurgersRersults}. For the standing shock, the results are very similar for all initializations, see Table~\ref{tab:BurgersErrors}. For the moving shock case, one network initialization was not able to learn the correct shock speed. In fact, the network initialization reproduced the standing shock. This is the reason for the increased errors for this test case.

\subsubsection{Rarefaction wave}\label{sec:NumExperiments_rarefaction}
We considered the initial data given by
$u_0(x)=-1$ for $x<0$ and $u_0(x)=1$ otherwise.
The weak admissible solution $u$ is given by
\begin{equation*}
	u(t,x) =
	\begin{cases}
		-1 \quad &\text{if }x<-t,\\
		x/t \quad &\text{if } x \in [-t,t],\\
		1  \quad &\text{if }x>t .
	\end{cases}
\end{equation*}
We set $R=4$ for all training runs.
This test case illustrates the importance of the additional entropy admissibility condition. It is easy to verify that the constant extension of the initial data in time yields a non-physical weak solution with the same values on the boundary, which does not satisfy the entropy inequality. As we also enforce the entropy inequality during training, the network picks up the right solution for all runs.

\subsubsection{Sine test case}\label{sec:NumExperiments_Sine}
	
	For the last test case we considered the initial data given by $u_0(x)=-\sin(\pi\,x)$ with periodic boundary conditions. The admissible weak solution develops a shock within the computational domain. We compared the results of enforcing periodic boundary conditions and of enforcing homogeneous Dirichlet boundary conditions, see Table~\ref{tab:BurgersErrors}. We set $R=8$ for all training runs for this test case. 
	Similar to \cite{de2022wpinns, chaumet2023}, we used more epochs for the sine test case. Instead of 8192 we used 16384 epochs. We present the results for the first network initialization with periodic boundary conditions in Figure~\ref{fig:2DSineTestCase}. For some of the random initializations, the shock position is shifted resulting in a larger error compared to the different test cases, see Table~\ref{tab:BurgersErrors}. In Appendix~\ref{App:sec:AdditionalNumericalBurgersRersults} the standard deviation of the networks is plotted to illustrate the issue, see Figure~\ref{App:fig:SineInitialDataStd}.

\begin{figure}
	\begin{subfigure}[b]{0.49\textwidth}
		\centering
		\includegraphics[width=0.8\linewidth]{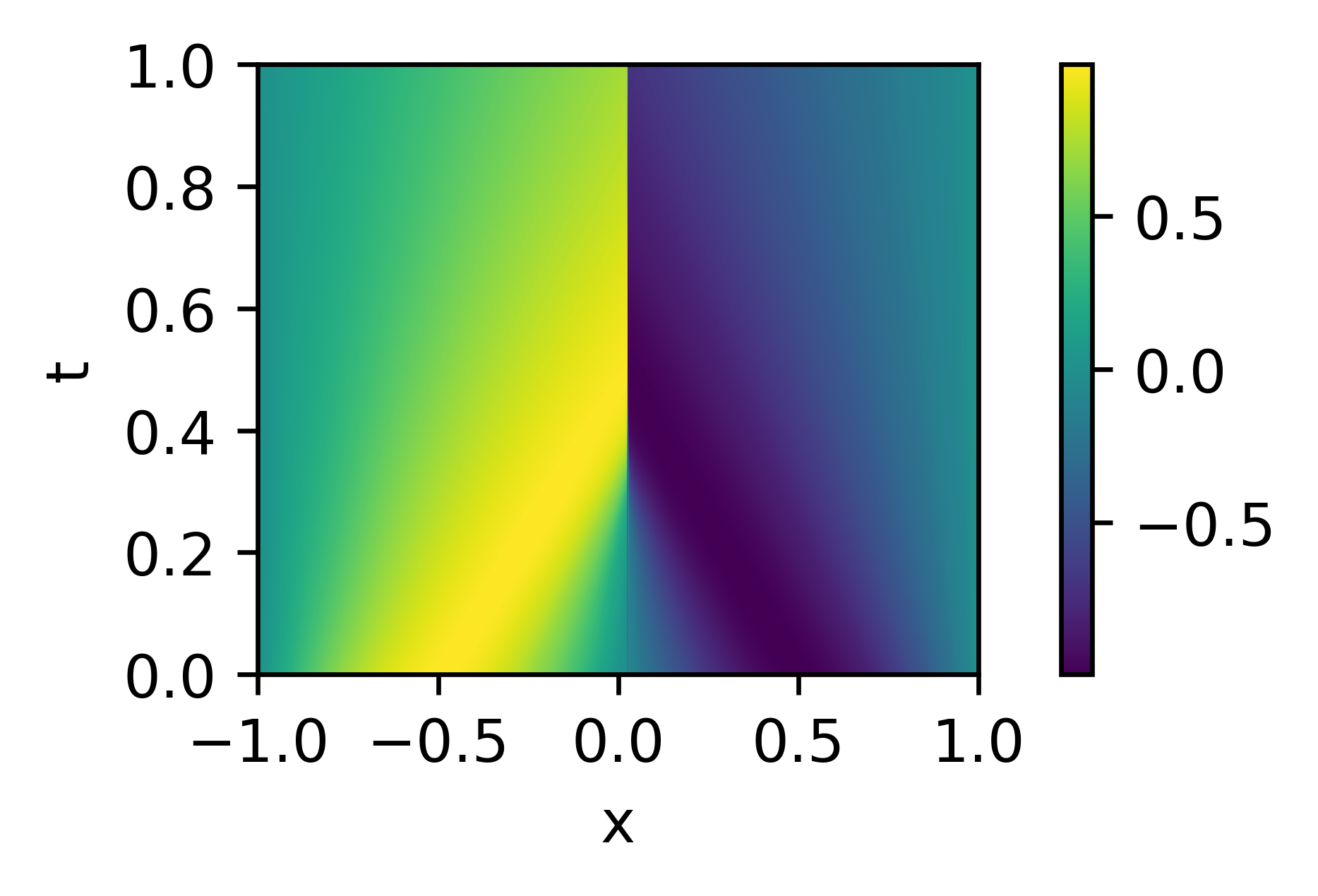}
		\caption{Network solution.}
		\label{fig:2DSineTestCase_net}
	\end{subfigure}
	\begin{subfigure}[b]{0.49\textwidth}
		\centering
		\includegraphics[width=0.8\linewidth]{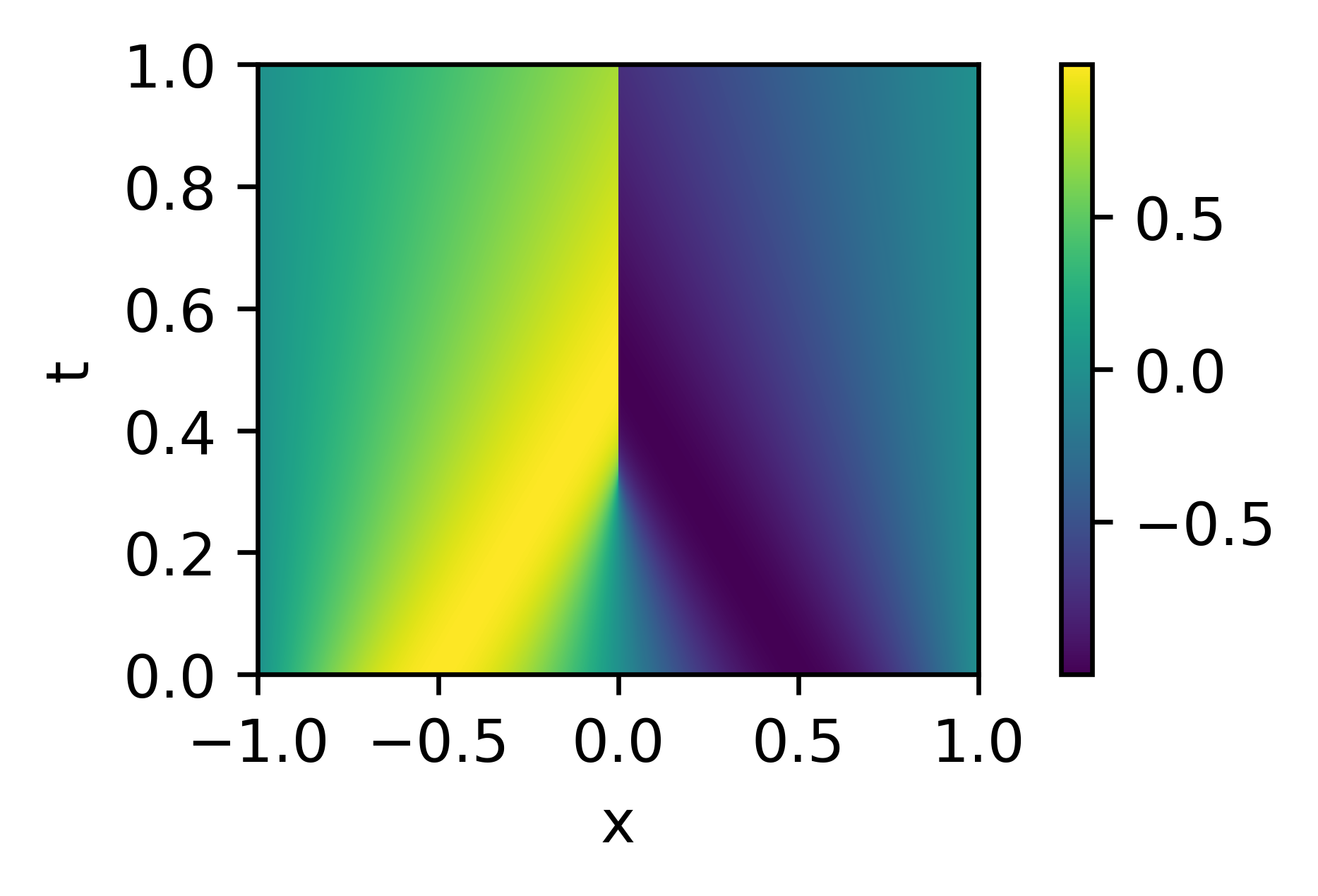}
		\caption{Reference solution (finite volume method).}
		\label{fig:2DSineTestCase_ref}
	\end{subfigure}
	
	\begin{subfigure}[b]{0.49\textwidth}
		\centering
		\includegraphics[width=0.8\linewidth]{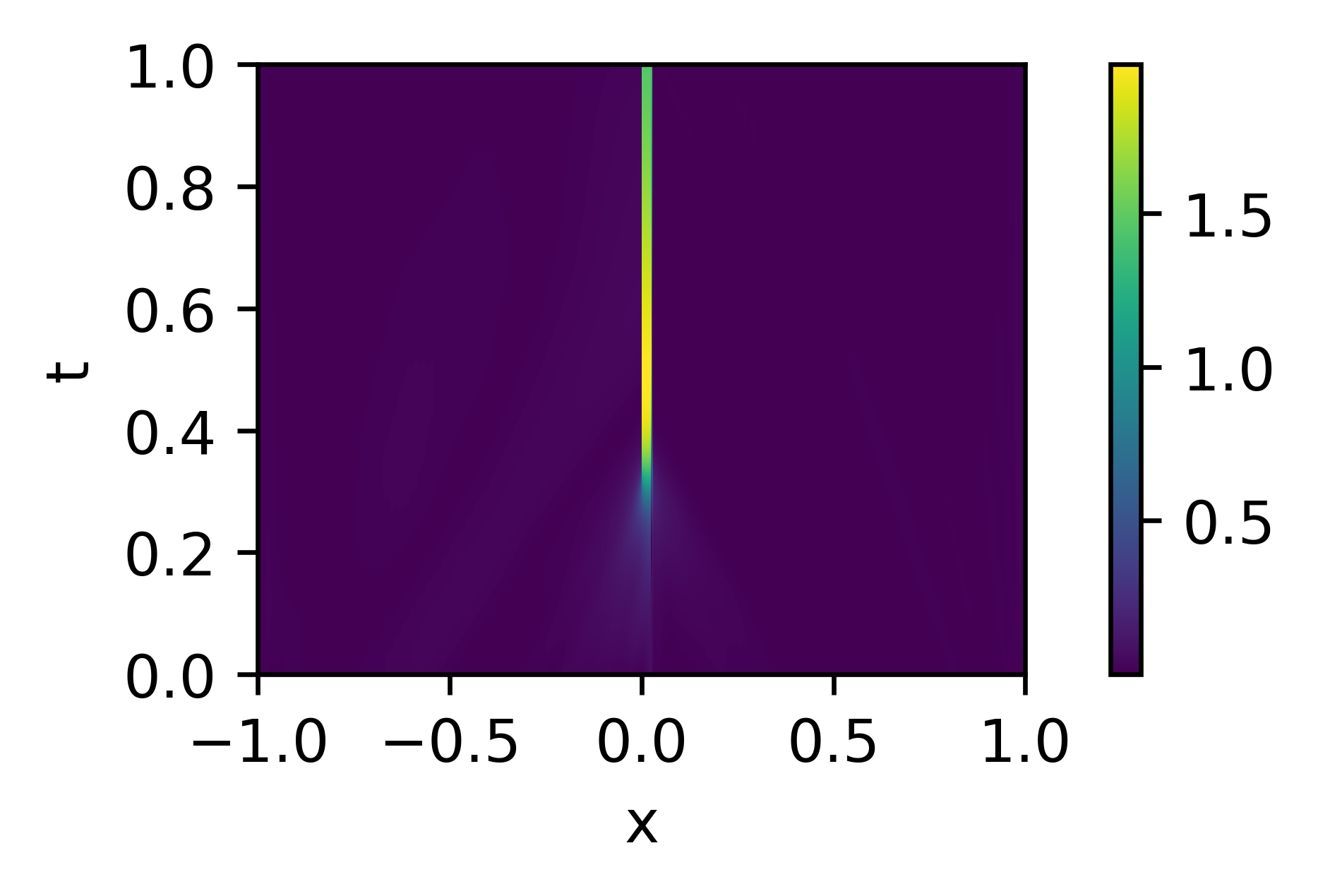}
		\caption{Error.}
		\label{fig:2DSineTestCase_err}
	\end{subfigure}
	\begin{subfigure}[b]{0.49\textwidth}
		\centering
		\includegraphics[width=0.8\linewidth]{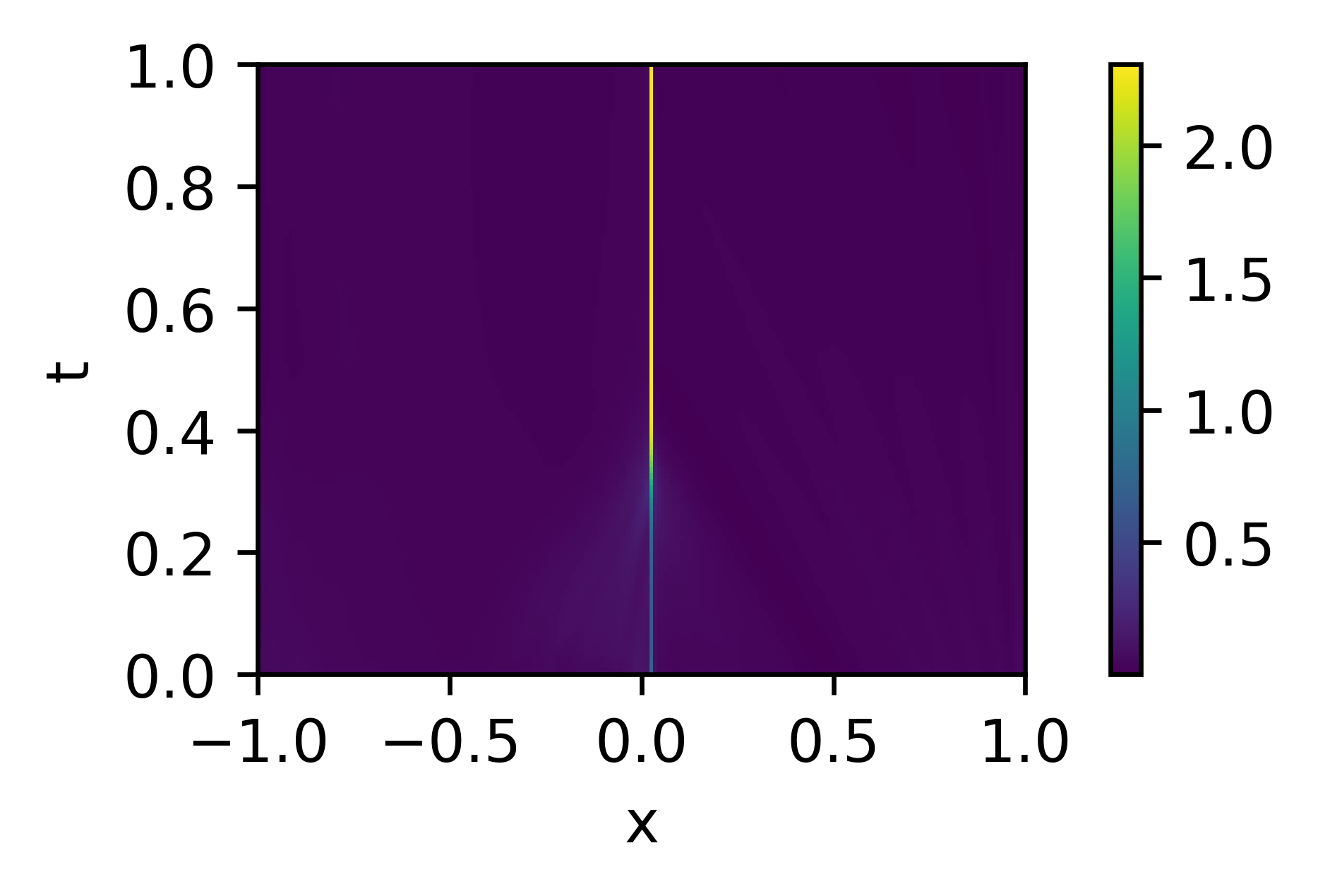}
		\caption{Estimated M-variation on a $100\times100$ grid.}
		\label{fig:2DSineTestCase_var}
	\end{subfigure}
\caption{The first network initialization trained for the sine initial data test case with periodic boundary condition.}
\label{fig:2DSineTestCase}
\end{figure}

	\section*{Acknowledgments}
{M\' aria Luk\' a\v cov\' a - Medvi{\leavevmode\setbox1=\hbox{d}%
		\hbox to 1.05\wd1{d\kern-0.3ex{\char039}\hss}}ov\' a} has been funded by the Deutsche Forschungsgemeinschaft (DFG, German Research Foundation) - 
Project number 233630050 - TRR 146 as well as by 
SPP 2410 Hyperbolic balance laws: complexity, scales and randomness project number 525853336. She is grateful to the Gutenberg Research College for supporting her research. The research of Simon Schneider was funded by Mainz Institute of Multiscale Modeling M3odel and the Gutenberg Research College and by the Deutsche Forschungsgemeinschaft (DFG, German Research Foundation) - Project number 257899354 TRR 165 Waves to Weather.  
\bibliography{lit.bib}
\newpage
\appendix

\section{Extended discussion of the approximation properties of EVNNs}\label{App:approxProps}

This section is an extended version of Section~\ref{approxProps}.
We prove several results concerning approximation by networks with architectures of type (A) or type (A'). We start by showing  that networks with architectures of type (A) and networks with architectures of type (A') can essentially approximate the same functions, even for activation functions which do not satisfy \begin{equation}\label{App:eq:generalized_uneven_property}
	a\sigma(-bx-c)=-\sigma(x) \qquad \text{for all }x\in\reals.
\end{equation}
Throughout this section we require the activation to satisfy
\begin{equation}\label{App:eq:actiProperty}
	\lim_{x\to-\infty}\sigma(x) = \sigma_{-}, \qquad \lim_{x\to \infty}\sigma(x) = \sigma_{+},
\end{equation}
for some constants $\sigma_{-}<\sigma_{+}$.

\begin{lemma}\label{App:lem:typeAvsAprime}
	Let $\sigma$ be continuous, nondecreasing and satisfy (\ref{App:eq:actiProperty}) and $U\subset \reals^d$ be compact. For each $\epsilon>0$, all $L,\,m_1,\dots,m_L\in\naturals$, $m_0=d$, $m_{L+1}=N$, all matrices $\matA_1,\dots, \matA_{L+1}$, $\matA_i\in \reals_{\geq 0}^{m_i\times m_{i-1}}$ and all $\vb_1,\dots, \vb_{L+1}$, $\vb_i\in \reals^{m_i}$, there exist $P \in \naturals$, matrices $\tilde \matA_1, \dots, \tilde \matA_{L+1}$, $\tilde\matA_i  \in\reals_{\geq 0}^{P m_i\times P m_{i-1}}$ for $2\leq i\leq L$, $\tilde\matA_1  \in\reals_{\geq 0}^{P m_1\times m_0}$, $\tilde\matA_{L+1}  \in\reals_{\geq 0}^{m_{L+1}\times P m_{iL}}$ and $\tilde \vb_1, \dots, \tilde \vb_{L}$ with $\tilde \vb_i\in \reals^{P m_i}$, $\vb_{L+1}\in \reals^N$, such that for all $\vx\in U$, $i\in\{1,\dots,N\}$,
	\begin{align}\label{App:eq:lemArchitectureComparison}
		\Bigg|\bigg(
		&\matA_{L+1} \cdot \hat{\sigma} \left( \dots 
		\matA_2 \cdot \hat{\sigma} \left(
		\matA_1 \cdot \vx + \vb_1
		\right)+\vb_2 \dots \right) + \vb_{L+1}
		\bigg)_i
		\notag
		\\
		-&
		\left(
		-\tilde \matA_{L+1} \cdot \hat{\sigma} \left( \dots 
		\tilde \matA_2 \cdot \hat{\sigma} \left(
		\tilde \matA_1 \cdot (-\vx) + \tilde \vb_1
		\right) + \tilde \vb_2 \dots \right) - \tilde \vb_{L+1}
		\right)_i
		\Bigg| \leq \epsilon.
	\end{align}
\end{lemma}

In order to prove the above result we first show the following auxiliary result.

\begin{lemma}\label{App:lem:aux_acti_property}
	Let $\sigma\colon \reals \to \reals$ be nondecreasing and satisfy (\ref{App:eq:actiProperty}). For any $R>0$, $f\in C([-R,R])$ nondecreasing and $\epsilon>0$ there exists $P=\max\{\lceil 3(f(R)-f(-R))/\epsilon\rceil,\, 1\}$, $a\in \reals_{\geq 0}^P$, $\ve{b}\in \reals^P$, $\ve{c}\in \reals_{\geq 0}^P$, such that for all $x\in [-R,R]$
	\begin{equation*}
		\left| f(x) - \ve{c} \cdot \hat{\sigma}(x\ve{a}+\ve{b}) \right| \leq \epsilon 
		.
	\end{equation*}
\end{lemma}

\begin{proof}
	Without loss of generality we may assume that $\sigma(x) \xrightarrow{x\to-\infty} 0$ and $\sigma(x) \xrightarrow{x\to \infty} 1$.
	We	define $x_P=R$ and for $i=0,1,\dots, P-1$
	\begin{equation*}
		x_i=\inf\left\{x\in[-R,R] \setsep f(x) =  f(-R)+i\frac{\epsilon}{3} \right\}.
	\end{equation*}
	Next, we choose $S\geq 0$ large enough such that 
	$\sigma(-S) \leq \frac{1}{P}$ and $\sigma( S) \geq 1-\frac{1}{P}$.
	We define 
	\begin{equation*}
		\psi \colon x \mapsto f(-R) + \sum_{i=1}^{P} \sigma\bigg( 2S \frac{x-x_{i-1}}{x_i-x_{i-1}} -S\bigg)
		\Big( f(x_i)-f(x_{i-1}) \Big).
	\end{equation*}
	Clearly, for $x_{j}\leq x \leq x_{j+1}$
	\begin{equation*}
		f(-R) + (1-1/P)\sum_{i=1}^{j}
		\Big( f(x_i)-f(x_{i-1}) \Big)
		\leq 
		\psi(x)
		\leq
		f(x_{j+1})+\sum_{i=j+2}^{P}
		\Big( f(x_i)-f(x_{i-1}) \Big)/P.
	\end{equation*}
	Thus, $|\psi(x)-f(x)| \leq \epsilon$.
\end{proof}

\begin{proof}[Proof of Lemma~\ref{App:lem:typeAvsAprime}] We fix $M\geq \max_{l,i,j}|(\matA_{l})_{i,j}|$, $M\neq 0$.
	Let $\delta$ be small enough such that $|\sigma(\vx)-\sigma(\vy)| \leq \epsilon/(2Mm_L)$ for all $|\vx-\vy|\leq\delta$.
	For $L=0$ hidden layers the statement of the lemma is trivial. The general case will be shown by induction with respect to $L$. Thus, we assume that we can choose $\bar \matA_1,\dots, \bar \matA_L$ such that for all $j=1,\dots,m_L$ 
	\begin{align*}
		\Bigg|\bigg(
		&\matA_{L} \cdot \hat{\sigma} \left( \dots 
		\matA_2 \cdot \hat{\sigma} \left(
		\matA_1 \cdot \vx + \vb_1
		\right)+\vb_2 \dots \right) + \vb_{L}
		\bigg)_j
		\\
		-&
		\left(
		-\bar \matA_{L} \cdot \hat{\sigma} \left( \dots 
		\bar \matA_2 \cdot \hat{\sigma} \left(
		\bar \matA_1 \cdot (-\vx) + \bar \vb_1
		\right) + \bar \vb_2 \dots \right) - \bar \vb_{L}
		\right)_j
		\Bigg| \leq \delta.
	\end{align*}
	We choose $R>0$ such that for all $j=1,\dots,m_L$ 
	\begin{equation*}
		\Bigg|\bigg(
		\matA_{L} \cdot \hat{\sigma} \left( \dots 
		\matA_2 \cdot  \hat{\sigma} \left(
		\matA_1 \cdot \vx + \vb_1
		\right)+\vb_2 \dots \right) + \vb_{L}
		\bigg)_j
		\Bigg| \leq R-\delta \qquad \text{for all }\vx\in U.
	\end{equation*}
	Let $P,\ve{a},\ve{b},\ve{c}$ be given by Lemma~\ref{App:lem:aux_acti_property} with $f(x)=-\sigma(-x)$ and tolerance $\epsilon/(2Mm_L)$.
	For $1\leq l\leq L-1$ we choose $\tilde \matA_l = \bar \matA_l$, $\tilde \vb_l = \bar \vb_l$. Furthermore, we define
	$\tilde \matA_L\in \reals_{\geq 0}^{ P m_L \times m_{L-1}}$ and $\tilde \vb_l \in \reals^{Pm_L}$ by
	\begin{equation*}
		(\tilde \matA_L)_{i,j} = (\ve{a})_r (\bar \matA_L)_{\lceil i/P \rceil, j}, 
		\qquad
		(\tilde \vb_L)_{i} = (\vb)_{r} + (\ve{a})_r(\overline{\vb}_L)_{\lceil i/P \rceil} 
		\quad \text{with }
		r = i-\left\lfloor\frac{i-1}{P}\right\rfloor P.
	\end{equation*}
	Finally we define $\tilde \matA_{L+1} \in \reals^{N \times Pm_L}$ and ${\tilde \vb}_{L+1} \in \reals^{N}$ by
	\begin{equation*}
		(\tilde \matA_{L+1})_{i,j} = (\ve{c})_{r} (\matA_{L+1})_{i,\lceil j/P \rceil},
		\quad \text{with }
		r = j-\left\lfloor\frac{j-1}{P}\right\rfloor P, \qquad
		{\tilde \vb}_{L+1} =  -\vb_{L+1}.
	\end{equation*}
	It is straightforward to verify that this choice for $\tilde\matA_{l}$ and $\tilde \vb_l$ satisfies (\ref{App:eq:lemArchitectureComparison}).
\end{proof} 
It is well-known that the set of fully-connected feedforward neural networks has the \textit{universal approximation property} \cite{cybenko1989approximation, hornik1989multilayer}.
It states that for any continuous function $f\in C(K)$ defined on a compact set $U\subset \reals^d$ and error tolerance $\epsilon>0$ there exists a neural network $ \eta :\reals^d\to \reals$ such that 
$
\|f-\eta\|_{C(K)}\leq \epsilon.
$
The proofs given in the classical works \cite{cybenko1989approximation, hornik1989multilayer} show the existence of shallow neural networks (i.e. networks with only one hidden layer) satisfying the above estimate but do not quantify how many neurons in the hidden layer are needed to achieve this accuracy. Therefore, the neural network $\eta$ might have a large number of parameters and the result does not necessarily indicate that neural networks are particularly efficient for approximation. Still neural networks can be used to approximate large classes of functions. For example, approximation in $L^P$ spaces, $1\leq p<\infty$, can be established by a density argument. 
These results are easily adaptable to EVNNs with architectures of type (A) or (A') which form a subset of all multilayer fully connected feedforward neural networks. This is due to shallow neural networks being approximable by networks with an architecture of type (A) and (A'), as the following lemma demonstrates.

\begin{lemma}[Shallow Neural Networks as EVNNs]\label{App:shallowEVNN}
	Let $\sigma$ be nondecreasing, continuous and satisfy (\ref{App:eq:actiProperty}).
	Let $U\subset \reals^d$ be compact, $\epsilon>0$ and $\eta$ a shallow neural network
	defined by $\eta(\vx)= \matA_2 \cdot \hat{\sigma} (\matA_1\cdot \vx+\vb_1)+\vb_2$.
	There exist networks $\eta^{(1)}$, $\eta^{(2)}$
	with architecture of type (A) and (A') respectively
	such that $|\eta(\vx)_i-\eta^{(1)}(\vx)_i|\leq \epsilon$, $|\eta(\vx)_i-\eta^{(2)}(\vx)_i|\leq \epsilon$ for all $\vx\in U$ and $i=1,\dots,N$. Furthermore, $\eta_2$ can be chosen such that $K\leq 2^{d-1}$.
\end{lemma}

\begin{proof}
	It is easy to see that
	\begin{equation*}
		\eta(\vx)_i = \sum_{j=1}^{m_1}  (\matA_2)_{i,j}\ \sigma \left( \sum_{k=1}^d  |(\matA_1)_{j,k}|\, s_{j,k} \, (\vx)_k + (\vb_1)_j  \right) +(\vb_2)_i
		= 
		\sum_{j=1}^{m_1}  (\matA_2)_{i,j}\ \sigma \left(  \matB_{1,j}\cdot(\vs_j\odot \vx) + b_j  \right) +(\vb_2)_i
	\end{equation*}
	with 
	\begin{equation*}
		\matB_{1,j} = (|(\matA_1)_{j,1}|,\,|(\matA_1)_{j,2}|,\,...,\,|(\matA_1)_{j,d}|)\in\reals^{1\times d}, \quad \vs_j = (s_{j,1},\,s_{j,2},\,...,\,s_{j,d})\in \reals^d, \quad s_{j,k} = \begin{cases}
			1,\; &\text{if }(\matA_1)_{j,k}\geq 0,\\
			-1,\; &\text{otherwise.}\end{cases}
	\end{equation*}
	According to Lemma~\ref{App:lem:typeAvsAprime} we may choose $\tilde \matA_{2,j}, \tilde \matA_{1,j}$ with positive elements and $\tilde \vb_{2}, \tilde \vb_{1,j}$ such that for all $\vx\in U$ and $i=1,\dots,N$
	\begin{equation*}
		\Bigg| 
		\sum_{j=1}^{m_1}  \min\{(\matA_2)_{i,j}, 0\} \, \sigma \left(  \matB_{1,j}\cdot(\vs_j\odot \vx) + (\vb_1)_j
		\right) + (\vb_2)_i
		-
		\bigg(
		\sum_{j=1}^{m_1} \tilde \matA_{2,j} \cdot \hat \sigma\Big( \tilde \matA_{1,j} \cdot(-\vs_j \odot \vx)+\tilde \vb_{1,j}\Big) + \tilde \vb_{2} 
		\bigg)_i
		\Bigg| < \epsilon.
	\end{equation*}
	This proves the existence of $\eta^{(1)}$. For $\eta^{(2)}$ it is sufficient to note that due to Lemma~\ref{App:lem:typeAvsAprime} we find $\tilde \matA_{2,j}, \tilde \matA_{1,j}, \tilde \vb_{2,j}, \tilde \vb_{1,j}$ such that for all $\vx\in U$ and $i=1,\dots,N$
	\begin{equation*}
		\Bigg|\sum_{j\in J'} (\matA_2)_{i,j}\sigma \left(  \matB_{1,j}\cdot(\vs_j\odot \vx) + (\vb_1)_j
		\right)  + (\vb_2)_i
		-
		\bigg(\sum_{j\in J'} \tilde \matA_{2,j} \cdot \hat \sigma\Big( \tilde \matA_{1,j} \cdot(-\vs_j \odot \vx)+\tilde \vb_{1,j}\Big) + \tilde \vb_{2,j}
		\bigg)_i
		\Bigg| < \epsilon,
	\end{equation*}
	where $J' = \Big\{ j\in \{1,\dots,m_1\}\setsep (\vs_j)_1 =-1 \Big\}$.
\end{proof}

\begin{lemma}[Universal Approximation by EVNNs]\label{App:lem:UniversalApprox}
	Given an arbitrary $L\in \naturals$, a non-decreasing continuous activation functions $\sigma$ satisfying (\ref{App:eq:actiProperty}), a compact set $U\subset \reals^d$, a finite measure $\mu$ on $U$, $p\in(0,\infty)$ and functions $f_1\in C^k(U;\reals)$, $f_2\in L^p_\mu(U;\reals)$ and a measurable function $f_3\colon U\to \reals$, there exist three sequences of networks with architectures of type (A) (or type (A')) with $L$ hidden layers  $(u^{(1)}_n)_{n\in\naturals}$, $(u^{(2)}_n)_{n\in\naturals}$ and $(u^{(3)}_n)_{n\in\naturals}$ with  activation function $\sigma$ and one single hidden layer such that:
	\begin{enumerate}[(i)]
		\item $u^{(1)}_n$ converges in $C(U;\reals)$ to $f_1$ as $n\to\infty$,\label{App:ena}
		\item $u^{(2)}_n$ converges in $L^p_\mu(U;\reals)$ to $f_2$ as $n\to\infty$,\label{App:enb}
		\item $u^{(3)}_n$ converges in measure to $f_3$ with respect to $\mu$ for $n\to\infty$.\label{App:enc}
	\end{enumerate}
\end{lemma}

\begin{proof}
	Due to Lemma~\ref{App:lem:typeAvsAprime} it is sufficient to consider EVNNs of type (A).
	We start with the case $L=1$.
	According to \cite[Theorem 2.4 and Corollary 2.2]{hornik1989multilayer}, there exist sequences of shallow neural networks with activation function $\sigma$ such that (\ref{App:ena})-(\ref{App:enc}) hold. Thus, the claim follows from Lemma~\ref{App:shallowEVNN}. 
	
	For networks with more hidden layers it is sufficient to show that the identity mapping can be approximated by the last $L-1$ layers of each subnetwork in a suitable way. This follows immediately from Lemma~\ref{App:lem:aux_acti_property}. 
\end{proof}

\begin{remark}
	Looking at the proof of Lemma~\ref{App:shallowEVNN}, the famous result by Barron \cite{barron1993universal} on the approximation of certain very regular functions without the curse of dimensionality remains true for networks with architectures of type (A) and (A'), albeit with a worse error rate in terms of the parameters of the network.
\end{remark}

While it is important to know that the networks with architectures of type (A) are able to approximate functions up to arbitrary accuracy in various common metrics, for our purposes a different question is more natural: 
\begin{question}\label{App:q:1}
	Can functions with bounded M-variation be approximated up to arbitrary accuracy with networks whose estimated M-variation remains bounded? 
\end{question}
This question is essential for the algorithms presented in Section~\ref{sec:applications}. For example, the M-variation of an EVNN determines how many sampling points are necessary to guarantee a given accuracy of the numerical integration scheme presented in Section~\ref{sec:IntegratingBMVFunctions}. 
If the M-variation increases as we approach a target function, the numerical integration would get increasingly computationally expensive. 
Clearly, without a bound on the computational cost an algorithm need not terminate within reasonable or even finite time.

For $d=1$ the answer to Question~\ref{App:q:1} is obviously positive. In fact, shallow neural networks are sufficient in this case. For $d>1$ the situation is more complex.
We illustrate this by presenting a negative result for shallow networks architectures of type (A) or (A') for $d=2$.

\begin{theorem}\label{App:thm:shallowMonotoneApproximationFailure}
	For any $\epsilon>0$ there exists a nondecreasing function $f\colon [0,1]^2 \to \reals$ such that for any network $\eta\colon [0,1]^2 \to \reals$ with one hidden layer, nondecreasing activation function $\sigma$ and nonnegative weights,
	\begin{equation}\label{App:shallow_monotone_net}
		\eta(x,y) = \sum_{j=1}^{W} \lambda_{j}\ \sigma\left( a_{j,1} x+a_{j,2} y + b_j  \right) + b_0, \qquad a_{j,1},\,a_{j,2} \in\reals_{\geq 0},\,\lambda_{j} \geq 0 \text{ for all }j=1,\dots,W,
	\end{equation}
	it holds that $ \|f-\eta\|_{L^1([0,1]^2)}>\epsilon$. 
\end{theorem}

\begin{proof}
	We explicitly construct such a function $f$ for a given $\epsilon$.
	Let $f = 16 \epsilon \ind{U}$ with
	$
	U=\{(x,y)\in\reals^2 \setsep x> 0.5,\, y>0.5\}.
	$
	Further
	\begin{equation*}
		\begin{array}{>{\displaystyle}r>{\displaystyle}l>{\displaystyle}r>{\displaystyle}l}
			Q_0 &= [0,1/4]^2, 
			&Q_1 &= [1/4, 1/2] \times [3/4,1], \\[2mm]
			Q_2 &= [1/2, 3/4] \times [1/2,3/4], 
			&Q_3 &= [3/4,1] \times [1/4, 1/2].
		\end{array}
	\end{equation*}
	By construction $f(x,y) = 0$ for all $(x,y)\in Q_1 \cup Q_3$, $f(x,y)=16 \epsilon$ for almost all $(x,y)\in Q_2$. 
	We want to bound the sum of the following errors from below:
	\begin{equation*}
		\begin{array}{>{\displaystyle}r>{\displaystyle}l>{\displaystyle}r>{\displaystyle}l}
			\mathcal{E}_0 &= \int_{Q_0} |\eta(x,y)| \ddd(x,y), \quad 
			&\mathcal{E}_1 &= \int_{Q_1} |\eta(x,y)| \ddd(x,y), \\[5mm]
			\mathcal{E}_2 &= \int_{Q_2} |16\epsilon-\eta(x,y)| \ddd(x,y), \quad 
			&\mathcal{E}_3 &= \int_{Q_3} |\eta(x,y)| \ddd(x,y).
		\end{array}
	\end{equation*}
	Due to $a_{j,1},a_{j,2}>0$, the functions $\eta_i$ defined by 
	\begin{equation*}
		\eta_i(x, y) \coloneqq \lambda_i \sigma\left( a_{i,1} x+a_{i,2} y + b_i  \right) - \lambda_i \sigma\left( a_{i,1}/4+a_{i,2}/4 + b_i  \right)
	\end{equation*}
	are nondecreasing with $\eta_i(\ve{q}_0)=0$, $\ve{q}_0=(1/4,1/4)$. Therefore, $\eta_i(x,y)\geq 0$ for all $(x,y) \in Q_1 \cup Q_2 \cup Q_3$. Note also, that $\eta(x,y) = \eta(\ve{q}_0)+\sum_{i=1}^{W} \eta_i(x,y)$.
	Moreover, on any fixed line, the functions $\eta_i$ are monotone due to the monotonicity of $\sigma$. Thus, for any $\ve{q}_2 \in Q_2$ and $i=1,\dots,W$, we find
	\begin{equation*}
		\min\{\eta_i(\ve{q}_1),\, \eta_i(\ve{q}_3)\}
		\leq \eta_i(\ve{q}_2) \leq
		\max\{\eta_i(\ve{q}_1),\, \eta_i(\ve{q}_3)\}
		\leq \eta_i(\ve{q}_1)+\eta_i(\ve{q}_3),
	\end{equation*}
	where $\ve{q}_1 = \ve{q}_2 - (1/4,\,-1/4) \in Q_1$ and $\ve{q}_3 = \ve{q}_2 + (1/4,\,-1/4) \in Q_3$. 
	Summing up over $i$ and adding $\eta(\ve{q}_0)$ we find
	\begin{equation*}
		\eta(\ve{q}_2) \leq \eta(\ve{q}_1)+\eta(\ve{q}_3)-\eta(\ve{q}_0).
	\end{equation*}
	Due to the monotonicity of $\eta$, we find $\eta(\ve{q}) \leq \eta(\ve{q}_0)$ for each $\ve{q}\in Q_0$. This leads to the following relation
	\begin{equation*}
		\mathcal{E}_0 \geq \int_{Q_0} -\eta(\ve{q}) \ddd \ve{q} \geq -\eta(\ve{q}_0) |Q_0|.
	\end{equation*}
	Realizing that
	\begin{equation*}
		|f(\ve{q}_2)-\eta(\ve{q}_2)| \geq f(\ve{q}_2) - \eta(\ve{q}_1)-\eta(\ve{q}_3)+\eta(\ve{q}_0)
		\geq f(\ve{q}_2)- |\eta(\ve{q}_1)|-|\eta(\ve{q}_3)| -\mathcal{E}_0/|Q_0|,
	\end{equation*}
	and integrating over $\ve{q}_2 \in Q_2$ finally leads to
	\begin{equation*}
		\|f-\eta\|_{L^1([0,1]^2)} \geq \mathcal{E}_0 +\mathcal{E}_1 +\mathcal{E}_2 +\mathcal{E}_3 \geq 16 \epsilon |Q_2| = \epsilon,
	\end{equation*}	
	and finishes the proof.	
\end{proof}

As a direct consequence of Theorem~\ref{App:thm:shallowMonotoneApproximationFailure} we get the following result.

\begin{corollary}
	For any $\epsilon>0$ there exists a nondecreasing continuous function $f\colon [0,1]^2 \to \reals$ such that for any network $\eta\colon [0,1]^2 \to \reals$ with one hidden layer, nondecreasing activation function $\sigma$ and nonnegative weights of the form (\ref{App:shallow_monotone_net})
	it holds that $ \|f-\eta\|_{\infty}>\epsilon$. 
\end{corollary}

Further, with the same ideas from Theorem~\ref{App:thm:shallowMonotoneApproximationFailure} but a slightly more involved construction it is possible to derive the following proposition. We omit the proof.

\begin{proposition}
	There exists a nondecreasing function $f\colon [0,1]^2\to \reals$ and constants $C_1, C_2>0$ such that for any network $\eta$ with an architecture of type (A) with $L=1$ hidden layers,
	\begin{equation*}
		\|f-\eta\|_{L^1([0,1]^2)} + C_1(\EMVar(\eta)-\MVar(\eta)) \geq C_2.
	\end{equation*}
\end{proposition}

In the proof of Theorem~\ref{App:thm:shallowMonotoneApproximationFailure} we have constructed a target function of the form $f = \ind{\{x>a_1,\,y>a_2 \}}$ for some fixed $a_1,a_2$ which cannot be approximated by shallow neural networks with architectures of type (A). It is straightforward to show that networks with architectures of type (A) with two hidden layers are easily able to approximate such functions with respect to the $L^1$ norm, even in higher dimensions. Indeed, for $\epsilon<1/3$, assuming without loss of generality that $\sigma_+=1$, $\sigma_-=0$ and choosing $S>0$ as in the proof of Lemma~\ref{App:lem:aux_acti_property}, i.e., $\sigma(-S)< \epsilon$ and $\sigma(S)> 1-\epsilon$, we can define an EVNN in the following way
\begin{equation*}
	\eta(x,y;a_1,a_2, \epsilon, \delta) = 
	\sigma\left(
	-\frac{3-\epsilon}{1-3\epsilon}S
	+\frac{2}{1-3\epsilon}S\sigma(S(x-a_1)/\delta)
	+\frac{2}{1-3\epsilon}S\sigma(S(y-a_2)/\delta)\right).
\end{equation*}
The network $\eta$ satisfies $0\leq \eta\leq 1$ and $\eta(x,y)<\epsilon$ for $x<a_1-\delta$ or $y<a_2-\delta$ and $\eta(x,y)>1-\epsilon$ if both $x>a_1+\delta$ and $y>a_2+\delta$. Accordingly, choosing $\epsilon$, $S$ and $\delta$ appropriately, we are able to approximate $f$ to arbitrary accuracy in $L^1([0,1]^2)$. 
More generally, a network of the form
\begin{align} \label{App:monotone_ind_approx}
	&\nu(x_1,x_2,\dots,x_d; a_1,a_2,\dots, a_d, \epsilon, \delta) \notag
	\\
	&\qquad= 
	\sigma\left(
	-\frac{2d-1-(d-1)\epsilon}{1-(d+1)\epsilon}S
	+\sum_{i=1}^{d}
	\frac{2S}{1-(d+1)\epsilon}\sigma(S(x_i-a_i)/\delta)
	\right)
\end{align}
can be used to approximate $f = \ind{\{\vx\in\reals^d \setsep x_i>a_i \text{ for all }i=1,\dots,d\}}$. This illustrates how increasing the number of hidden layers of the network enables the network to approximate larger classes of functions efficiently. Intuitively speaking, with one hidden layer it is possible to approximate indicator functions of half spaces $\{\vx\in\reals^d \setsep \vx\cdot \vn\geq \ve{a}\}$. With two hidden layers, the indicator function of finite intersections or finite unions of half spaces can be approximated. Each additional hidden layer enables the approximation indicator functions of unions or intersections of the previous classes of sets.

This intuition can be formalized to prove that EVNNs can uniformly approximate any continuous function in $BMV$ without overestimating the M-variation. For simplicity, we restrict ourselves to functions defined on hyperrectangles $H = \prod_{i=1}^d [a_i, b_i]$.

\begin{theorem}[Universal approximation of $BMV$ functions]\label{App:thm:UniBMVapprox}
	For any continuous $BMV$ function $f\colon H \to \reals$, any $\epsilon>0$, any nondecreasing activation function $\sigma$ satisfying (\ref{App:eq:actiProperty}) and any $L\geq 3$, there exists a EVNN with an architecture of type (A) with $L$ hidden layers and activation function $\sigma$ satisfying $\|f-\eta\|_\infty \leq \epsilon$ and $\EMVar(\eta;H) < \MVar(f;H)$.
\end{theorem}

\begin{proof}
	We restrict ourselves to the case of $L=3$. The same approximation of the identity function used at the end of Lemma~\ref{App:lem:UniversalApprox} can be applied to obtain the assertion in the general case.
	Without loss of generalization, let $\epsilon<(d+1)^{-1}$.
	The case $\max_Hf+\min_Hf=0$ is trivial.
	Thus, we may assume $3\epsilon< (\max_Hf+\min_Hf)$. 
	There exist bounded monotone functions $f_1, \dots, f_{2^d} \colon \reals^d \to \reals$ with 
	\begin{equation*}
		f(\vx) = \sum_{i=1}^{2^d} f_i(\vx), \qquad\MVar(f) + \frac{2\epsilon}{13} > \sum_{i=1}^{2^d} \left(\max_H f_i - \min_H f_i \right)
		=\sum_{i=1}^{2^d} \MVar(f_i).
	\end{equation*}
	We extend $f$ to $C(\reals^d)\cap BMV(\reals^d)$ without changing the M-variation by setting
	\begin{equation*}
		f(x_1,\dots, x_d) = f(\max\{\min\{x_1, a_1\}, b_1\}, \dots,\max\{\min\{x_d, a_d\}, b_d\}). 
	\end{equation*}
	By a standard mollification argument, we find $\hat{f}$ with $\|\hat{f}-f\|_\infty \leq \epsilon/3$
	and continuous monotone functions $\hat{f}_i$, $i=1,\dots,2^d$ such that
	\begin{equation*}
		\hat{f}(\vx) = \sum_{i=1}^{2^d} \hat{f}_i(\vx), \qquad\MVar(f) + \frac{2\epsilon}{13} > \sum_{i=1}^{2^d} \left(\max_H \hat{f}_i - \min_H \hat{f}_i \right).
	\end{equation*}
	Note that $2\epsilon<(\max_Hf+\min_Hf)$.
	Clearly, it is sufficient to show that we can approximate $\hat{f}_i$ by a subnetwork $\eta_i$ such that 
	\begin{equation*}
		\sup_H|\eta_i-\hat{f}_i|\leq \epsilon \frac{\MVar(\hat{f}_i)}{2 \sum \MVar(\hat{f}_j)}, \qquad
		\MVar(\eta_i) \leq \left( 1- \frac{2\epsilon}{13 \sum \MVar(\hat{f}_j)} \right)\MVar(\hat{f}_i).
	\end{equation*}
	Therefore, now denoting $f=\hat{f}_i$, it is sufficient to show that for any continuous nondecreasing function $f$ and any $\tilde\epsilon \leq \MVar(f)/4$ there exists a network $\eta$ with $L$ hidden layers, positive weights and activation function $\sigma$ such that
	$|f-\eta|_\infty< \tilde\epsilon$ and
	$\EMVar(\eta) \leq \MVar(f) - 4\tilde\epsilon/13$. 
	Moreover, we may assume $H=[0,1]^d$, $\sigma_{+}=1$, $ \sigma_{-}=0$ and $\MVar(f)>0$.
	
	We define $K_1= \lceil3\MVar(f)/\tilde{\epsilon} \rceil$
	and $\delta_1 = \MVar(f)/K_1$. Due to $\tilde{\epsilon} \leq \MVar(f)/4$ we find $ 4\tilde{\epsilon}/13 \leq \delta_1 \leq \tilde{\epsilon}/3$.
	We denote the family of upper level sets of $f$ for $z\in[f(0,\dots,0), f(1,\dots,1)]$ by
	\begin{equation*}
		V_{z} \coloneqq \left\{(x_1,\dots,x_d) \in [0,1]^d\setsep f(x_1,\dots,x_n)>z \right\}.
	\end{equation*}
	Moreover, for $k=0,1,...,K_1-1$ we set $z_k = f(0,\dots,0)+k\delta_1$.
	Fixing $k \in \{1,\dots,K_1-1\}$, we want to approximate the function $\ind{V_{z_k}}$ by a neural network with nonnegative weights and three hidden layers. 
	Due to the continuity of $f$, there exits $\delta_2>0$ such that
	$\vx\in V_{z_{k-1}}$ for all $\vx$ with $\operatorname{dist}(\vx,V_{z_k})< 3\delta_2$.
	Moreover, 
	due to compactness there exists a finite covering of $\partial V_{z_k}$ by open rectangles $\prod_{i=1}^{d}(x^{(l)}_i-\delta_2, x^{(l)}_i+\delta_2)$, with $x^{(l)}= (x^{(l)}_1, \dots,x^{(l)}_d)\in V_{z_k}$, $l=1,...,K_2$. We choose $\delta_3\in (0,1)$ sufficiently small such that $3K_2\delta_3<(1-\delta_3)$. We define 
	\begin{equation*}
		\eta^{(k, l)}(x_1,\dots,x_d) = \nu\left(x_1,\dots,x_d;\; x^{(l)}_1-2\delta_2,\,\dots,\,x^{(l)}_d-2\delta_2, \delta_3, \delta_2 \right),
	\end{equation*}
	where $\nu$ is given by (\ref{App:monotone_ind_approx}). Similarly to before, we choose $S>0$ large enough such that
	\begin{equation*}
		\sigma(S) \geq 1-\delta_4,\qquad \sigma(-S)\leq \delta_4, \qquad \delta_4 \coloneqq \frac{1}{(K_1-2)}.
	\end{equation*}
	Next, we define
	\begin{equation*}
		\eta^{(k)}(x_1,\dots,x_d) = \sigma\left(-2S+3S/(1-\delta_3)\sum_{l=1}^{K_2} \eta^{(k, l)}(x_1,x_2,\dots,x_d)\right).
	\end{equation*}
	For $\vx \in V_{z_{k}}$ there exists a point $\hat \vx \in \partial V_{z_k}$ with $\hat \vx < \vx$. Here, we use the partial order $<$ on $\reals^n$ defined in 
	Section~2. 
	Accordingly, there exists $l\in \{1,\dots,K_2\}$ with $\hat{\vx} \in \prod_{i=1}^{d}(x^{(l)}_i-\delta_2, x^{(l)}_i+\delta_2)$ and thus $\vx\in \prod_{i=1}^{d}(x^{(l)}_i-\delta_2, \infty)$. Therefore, $x_i-(x^{(l)}_i-2 \delta_2) \geq \delta_2$ which implies $\eta^{(k,l)}(\vx)>1-\delta_3$. Consequently, $\eta^{(k)}(\vx)>1-\delta_4$. On the other hand, if $\vx \notin V_{z_{k-1}}$, i.e., if $f(\vx)\leq z_{k}-\delta_1$, then, due to the definition of $\delta_2$, we know that $\dist(\vx,V_{z_k})\geq 3 \delta_2$ and thus $x_i \leq x^{(l)}_i - 3 \delta_2$ for all $i=1,\dots,d$ and $l=1,\dots, K_2$. This implies $\eta^{(k,l)}(\vx)< \delta_3 $ for all $l=1,\dots, K_2$. Thus, by our choice for $ \delta_3$, $\eta^{(k)}(\vx) \leq \delta_4$.
	
	We claim that the network with three hidden layer and nonnegative weights defined by
	\begin{equation*}
		\eta(\vx) \coloneqq f(0,\dots,0) + \delta_1\sum_{k=1}^{K_1-1}  \eta^{(k)}(\vx)
	\end{equation*}
	satisfies $\|f-\eta\|_{\infty} \leq \tilde \epsilon$. Indeed, if $\vx \in V_{z_{k-1}}\backslash V_{z_k}$, i.e., if $(k-1)\delta_1<f(\vx)-f(\ve{0}) \leq k \delta_1$, we see 
	$\eta^{(j)}(\vx)\in[1-\delta_4,1]$ for all $j=1,\dots,k-1$, $\eta^{(k)}(\vx) \in [0,1]$ and $\eta^{(j)}(\vx) \in [0, \delta_4] $ for all $j=k+1,\dots,K_1-1$. Therefore
	\begin{equation*}
		|\eta(\vx)-f(\vx)| \leq |\eta(\vx)-f(\ve{0})-\delta_1(k-1)| + \delta_1
		\leq \delta_1(2+(K_1-2) \delta_4) \leq \tilde \epsilon.
	\end{equation*}
	Moreover,
	$
	\eta(1,\dots,1) \leq f(1,\dots,1)-\delta_1.
	$
	This completes the proof.
\end{proof}

If we drop the assumption of $f$ being continuous, we have to restrict ourselves to convergence in $L^p(H)$, $1 \leq p < \infty$.

\begin{corollary}\label{App:cor:BMVApproxinL1}
	Let $f \in BMV(H)$, $\epsilon>0$, $1\leq p< \infty$. For any nondecreasing activation function $\sigma$ satisfying (\ref{App:eq:actiProperty}) and any $L\geq 3$ there exists a EVNN with an architecture of type (A) with $L$ hidden layers and activation function $\sigma$ satisfying $\|f-\eta\|_{L^1(H)} \leq \epsilon$ and $\EMVar(\eta) < \MVar(f)$.
\end{corollary}

\begin{proof} We define the mollified extension 
	$
	E_h[f] =  \kappa_{h} * E[f],
	$
	with $E$ defined by
	\begin{align*}
		E[f;U](\vx) &=
		\sup \left( \{f(\vy) \setsep \vy\in U,\, \vy \leq \vx\} \cup \{\inf_{\vy\in U} f(\vy)\}\right)
		.
	\end{align*}
	Moreover, $\kappa$ denotes the standard mollifier and $\kappa_h(\vx) = h^{-d} \kappa(h \vx)$.
	Choosing $h$ small enough, we have $\|E_h(f)-f\|_{L^p(H)} \leq \frac{\epsilon}{2}$ and $\MVar(E_h[f], H) \leq \MVar(f)$. Using Theorem~\ref{App:thm:UniBMVapprox} we derive the conclusion of the corollary.
\end{proof}

\section{A limitation of the estimated M-variation}\label{App:sec:shockLimitation}

The estimated M-variation may predict shocks even if there are none. If $d>1$, this is not necessarily the fault of the chosen optimizing strategy, but can be caused by a limitation of the estimated M-variation itself. The following example illustrates this behavior.

\begin{example}\label{App:ex:EMVar}
	We consider the domain $U=[0,1]^2 \subset \reals^2$ and define $f=\ind{A}$ with
	\begin{equation*}
		A = \Big\{\vx=(x_1,x_2) \in U \setsep |x_1-1/2| \leq 1/6,\;|x_2-1/6| \leq 1/6\Big\}.
	\end{equation*}
	Clearly, $f\in BMV(U)$. Given a family of nondecreasing functions $f_i\colon \reals^2 \to \reals$ and $\vs_i\in\{-1,1\}^2$ for all $i\in I$ for a finite set $I$ with
	\begin{equation*}
		f(\vx) = \sum_{i\in I} f_i(\vx \odot \vs_i) \qquad \text{for all }\vx \in U,
	\end{equation*}
	we define $I_+\coloneqq \{i\in I\setsep (\vs_i)_1=1 \}$. Defining 
	\begin{equation*}
		\ve{p}_1 = (1/6, 1/6), \quad \ve{p}_2 = (1/2, 1/6), \quad \ve{p}_3 = (5/6, 1/6)
	\end{equation*}
	we can find an open connected set $B\subset U$ with $B\cap A = \emptyset$ and $\ve{p}_1,\ve{p}_2 \in B$, see Figure~\ref{App:fig:exEMVar}. 
	In particular, if we use $f_i$ to estimate the M-variation of $f$ on $B$, we get
	\begin{align*}
		\EMVar(f; B)
		&= \sup_{\vx\in B} \sum_i f_i(\vx\odot \vs_i)
		- \inf_{\vx\in B} \sum_i f_i(\vx\odot \vs_i)
		\geq 
		\sum_{i\in I_+} f_i (\vs_i \odot \ve{p}_3) - f_i (\vs_i \odot \ve{p}_1)\\
		&\geq
		\sum_{i\in I_+} f_i (\vs_i \odot \ve{p}_2) - f_i (\vs_i \odot \ve{p}_1)
		\geq f(\ve{p}_2)-f(\ve{p}_1)
		= 1 > 0=\MVar(f;B).
	\end{align*}
	
\end{example}
\begin{figure}
	\centering
	\includegraphics{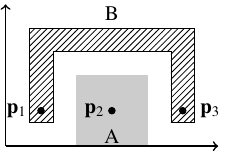}
	\caption{Illustration of Example~\ref{App:ex:EMVar}.}
	\label{App:fig:exEMVar}
\end{figure}

\section{Application to ordinary differential equations}\label{App:sec:ApplicationToODEs}

In this section we are concerned with the existence of explicitly computable loss functions $(\Loss_n)_{n\in\naturals}$ such that any sequence of networks $(\eta_n)_{n\in\naturals}$ with architectures of type (A) or (A') satisfying 
\begin{equation}\label{App:eq:convCon}
	\Loss_n(\eta_n) \xrightarrow{n\to\infty} 0
\end{equation}
converges to the solution of
\begin{equation}\label{App:eq:ode_pointwise}
	\ve{y}'(t) = \ve{F}(t,\,\ve{y}(t)) \quad \text{ for all }t\in [0,T], \quad \ve{y}(0) = \ve{y}_0.
\end{equation}
Moreover, we show that networks satisfying (\ref{App:eq:convCon}) exist.
Here, we assume $\ve{F}\colon \reals\times \reals^N \to \reals^N$ to be continuous and locally Lipschitz continuous with respect to the second argument in the sense that for each $(t,x) \in \reals\times \reals^N$ there exists a neighborhood $U\ni (t,x)$ such that $\ve{F}$ is Lipschitz continuous in the second argument on $U$. We want to solve the following problem using neural networks with architectures of type (A) as trial functions: Given $\ve{y}_0\in \reals^N$ and $T>0$ approximate the solution $\ve{y}\colon[0,T] \to \reals^N$ of (\ref{App:eq:ode_pointwise}).

We propose to base our loss function on the integral formulation of (\ref{App:eq:ode_pointwise}),
\begin{equation}\label{App:eq:ode_int_form}
	\ve{y}(t) = \int_{[0,t]} \ve{F}(s, \ve{y}(s)) \ddd s + \ve{y}_0 \qquad \text{for all }t\in[0,T],
\end{equation}
to find an approximation $\veta$ of $\ve{y}$. 
Due to the regularity of $\ve{F}$, we can integrate the composition $t \mapsto \ve{F}(t, \veta(t))$ numerically. Noting that $\ve{F}$ can be approximated uniformly on compact sets by mappings of the form
\begin{equation*}
	\ve{F}(s,\ve{z}) \approx \sum_{i=1}^{K_\tau} \ind{[(i-1)\tau,\, i\tau ]} \ve{F}(i\tau,\ve{z}), \qquad \ve{F}(i\tau,\cdot) \text{ Lipschitz continuous},
\end{equation*}
we can use the integration algorithm proposed in Section~2.1.
This naturally leads to the following loss function:

\begin{align}\label{App:eq:ODELossStrong}
	\Loss(\veta;\,\mathcal{E}_{\text{var}},\mathcal{E}_{\text{step}},\mathcal{E}_{\text{int}}) 
	&= \Loss^\text{ode}(\veta;\,\mathcal{E}_{\text{var}},\mathcal{E}_{\text{step}},\mathcal{E}_{\text{int}}) + \Loss^\text{IC}(\veta) ,\\ \notag
	\Loss^\text{ode}(\veta;\,\mathcal{E}_{\text{var}},\mathcal{E}_{\text{step}},\mathcal{E}_{\text{int}}) 
	&=
	\sum_{k=1}^{K_t} \bigg|\veta(t_{k}) - \veta(t_{k-1}) - I_{\mathcal{E}_{\text{int}}}[\ve{F}(\cdot,\veta(\cdot)); t_{k-1},t_k]  \bigg|,\\ \notag
	\Loss^\text{IC}(\veta) 
	&=
	|\veta(0)-\vy_0|,
\end{align}
where $I_{\mathcal{E}_{\text{int}}}[\ve{F}(\cdot,\veta(\cdot)); a,b]$ denotes an approximation of the integral satisfying
\begin{equation}\label{App:eq:approximateIntegration}
	\left|I_{\mathcal{E}_{\text{int}}}[\ve{F}(\cdot,\veta(\cdot)); a,b] - \int_{a}^{b} \ve{F}(s,\veta(s)) \ddd s \right| \leq \mathcal{E}_{\text{int}}.
\end{equation}
Moreover, the points $0=t_0<t_1<...<t_{K_t}=T$ with $t_k=t_k(\veta, \mathcal{E}_{\text{var}},\mathcal{E}_{\text{step}})$ are assumed to satisfies the following two properties:
\begin{align}
	\label{App:eq:partitionCondition1}
	\EMVar(\veta,[t_{k-1}, t_k]) \leq \mathcal{E}_{\text{var}} \qquad \text{for all }k=1,...,K_t,\\
	\label{App:eq:partitionCondition2}
	\int_{t_{k-1}}^{t_k} \left|\ve{F}(s, \veta(s))\right| \ddd s \leq \mathcal{E}_\text{step}
	\qquad\text{ or }\qquad t_{k}-t_{k-1} \leq \mathcal{E}_\text{step}.
\end{align}

We can find a partition of $[0,T]$ satisfying (\ref{App:eq:partitionCondition1}) and (\ref{App:eq:partitionCondition2}) by iterative refinement. One example of an adaptive algorithm to compute $t_0, \dots,t_{K_t}$ with the above properties is given by Algorithm~\ref{App:algo:adaptivePartitionODE}. Clearly, for any Lipschitz continuous function on a closed interval, Algorithm~\ref{App:algo:adaptivePartitionODE} terminates in finite time.

\begin{algorithm}
	\caption{Algorithm to compute an adaptive partition of $[0,T]$.}
	\label{App:algo:adaptivePartitionODE}
	\begin{algorithmic}
		\Require network net, Boolean $\mathrm{flag}\in \{\mathrm{True}, \mathrm{False}\}$, tolerances $\mathcal{E}_\text{var}>0$,  $\mathcal{E}_\text{step}>0$ and end time $T>0$
		
		\If{$\mathrm{flag}=\mathrm{True}$}
		\State 
		$\text{partition} \gets (0,T)$
		\Else
		\State $N \gets \left\lceil \frac{\mathcal{E}_\text{step}}{T} \right\rceil$
		\State $\text{partition} \gets \left(0,\frac{1}{N}T, \frac{2}{N}T,\dots,T \right)$
		\EndIf
		\State variation $\gets$ calculateMVarOnPartition(net, partition)
		\If{flag = True}
		\State 
		integrals $\gets$ integrateOnPartition($s \mapsto \ve{F}(s,\text{net}(s))$, partition, tolerance=$\mathcal{E}_\text{step}/2$)
		\EndIf
		\While{$\max(\ \text{variation}\ ) > \mathcal{E}_\text{var} \text{ or } ( \text{flag and } \max(\ \text{integrals}\ ) > \mathcal{E}_\text{step}/2)$}
		\State indexTuple $\gets$  $\{i \setsep \text{variation}[i] > \mathcal{E}_\text{var} \text{ or } ( \text{flag and } \text{integrals}[i] > \mathcal{E}_\text{step}/2 \}$ ordered monotonously decreasing
		\For{ index $i$ in indexTuple}
		\State
		partition $\gets$ concatenate(partition[all elements with index $\leq i$],
		\\
		\hspace{6cm}
		$(\text{partition}[i]+\text{partition}[i+1])/2$,
		\\
		\hspace{6cm}
		partition[all elements with index $>i$])
		\EndFor
		\EndWhile
		\State
		\Return partition
	\end{algorithmic}
\end{algorithm}

The following theorem provides an error bound for
the proposed loss of a network with an architecture of type (A) or (A').
\begin{theorem}\label{App:thm:convergenceODE}
	Let $\vy\in C^1([0,T];\reals^N)$ be a solution of the ODE~(\ref{App:eq:ode_pointwise}) and $\veta=\veta(\mathcal{E}_{\text{var}},\mathcal{E}_{\text{step}}, \mathcal{E}_{\text{int}}, \mathcal{E}_{\text{loss}})$ be a network with an architecture of type (A) satisfying
	$
	\Loss(\veta;\,\mathcal{E}_{\text{var}},\mathcal{E}_{\text{step}},\mathcal{E}_{\text{int}}) \leq \mathcal{E}_{\text{loss}}.
	$
	Then there exist two constants $C,C'>0$ depending only on $T$, $\ve{F}$ and $\vy$ such that 
	\begin{equation*}
		\big|\veta(t)-\vy(t)\big| \leq C \left(\mathcal{E}_{\text{loss}}
		+K_t\mathcal{E}_{\text{int}}
		+\mathcal{E}_{\text{var}}
		+\mathcal{E}_{\text{step}}\right)
	\end{equation*}
	for all $(\mathcal{E}_{\text{loss}}
	+K_t\mathcal{E}_{\text{int}}
	+\mathcal{E}_{\text{var}}
	+\mathcal{E}_{\text{step}})< C'$ .
\end{theorem}
\begin{proof}
	The proof is a straightforward application of Gronwall's inequality.
\end{proof}
Without Lipschitz continuity of $\ve{F}$, the Gronwall argument used in the proof of Theorem~\ref{App:thm:convergenceODE} might not be applicable and furthermore uniqueness might fail. Nonetheless,
since $\Loss^\text{ode}$ is based on the integral formulation of the ODE, we are not restricted to settings with solutions in $C^1([0,T])$ satisfying (\ref{App:eq:ode_pointwise}) pointwise. As long as a suitable numerical integration $I[\cdot]$ is available, it is possible to work with less regular $\ve{F}$ with solutions $\vy$ with possibly discontinuous derivative. For this setting, we propose the following modification of the loss function
\begin{align} \label{App:eq:ODELossWeak}
	\Loss_n(\veta) &= 
	\Loss^{\text{ode}}\left(\veta, \epsilon_n, \epsilon_n, \frac{\epsilon_n}{K_t}\right) 
	+
	\Loss^\text{IC}(\veta) 
	+ \Loss^{\text{c}}(\veta,R),\\ \notag
	\Loss^{\text{c}}(\veta,R)&=
	\max\{0,\, R-\EMVar(\eta) \}.
\end{align}
Here, $(\epsilon_n)_{n\in\naturals}$ is a sequence converging to $0$ and $R$ is a positive constant. The addition of $\Loss^{\text{c}}$ for a fixed $R$ enforces precompactness in $L^p$, $1\leq p<\infty$, for any sequence with bounded $\Loss_n$. Of course, $R$ has to be chosen sufficiently large such that there exists a solution $\vy$ with $\MVar(\vy,[0,T]) \leq R$, otherwise no sequence of networks $(\veta_n)_{n\in\naturals}$ exists which satisfies (\ref{App:eq:convCon}).
On the other hand, we have the following result for every sequence satisfying (\ref{App:eq:convCon}).
\begin{theorem}\label{App:thm:convergenceODEWeak}
	Let $F\colon [0,T]\times \reals^N \to \reals^N$ be continuous in the second argument, Borel measurable and map bounded sets to bounded sets. Moreover, assume that there exists a quadrature $I_{1/n}$ for $F$ satisfying (\ref{App:eq:approximateIntegration}) for each network $\eta$ with an architecture of type (A). Every sequence of networks $(\eta_n)_{n\in\naturals}$ with architectures of type (A) satisfying 
	\begin{equation*}
		\Loss_{n}(\veta_n) \xrightarrow{n\to\infty}0
	\end{equation*}
	has a subsequence which converges in $L^1$ to a solution $\vy\in W^{1,1}([0,T];\reals^N)$ satisfying (\ref{App:eq:ode_int_form}). In particular, if there exists only one such solution on $[0,T]$ then
	\begin{equation*}
		\veta_n \xrightarrow{n\to\infty} \vy \qquad \text{in }L^1([0,T]).
	\end{equation*}
\end{theorem}
\begin{proof}
	Note that $|\veta_n(0)-\vy_0| \leq \Loss_n(\veta_n)$.
	Therefore, the boundedness of $\Loss^c(\veta_n,\,R)$ implies precompactness of the sequence $(\veta_n)_{n\in\naturals}$ in $L^p([0,T])$, $p\in [1,\infty)$. We may choose a subsequence $(\veta_{n_k})_{k\in\naturals}$ such that
	$\veta_{n_k} \to \vy\in L^\infty([0,T])$ in $L^1([0,T])$ and pointwise a.e. In particular, $s \mapsto \ve{F}(s,\veta_{n_k}(s))$ converges in $L^1([0,T],\reals^N)$ to the mapping $s \mapsto \ve{F}(s,\vy(s))$.
	Fixing $t\in [0,T]$ with
	$
	\veta_{n_k}(t) \xrightarrow{k\to\infty} \vy(t), 
	$
	choosing $l$ such that $t_{l} \leq t\leq t_{l+1}$ and defining $\mathcal{E}_\text{int}^k = \epsilon_{n_k}/K_t(\veta_{n_k})$ we derive
	\begin{align*}
		&\left|\vy(t)-\vy_0 - \int_{[0,t]} \ve{F}(s,\vy(s)) \ddd s\right| 
		\xleftarrow{k\to\infty}
		\left|\veta_{n_k}(t)-\vy_0 - \int_{[0,t]} \ve{F}(s,\veta_{n_k}(s)) \ddd s\right|
		\\
		&\quad\leq 
		\big|\veta_{n_k}(t_l)-\veta_{n_k}(t)\big|
		+
		\Big|
		\veta_{n_k}(t_l)-\veta_{n_k}(0)-
		I_{\mathcal{E}_\text{int}^k}[\ve{F}(\cdot,\veta_{n_k}(\cdot)); 0,t_l]
		\Big|
		+
		|\veta_{n_k}(0)-\vy_0|
		\\
		&\quad\phantom{\leq }
		+
		\left|I_{\mathcal{E}_{\text{int}}}[\ve{F}(\cdot,\veta_{n_k}(\cdot)); 0,t_l] - \int_{0}^{t_l} \ve{F}(s,\veta_{n_k}(s)) \ddd s \right|
		+
		\left| \int_{t_l}^{t} \ve{F}(s,\veta_{n_k}(s)) \ddd s \right|
		\xrightarrow{k\to\infty}0
		.
	\end{align*}
	In particular, by possible changing $\vy$ on a negligible set we derive that (\ref{App:eq:ode_int_form}) is satisfied.
\end{proof}

In both settings, assuming the existence of a solution $\vy\colon[0,T]\to \reals^N$ of (\ref{App:eq:ode_pointwise}) or (\ref{App:eq:ode_int_form}), respectively, there exist sequences $(\eta_n)_{n\in\naturals}$ of networks with architectures of type (A) with sufficiently small loss. For $\Loss_n$, we have to assume that $R\geq\MVar(\vy,[0,T])$.
\begin{theorem}[Existence of minimizers]\label{App:thm:ode_existence_of_minimizers}
	For any $\epsilon>0$, $\vy\in W^{1,1}([0,T]; \reals^N) $ and any Borel-measurable $\ve{F}\colon[0,T]\times \reals^N \to \reals^N$ mapping bounded sets to bounded sets (neither necessarily locally Lipschitz continuous with respect to the second argument nor necessarily continuous) satisfying (\ref{App:eq:ode_int_form}) there exists a network $\veta = \veta(\epsilon)$ with an architecture of type (A) satisfying
	\begin{enumerate}[(i)]
		\item $\esssup_{t\in[0,T]}|\veta(\epsilon)(t)-\vy(t)|_\infty \leq \epsilon$,
		\label{App:enum:ode_existence_1}
		\item $\EMVar(\eta_i(\epsilon);\,[0,T]) \leq \Var((y_i; \,[0,T])$ for all $i=1,\dots,N$,
		\label{App:enum:ode_existence_2}
		\item $|\EMVar(\eta_i(\epsilon);\,[a,b])-\Var(y_i;\,[a,b])|  \leq \epsilon $ 
		for all  $a,b\in[0,T]$ and all $i=1,\dots,N$,
		\label{App:enum:ode_existence_3}
		\item there exists $C$ independent of $\epsilon$ such that $\sup_{t\in[0,T]}|\ve{F}(t,\veta(\epsilon)(t))|_\infty\leq C$.
		\label{App:enum:ode_existence_4}
	\end{enumerate}
\end{theorem}
We start with some remarks on the consequences of Theorem~\ref{App:thm:ode_existence_of_minimizers}.
\begin{remark}
	Let $\delta\in(0, \mathcal{E}_\text{step}/\max\{1, 2C\})$ be such that 
	$\int_A |\vy'| \ddd t < \frac{\mathcal{E}_\text{var}}{2}$
	for all measurable $A\subset [0,T]$ with $|A|<\delta$.
	Together, properties (\ref{App:enum:ode_existence_3}) and (\ref{App:enum:ode_existence_4}) imply that any partition $0=t_0<t_1<\dots<t_{K_t}=T$ with 
	\begin{equation}\label{App:eq:boundOnMaximumStep}
		\max\{t_k-t_{k-1}\setsep k=1,\dots,K_t\} < \delta
	\end{equation}
	satisfies (\ref{App:eq:partitionCondition1}) and (\ref{App:eq:partitionCondition2}) for $\veta_\epsilon$ for all $\epsilon \leq \mathcal{E}_\text{var}/2$.
	In particular, Algorithm~\ref{App:algo:adaptivePartitionODE} terminates for $\veta(\epsilon)$, $\epsilon \leq \mathcal{E}_\text{var}/2 $, as soon as (\ref{App:eq:boundOnMaximumStep}) is satisfied. Moreover, the resulting partition has $K_t \leq \lceil 2T/\delta\rceil$.
\end{remark}
\begin{remark}
	If $\ve{F}$ is continuous in the second argument, property (\ref{App:enum:ode_existence_1}) guarantees the existence of $\veta(\epsilon)$ satisfying (\ref{App:enum:ode_existence_1})-(\ref{App:enum:ode_existence_4}) and
	\begin{equation*}
		\int_{[0,T]} \big|\ve{F}(t,\vy(t)) - \ve{F}(t,\veta(t))\big| \ddd t \leq \epsilon.
	\end{equation*} 
\end{remark}
\begin{remark}
	As a consequence of the previous two remarks, we have the following result: For every 
	$r, \mathcal{E}_{\text{var}},
	\mathcal{E}_{\text{step}},
	\mathcal{E}_{\text{int}}>0$, 
	$\ve{F}$ continuous in the second argument and $t_0,\dots,t_{K_t}$ chosen with Algorithm~\ref{App:algo:adaptivePartitionODE}, there exists $\epsilon_0$ such that for all $\epsilon<\epsilon_0$,
	\begin{equation*}
		\Loss^\text{ode}(\veta(\epsilon);\,\mathcal{E}_{\text{var}},\mathcal{E}_{\text{step}},\mathcal{E}_{\text{int}}) < r+K_t\mathcal{E}_{\text{int}}.
	\end{equation*}
\end{remark}
\begin{proof}[Proof of Theorem~\ref{App:thm:ode_existence_of_minimizers}]
	We fix $\vy$ to be the continuous representative with respect to equality almost everywhere.
	Denote by $\vy^+$, $\vy^-\colon [0,T]\to \reals^N$ componentwise nondecreasing functions with 
	\begin{align*}
		\vy(t) &= \vy_+(t)-\vy_-(t) \quad &\text{for all }t\in [0,T]
		\qquad \text{and}
		\\
		\pVar(\vy;a,b)&= \vy_+(b)+\vy_-(b)-\vy_+(a)-\vy_-(a)
		\quad &\text{for all }0\leq a \leq b \leq T.
	\end{align*}
	According to Theorem~\ref{App:thm:UniBMVapprox} there exists a network $\veta = \veta_+ + \veta_-$ with an architecture of type (A) and subnetworks $\veta_+$ and $\veta_-$ satisfying 
	\begin{equation*}
		\|\pm \veta_\pm-\vy^\pm\|< \epsilon/4,
		\qquad
		\MVar(\veta_\pm,\; [0,T]) \leq \MVar(\vy_\pm, \; [0,T]).
	\end{equation*}
	Thus, $\veta$ satisfies (\ref{App:enum:ode_existence_1})-(\ref{App:enum:ode_existence_3}). 
	Without loss of generality, we may assume $\epsilon<1$. We find $\sup_{t\in[0,T]}|\veta(t)|\leq 2+\sup_{t\in[0,T]}|\vy(t)|$ due to properties (\ref{App:enum:ode_existence_1}) and (\ref{App:enum:ode_existence_3}) imply. Thus, there exists $C>0$ independent of $\epsilon$ such that $\sup_{t\in[0,T]}|\ve{F}(t,\veta(t))| < C$.
\end{proof}

\section{Numerical test cases for ordinary differential equations}\label{App:sec:numExperiments_ODE}

We illustrate our theoretical findings from Section~\ref{App:sec:ApplicationToODEs} on two simple numerical test cases.
Theorems~\ref{App:thm:convergenceODE}, \ref{App:thm:convergenceODEWeak}, \ref{App:thm:ode_existence_of_minimizers} motivate the minimization of the loss functions defined in (\ref{App:eq:ODELossStrong}) and (\ref{App:eq:ODELossWeak}) over the set of neural networks. To avoid the optimization algorithm getting stuck in poor local minima, in practice it is preferred to optimize modified versions of the loss function. We used the following common modifications:
\begin{itemize}
	\item \textbf{Enforcement of the boundary condition in $\Loss^\text{ode}$}. Analogously to \cite{cvPINNs}, we replace $\eta(0)$ by $u_0$ in $\Loss^\text{ode}$ instead of enforcing $\Loss^\text{IC}$. Note that Theorem~\ref{App:thm:convergenceODE} remains valid with this change.  
	
	\item \textbf{Additional regularization penalty terms.} We add the penalty term
	$\Loss^\text{reg}(\eta) = \EMVar(\eta;\,[0,T])$
	to the loss function with a weight of $10^{-5}$.
	
	\item \textbf{Causal Training}. As proposed in \cite{wang2022respecting}, we replace $\Loss^\text{ode}$ by
	\begin{align*}
		\Loss^\text{ode}(\veta;\,\mathcal{E}_{\text{var}},\mathcal{E}_{\text{step}},\mathcal{E}_{\text{int}}) 
		&=
		\sum_{k=1}^{M} w_k\bigg|\veta(t_{k}) - \veta(t_{k-1}) - I_{\mathcal{E}_{\text{int}}}[\ve{F}(\cdot,\veta(\cdot)); t_{k-1},t_k]  \bigg|,\\
		w_k &= \exp\left(-\zeta\sum_{l=1}^{k} \bigg|\veta(t_{l}) - \veta(t_{l-1}) - I_{\mathcal{E}_{\text{int}}}[\ve{F}(\cdot,\veta(\cdot)); t_{l-1},t_l]  \bigg|\right)
		. 
	\end{align*}
	We set $\zeta=10$ in our experiments.
\end{itemize} 

The loss functions in (\ref{App:eq:ODELossStrong}) and (\ref{App:eq:ODELossWeak}) lead to $\ell^1$ like minimization. This makes gradient based optimization challenging. The gradient of sums of errors is simply the sum of the gradients of the errors:
$
E = E_1+E_2 \Rightarrow \nabla E = \nabla E_1 + \nabla E_2.
$
This corresponds to giving all errors the same weight when determining the update direction. 
With a $\ell^2$-type error, which is typically used in machine learning, the gradients are weighted with the corresponding error magnitude.
\begin{equation*}
	\tilde E = E_1^2+ E_2^2 \Rightarrow \nabla \tilde E = 2E_1\nabla E_1 + 2E_2\nabla E_2
\end{equation*}
The gradients of the largest errors have the strongest contribution to the learning direction. Heuristically, this reduces the risk of small errors having a large influence on the update direction. 
Consequently, we applied the following final modification to $\Loss^\text{ode}$ during the gradient based optimization.
\begin{align*}
	\Loss^\text{ode}(\veta;\,\mathcal{E}_{\text{var}},\mathcal{E}_{\text{step}},\mathcal{E}_{\text{int}})
	&=
	\sum_{k=1}^{M} \left(z_k+\frac{1}{2}z_k^2\right) 
	\\
	z_k 
	&\coloneqq 
	\bigg|\veta(t_{k}) - \veta(t_{k-1}) - I_{\mathcal{E}_{\text{int}}}[\ve{F}(\cdot,\veta(\cdot)); t_{k-1},t_k]  \bigg|
\end{align*}
By construction, our modification is greater than or equal to the unmodified version of $\Loss^\text{ode}$. In the modified $\Loss^\text{ode}$ intervals $[t_{k-1},t_k]$ with large errors have a larger impact on the parameter update. However, Table~\ref{App:tab:ODEerrors} indicates that this change barely impacted the performance.

For optimization of the loss function, we use the popular Adam optimizer \cite{kingma2017adam} with a learning rate of 0.01. We do not employ weight decay as we already penalize the variation of the network for regularization. Moreover, the values of the hyperparameters listed in Table~\ref{App:tab:numExperiments_hyperparametersettings} were used.

\begin{table}
	\centering
	\begin{tabular}{|c|c|c|c|c|c|}
		\hline 
		LSUV gain & depth	& width &activation & weight modifier  	& epochs \\
		\hline 
		1		& 3		& 20	&$\tanh$	& $[x\mapsto |x|]$	& 8192 \\
		\hline 
	\end{tabular} 
	\caption{Default hyperparameter settings.}
	\label{App:tab:numExperiments_hyperparametersettings}
\end{table}

We use network checkpointing during the training: We track the network parameters which lead to the smallest loss and load these parameters after the training. Here, we use the original unmodified loss function from Section~\ref{App:sec:ApplicationToODEs} as Theorem~\ref{App:thm:convergenceODE} provides error bounds depending on this loss. These parameters mostly do not coincide with the network parameters after the final epoch indicating an unstable training procedure.

To compute $(t_k)_{k=1}^M$, a slight modification of Algorithm~\ref{App:algo:adaptivePartitionODE} is employed. We still ensure that the resulting partition satisfies both (\ref{App:eq:partitionCondition1}) and (\ref{App:eq:partitionCondition2}). Instead of satisfying (\ref{App:eq:partitionCondition2}) directly, we ensure
$C \max\{t_k-t_{k-1}\setsep k=1,\dots,M\} \leq \mathcal{E}_\text{step},$
with $C\geq \sup\{F(\ve{z}) \setsep z\in \veta([0,T]) \}$. We can compute such a bound easily since $\veta$ is an EVNN using $|\veta(x)-\veta(y)| \leq |\EMVar(\veta; x,y)|$ and the Lipschitz continuity of $\ve{F}$.

To test the reliability of the algorithm, we used 100 runs with independently initialized networks for each test case and each setting. We present errors to the respective analytical reference solutions in Table~\ref{App:tab:ODEerrors}.

Motivated by Theorems~\ref{App:thm:convergenceODE} and \ref{App:thm:ode_existence_of_minimizers} we set
$\mathcal{E}_\text{step} = \mathcal{E}$,
$\mathcal{E}_\text{var} = \mathcal{E}$
and
$\mathcal{E}_\text{int} = \mathcal{E}/M$
for all experiments for the values of $\mathcal{E}>0$ specified in Table~\ref{App:tab:ODEerrors}.

\subsection{Exponential decay}\label{App:sec:odeExperiment_expDecay}
As a first simple benchmark to illustrate our theoretical results, we consider the following ODE
\begin{equation*}
	y'(x) = -y(x), \qquad y(0) = 1
\end{equation*}
on $[0,1]$ with solution $y(x) = \exp(-x)$. Table~\ref{App:tab:ODEerrors} and Figures~\ref{App:fig:exponential_decay}
illustrate the robustness of the algorithm with respect to initialization. For this test case we experimented with different values for $\mathcal{E}$. Surprisingly, even for relatively large values, the network achieves good accuracy even for the worst initializations, see Table~\ref{App:tab:ODEerrors}.

\subsection{Poisson equation in 1D with discontinuous coefficient}\label{App:sec:odeExperiment_1dPoisson}

As a simple toy problem with a non-smooth solution,
we consider the following ODE\footnote{The authors thank Boris Kaus (Mainz) for proposing this example to test the EVNNs.}
\begin{equation*}
	\frac{d}{dx} \left(k \frac{d T}{dx}\right) = 0, \qquad k(x) = \begin{cases}
		1 \qquad &\text{for }x<0.5,\\
		10 \qquad &\text{for }x\geq 0.5.
	\end{cases}
\end{equation*}
The boundary condition for the solution $T\colon [0,1] \to \reals$ are $T(0)=1$ and $T(1)=100$. We are interested in the unique weak solution $T\in H^1([0,1])$, i.e.,
\begin{equation*}
	\int_{[0,1]} \varphi'(x) k(x) T'(x) = 0 \qquad \text{for all }\varphi \in H^1_0([0,1]),
\end{equation*}
which is given by
\begin{equation*}
	T(x) = \frac{T(1)-T(0)}{\int_{0}^{1} \frac{1}{k(s)} \ddd s} \int_{0}^{x} \frac{1}{k(s)} \ddd s  + T(0).
\end{equation*}
Note that the solution is continuous but not in $C^1([0,1])$ and that there is a shock in the derivative of the solution for the specified $k$. Since our proposed algorithm works with the integral formulation of the ODE, the algorithm is still applicable to this problem, after reformulating the ODE as a system of ODEs solving for $(T, T'k)$:
\begin{equation*}
	\frac{d}{dx} T = \frac{(T'k)}{k}, \qquad \frac{d}{dx} (T'k) = 0.
\end{equation*}
Besides the non-smoothness of the first derivative, another challenge for this test case is the large boundary value $T(1)=100$. Nonetheless, the network performance is good and again stable with respect to initialization, see Figure~\ref{App:fig:1D_poisson} and Table~\ref{App:tab:ODEerrors}. We remark that we did not use the euclidean norm in $\reals^2$ but used the $\ell^1$ norm instead in $\Loss^\text{ode}$. Similarly, we used the sum of M-variation of the components of the network for $\Loss^\text{reg}$.

Since we are considering a boundary value problem for $T$, we did not employ causal training for this test case.

\begin{table}
	\centering
	\begin{tabular}{|c|c|c|c|c|c|c|}
		\hline       experiment& $\mathcal{E}$ & remark&  mean error & error std. & max error & min error
		\\
		\hline
		\ref{App:sec:odeExperiment_expDecay} & 0.1 &
		
		&2.10e-03&5.46e-04&3.42e-03&1.07e-03
		\\\hline
		
		\ref{App:sec:odeExperiment_expDecay} & 0.01 &
		
		&7.04e-04&2.64e-04&1.48e-03&2.56e-04
		\\\hline
		
		\ref{App:sec:odeExperiment_expDecay} & 0.001 &
		
		&7.45e-04&2.81e-04&1.52e-03&2.50e-04
		\\\hline
		
		\ref{App:sec:odeExperiment_expDecay} &0.01 & no checkp. 
		
		&1.06e-02&6.46e-03&3.45e-02&1.33e-03
		\\\hline
		
		\ref{App:sec:odeExperiment_expDecay} & 0.01 & MAE loss
		
		&7.11e-04&2.73e-04&1.38e-03&3.02e-04
		\\\hline
		
		\ref{App:sec:odeExperiment_1dPoisson} & 0.1 & no cau. tr.
		
		&8.24e-03&5.33e-04&9.42e-03&7.23e-03
		\\\hline
	\end{tabular}
	\caption{Relative errors for ODEs with respect to the uniform norm. $\mathcal{E}$ is the chosen value for $M\mathcal{E}_\text{int}$, $\mathcal{E}_\text{var}$ and $\mathcal{E}_\text{step}$. In the third to last row, no checkpointing was used. In the second to last row, the mean absolute error based loss as in Section~\ref{App:sec:ApplicationToODEs} was employed instead of our modification of the $\ell^1$ Loss. In the last row, no causal training was used.}
	\label{App:tab:ODEerrors}
\end{table}

\begin{figure}\centering
	\captionbox{Learning exponential decay by enforcing the ODE. The reference solution as well as the average and range at each point in time of 100 network runs are plotted. We show the network solution with $\mathcal{E}=0.1$.\label{App:fig:exponential_decay}}{\includegraphics[width=0.4\linewidth]{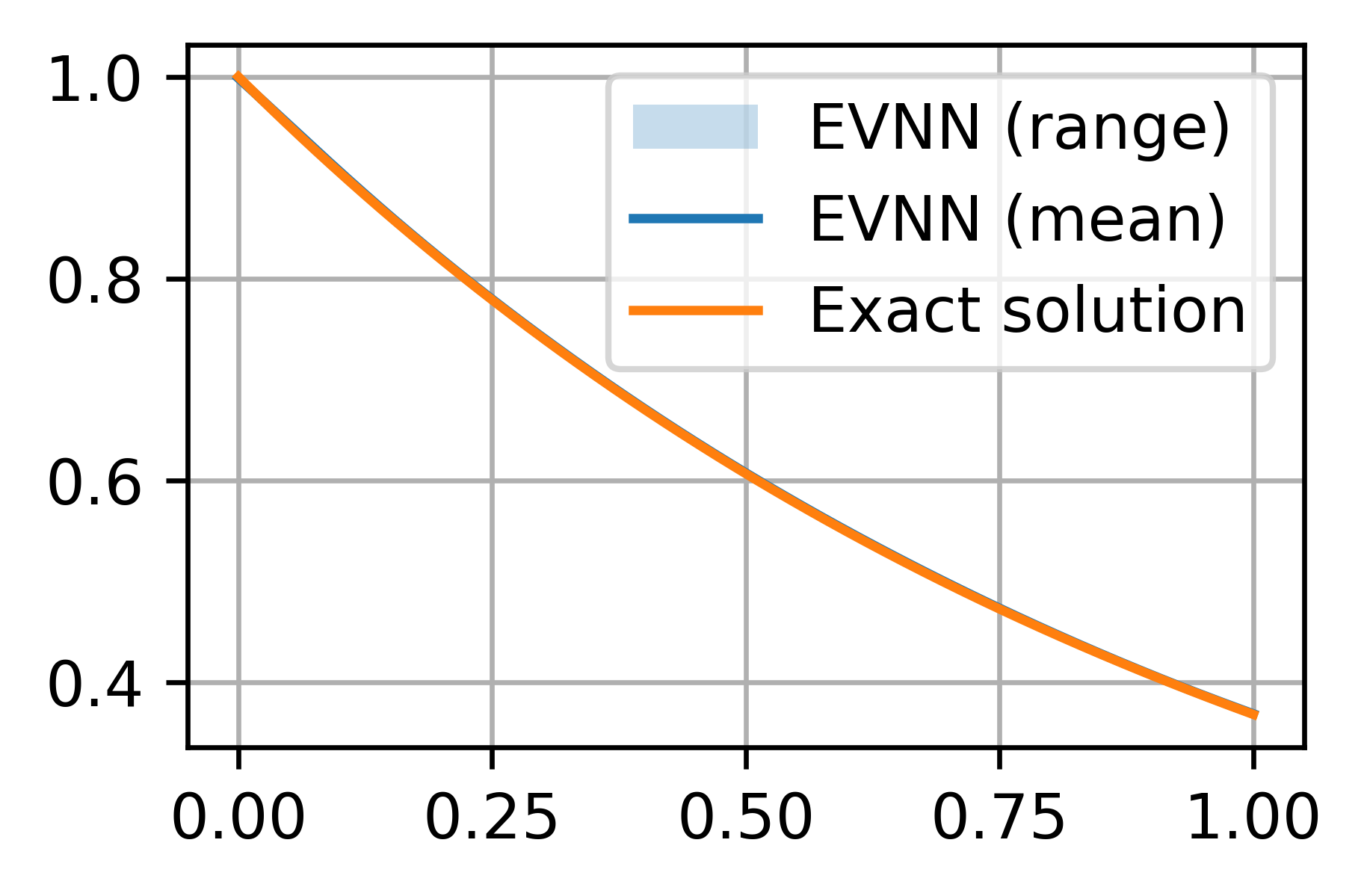}}
	\hspace{0.1\textwidth}
	\captionbox{Poisson problem in one dimension with non-smooth solution. The reference solution as well as the average and range at each point in time of 100 network runs are plotted.\label{App:fig:1D_poisson}}{\includegraphics[width=0.4\linewidth]{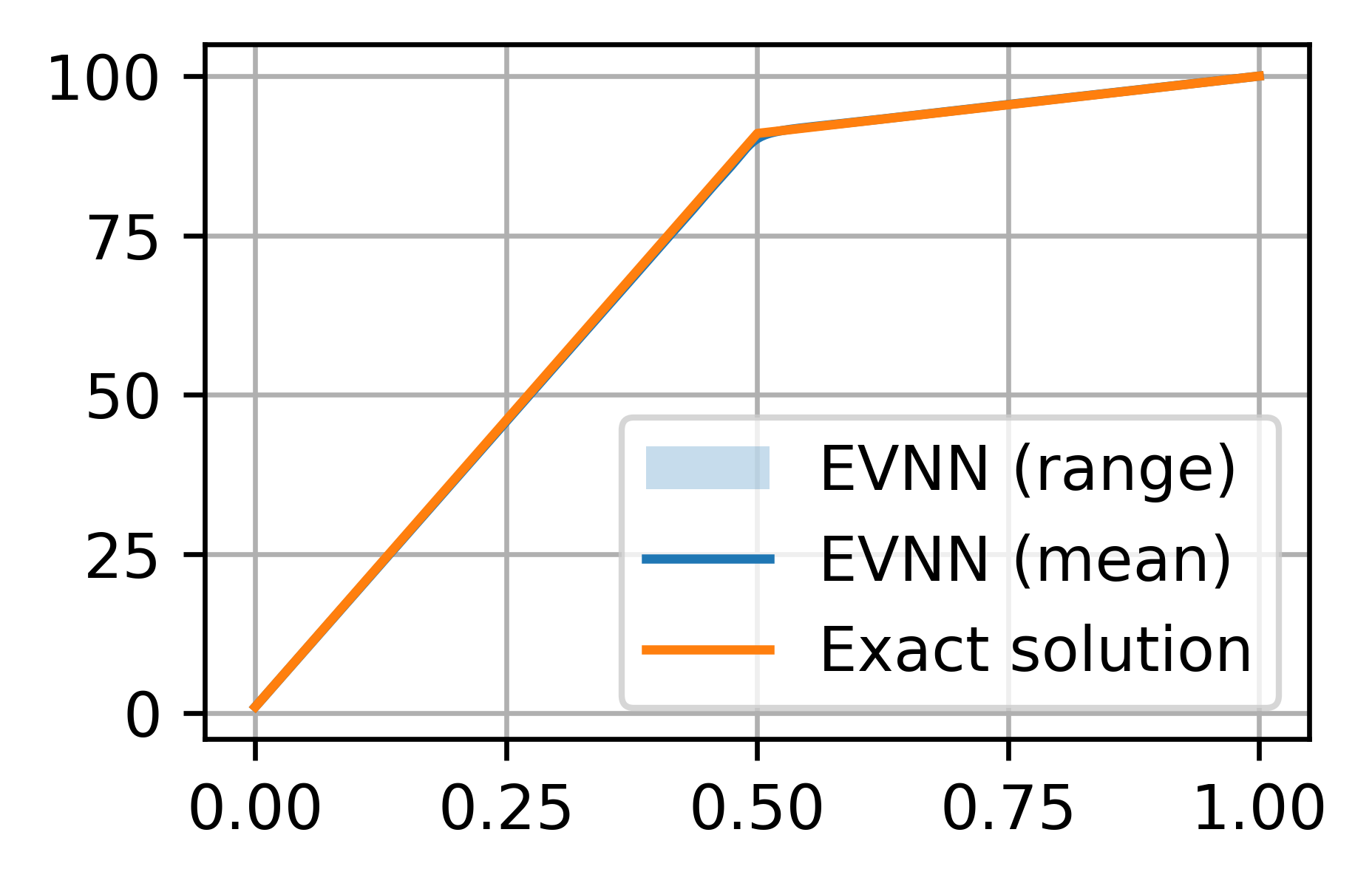}}
\end{figure}

\section{Additional material regarding the numerical simulations for the Burgers equation}\label{App:sec:AdditionalNumericalBurgersRersults}

We present additional numerical results for the test cases we investigated. Figures~\ref{App:fig:2DStandingShock}, \ref{App:fig:2DMovingShock} and \ref{App:fig:2DRarefaction} show the results of the first trained network for the test cases investigated in Sections 5.1.1 and 5.1.2. 
Figure~\ref{App:fig:stdPlots} shows the empirical standard deviation with respect to network initialization for the Burgers test cases.
\begin{figure}
	\begin{subfigure}[b]{0.49\textwidth}
		\centering
		\includegraphics[width=0.8\linewidth]{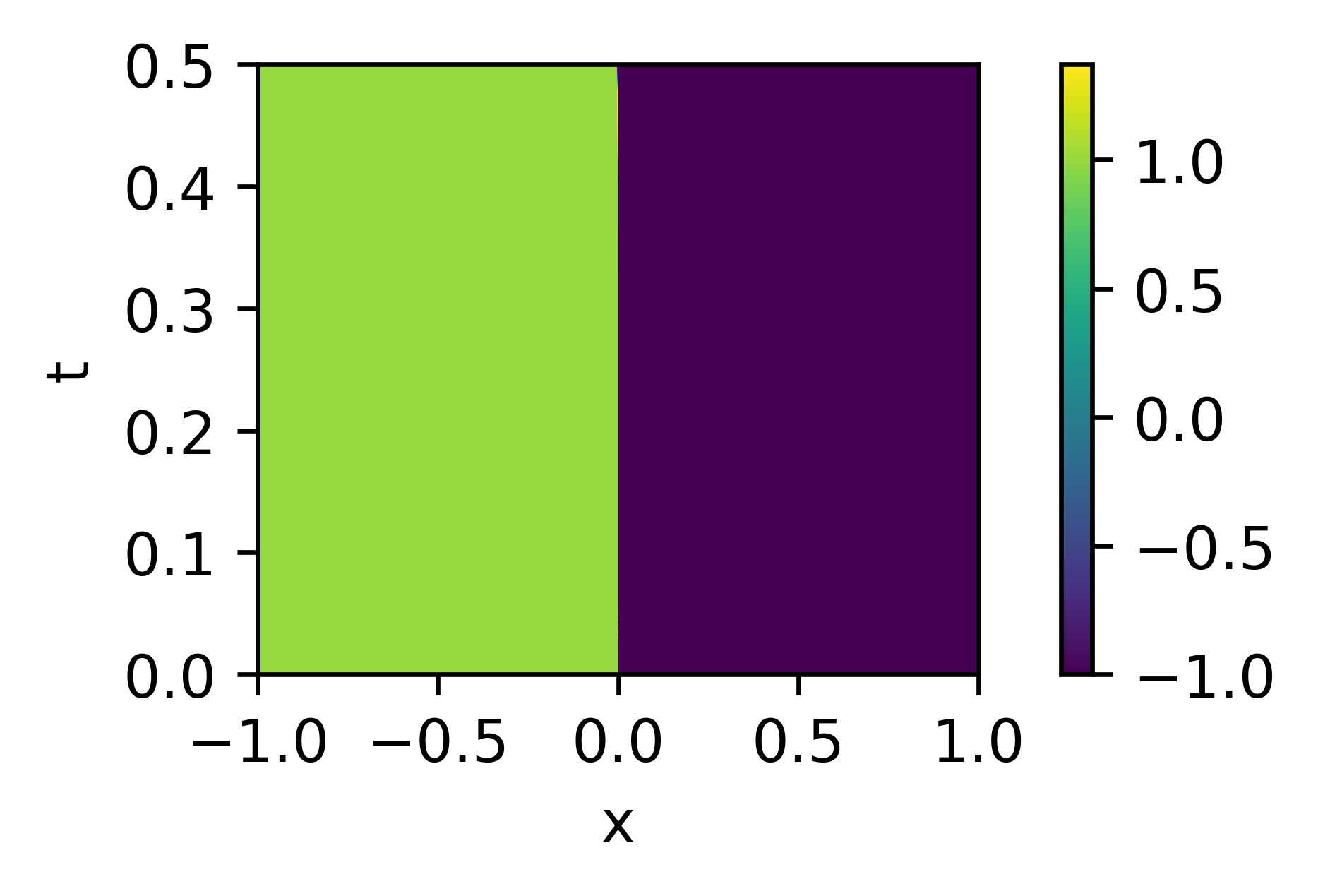}
		\caption{Network solution.}
		\label{App:fig:2DStandingShock_net}
	\end{subfigure}
	\begin{subfigure}[b]{0.49\textwidth}
		\centering
		\includegraphics[width=0.8\linewidth]{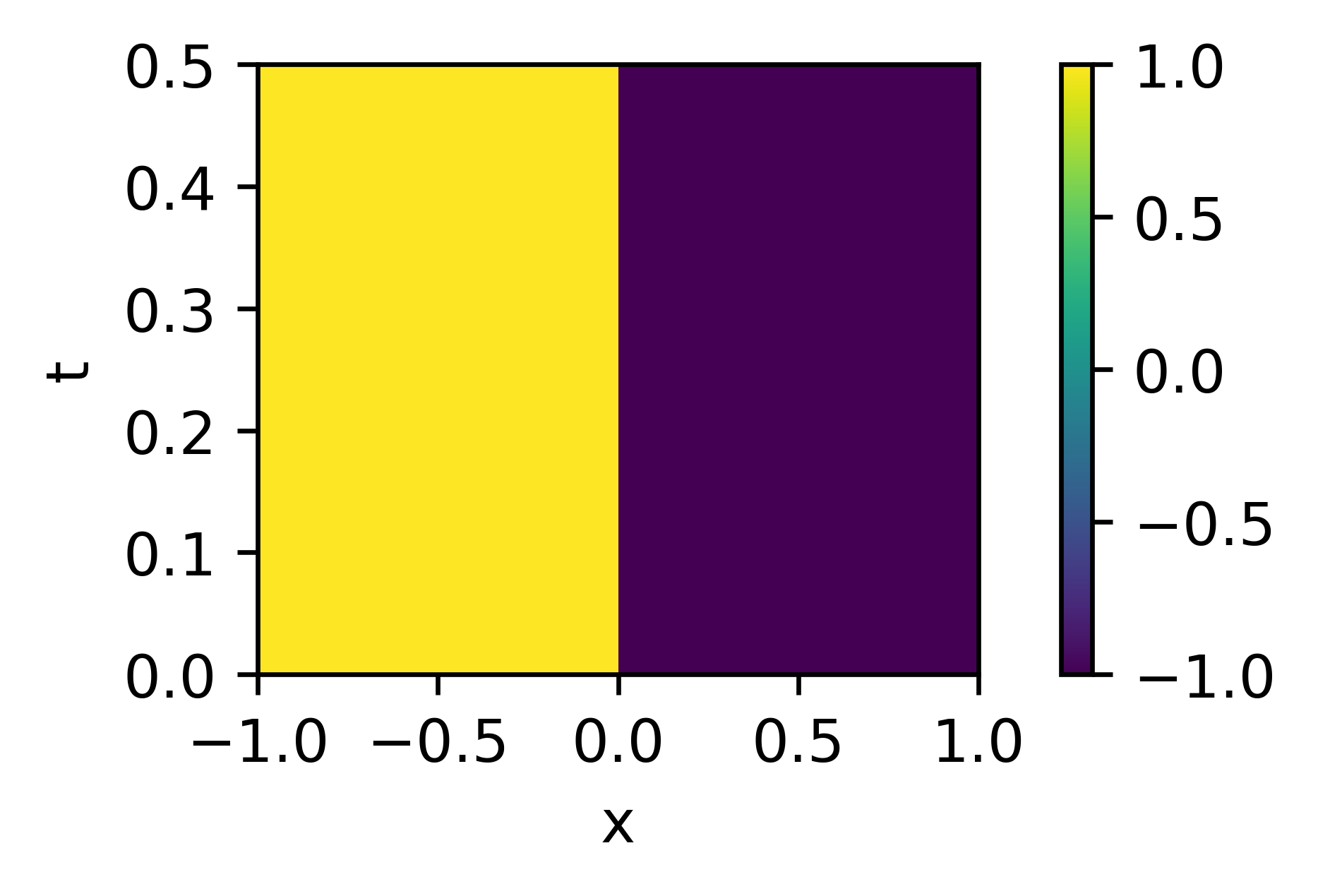}
		\caption{Reference solution.}
		\label{App:fig:2DStandingShock_ref}
	\end{subfigure}
	
	\begin{subfigure}[b]{0.49\textwidth}
		\centering
		\includegraphics[width=0.8\linewidth]{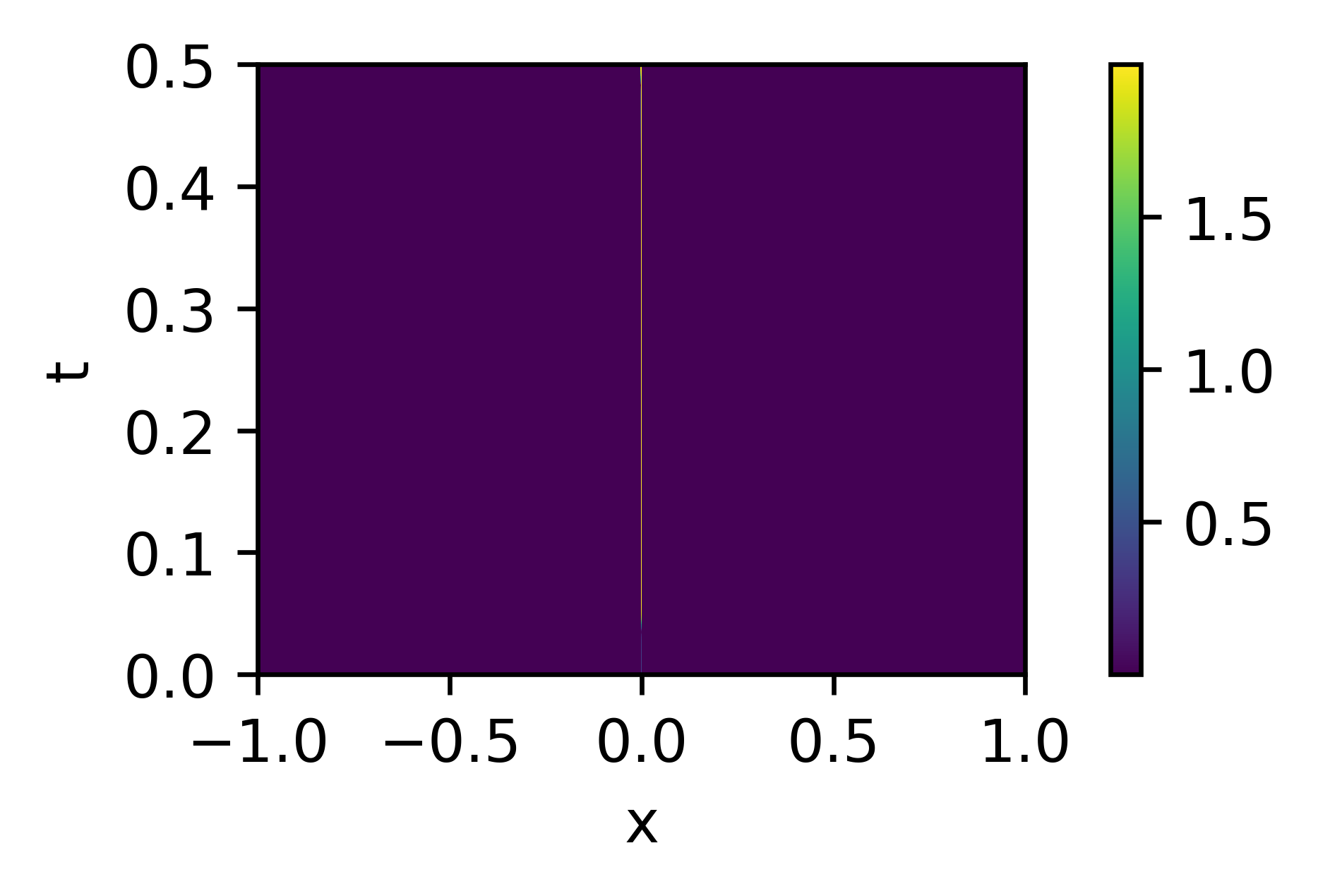}
		\caption{Error.}
		\label{App:fig:2DStandingShock_err}
	\end{subfigure}
	\begin{subfigure}[b]{0.49\textwidth}
		\centering
		\includegraphics[width=0.8\linewidth]{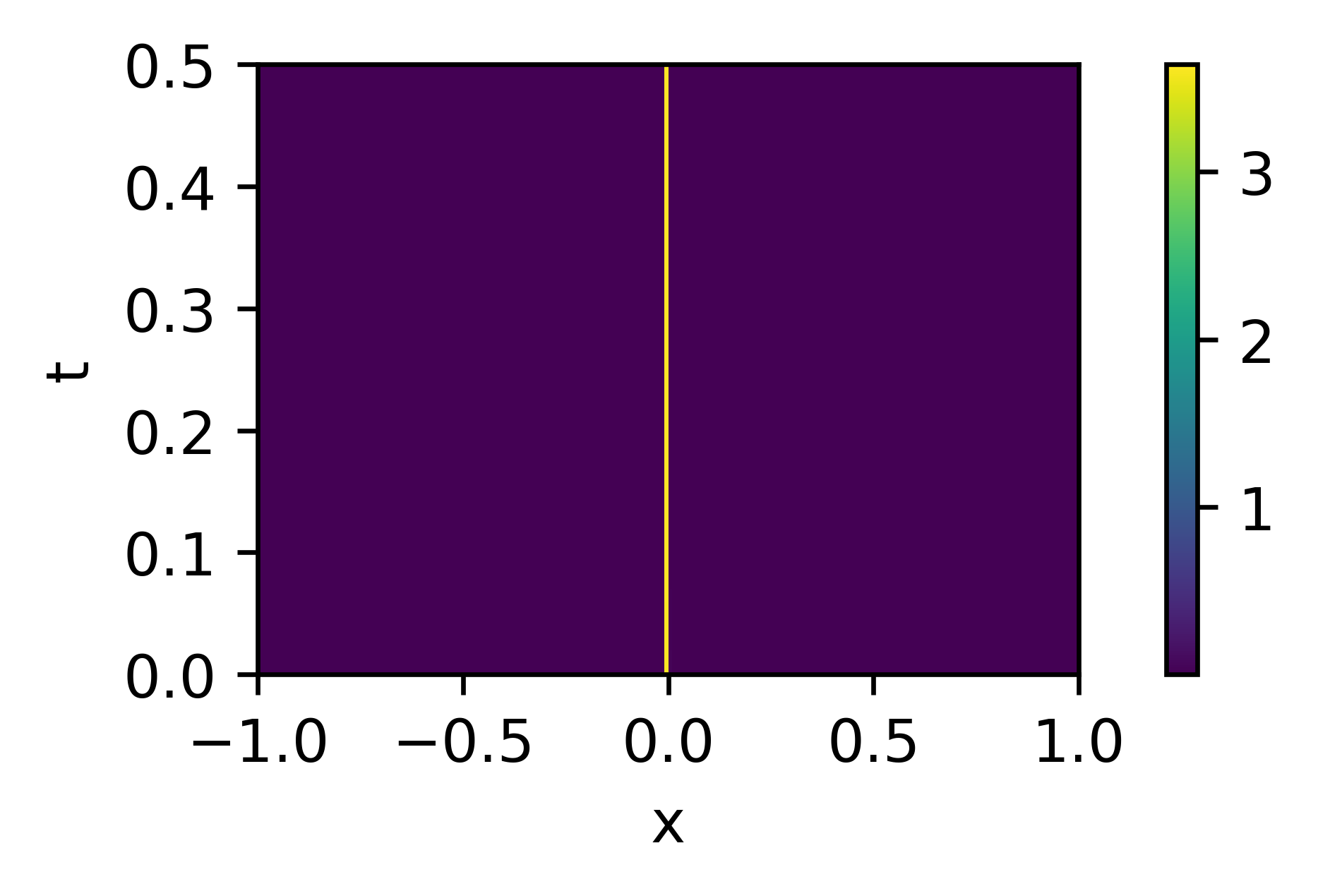}
		\caption{Estimated M-variation on a $100\times100$ grid.}
		\label{App:fig:2DStandingShock_var}
	\end{subfigure}
	\caption{The first network initialization trained for the standing shock test case.}
	\label{App:fig:2DStandingShock}
\end{figure}

\begin{figure}
	\begin{subfigure}[b]{0.49\textwidth}
		\centering
		\includegraphics[width=0.8\linewidth]{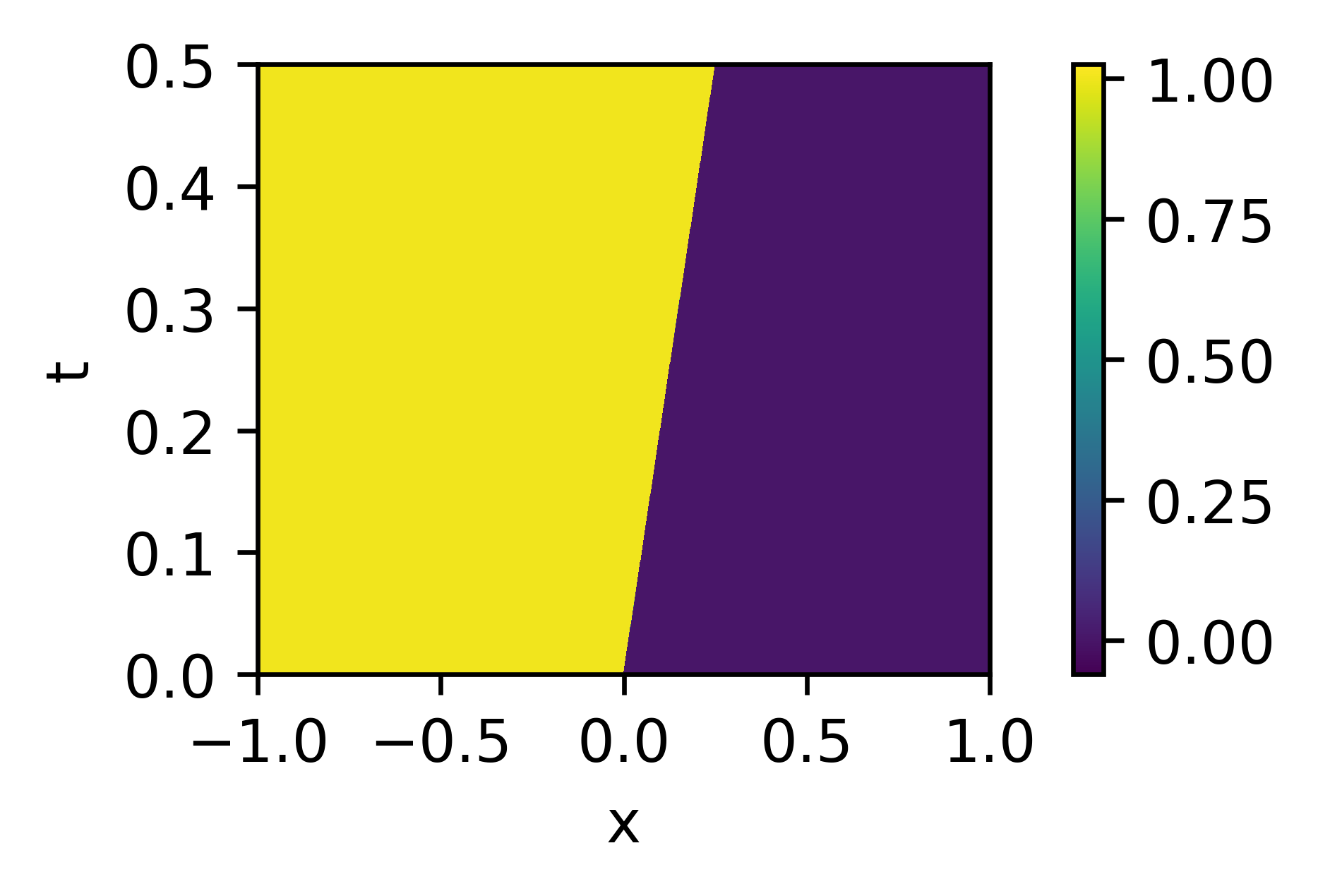}
		\caption{Network solution.}
		\label{App:fig:2DMovingShock_net}
	\end{subfigure}
	\begin{subfigure}[b]{0.49\textwidth}
		\centering
		\includegraphics[width=0.8\linewidth]{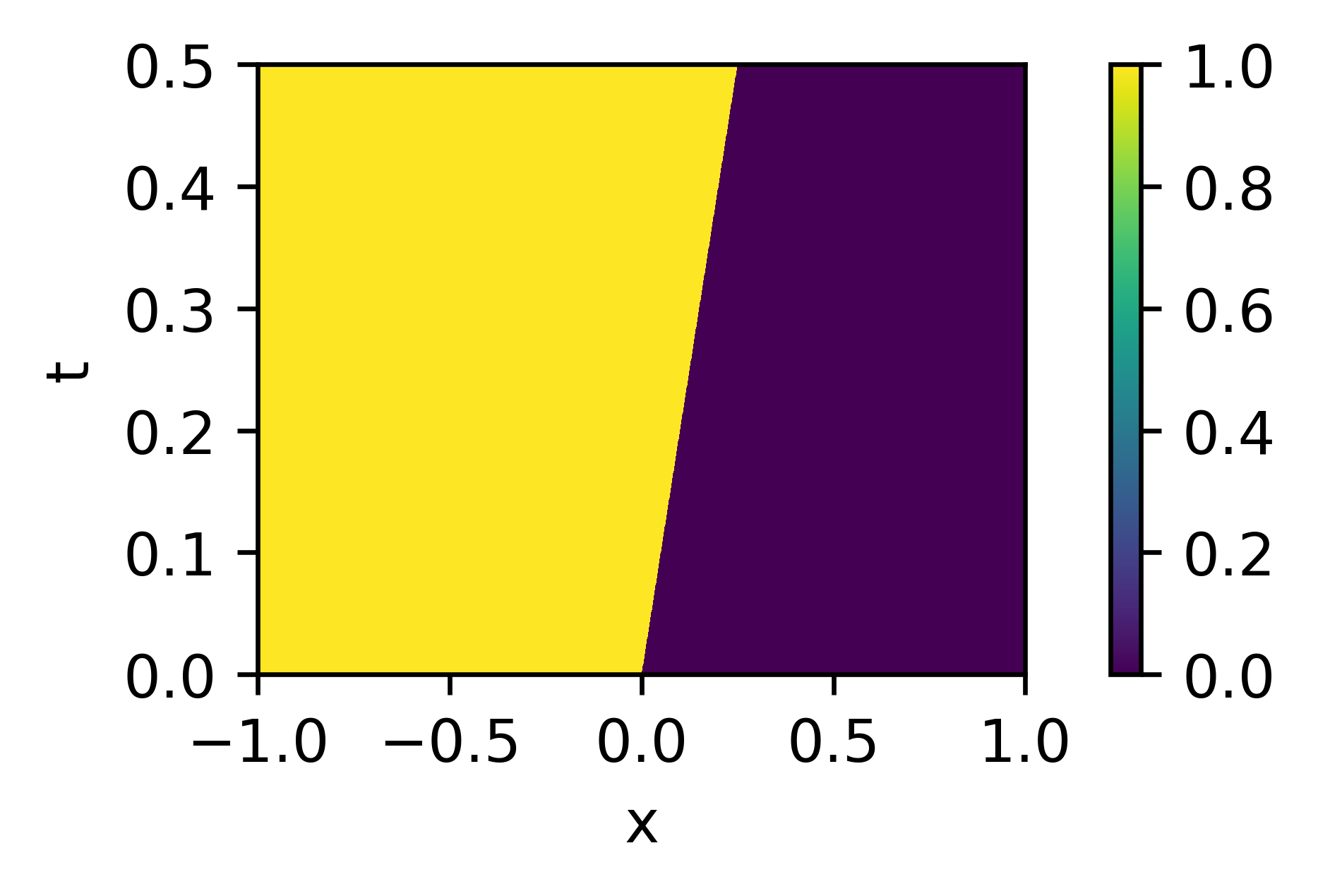}
		\caption{Reference solution.}
		\label{App:fig:2DMovingShock_ref}
	\end{subfigure}
	
	\begin{subfigure}[b]{0.49\textwidth}
		\centering
		\includegraphics[width=0.8\linewidth]{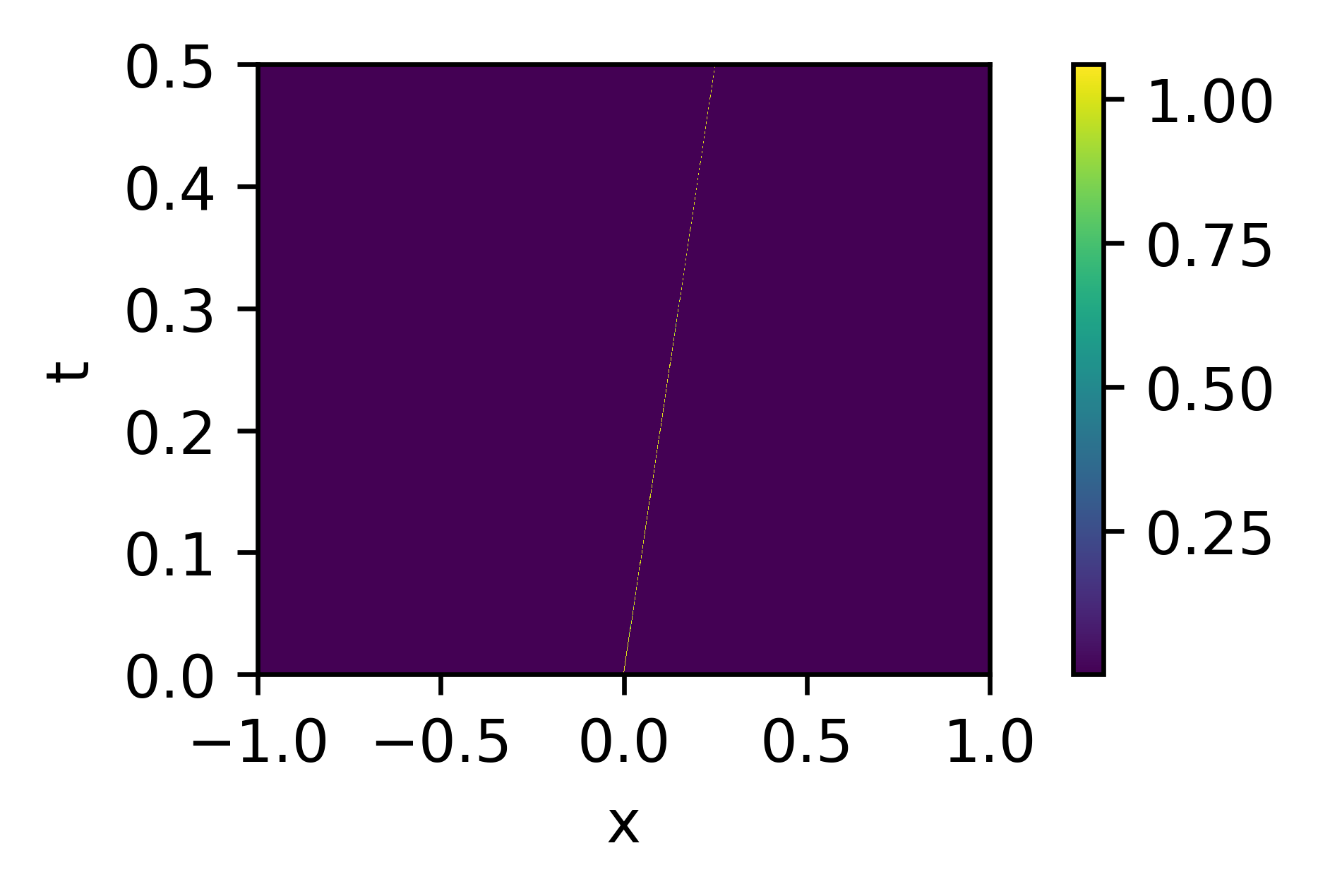}
		\caption{Error.}
		\label{App:fig:2DMovingShock_err}
	\end{subfigure}
	\begin{subfigure}[b]{0.49\textwidth}
		\centering
		\includegraphics[width=0.8\linewidth]{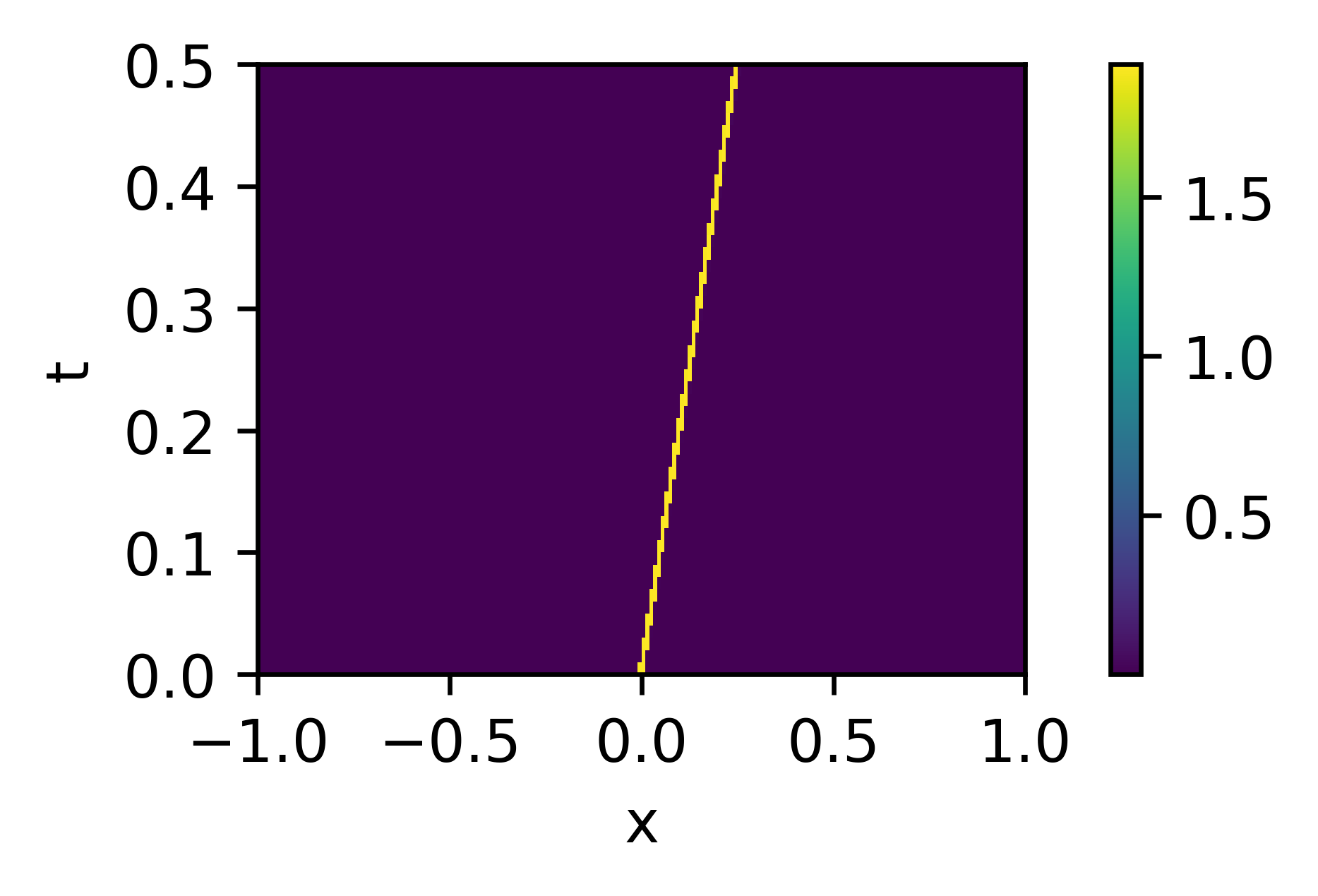}
		\caption{Estimated M-variation on a $100\times100$ grid.}
		\label{App:fig:2DMovingShock_var}
	\end{subfigure}
	\caption{The first network initialization trained for the moving shock test case.}
	\label{App:fig:2DMovingShock}
\end{figure}	

\begin{figure}
	\begin{subfigure}[b]{0.49\textwidth}
		\centering
		\includegraphics[width=0.8\linewidth]{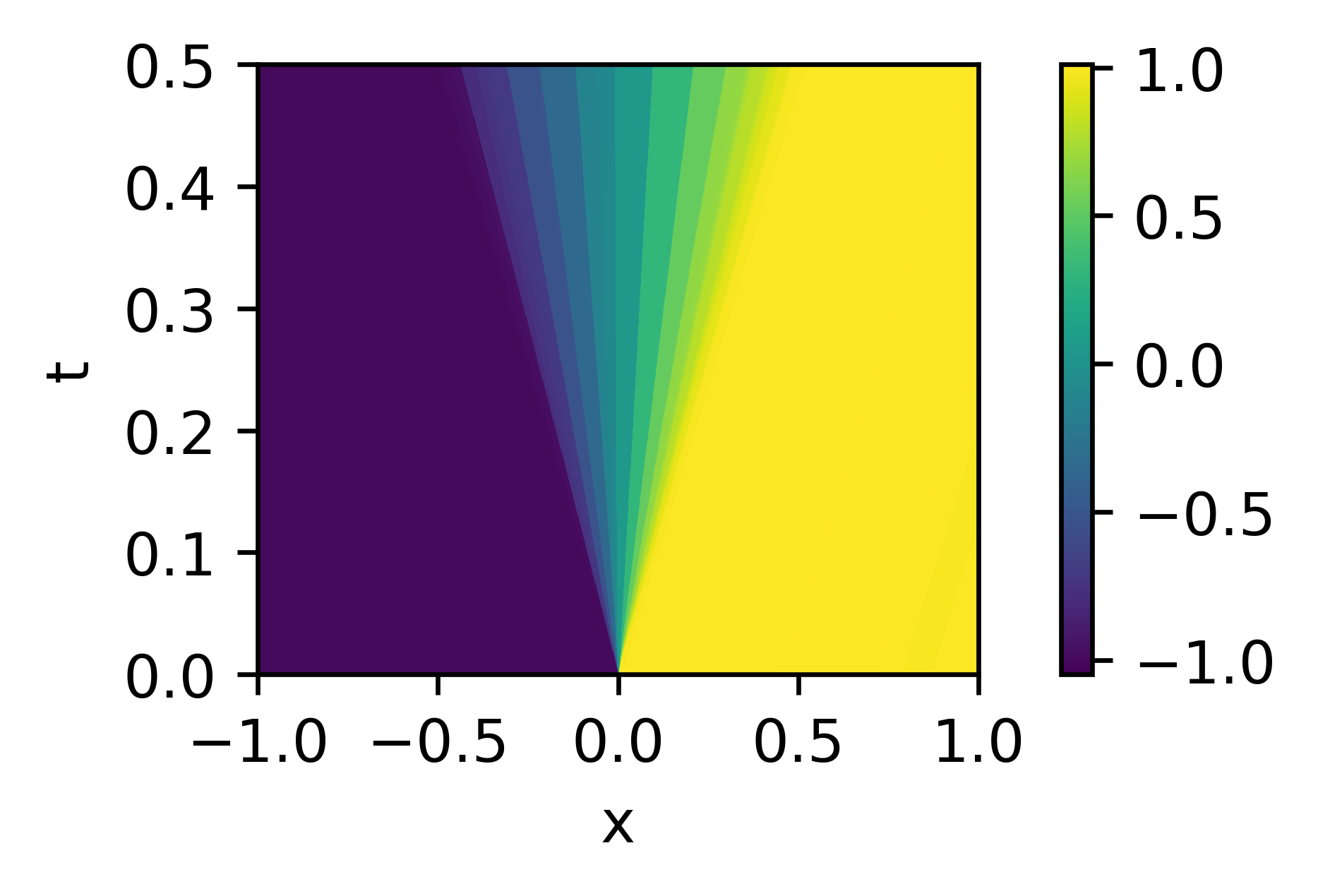}
		\caption{Network solution.}
		\label{App:fig:2DRarefaction_net}
	\end{subfigure}
	\begin{subfigure}[b]{0.5\textwidth}
		\centering
		\includegraphics[width=0.8\linewidth]{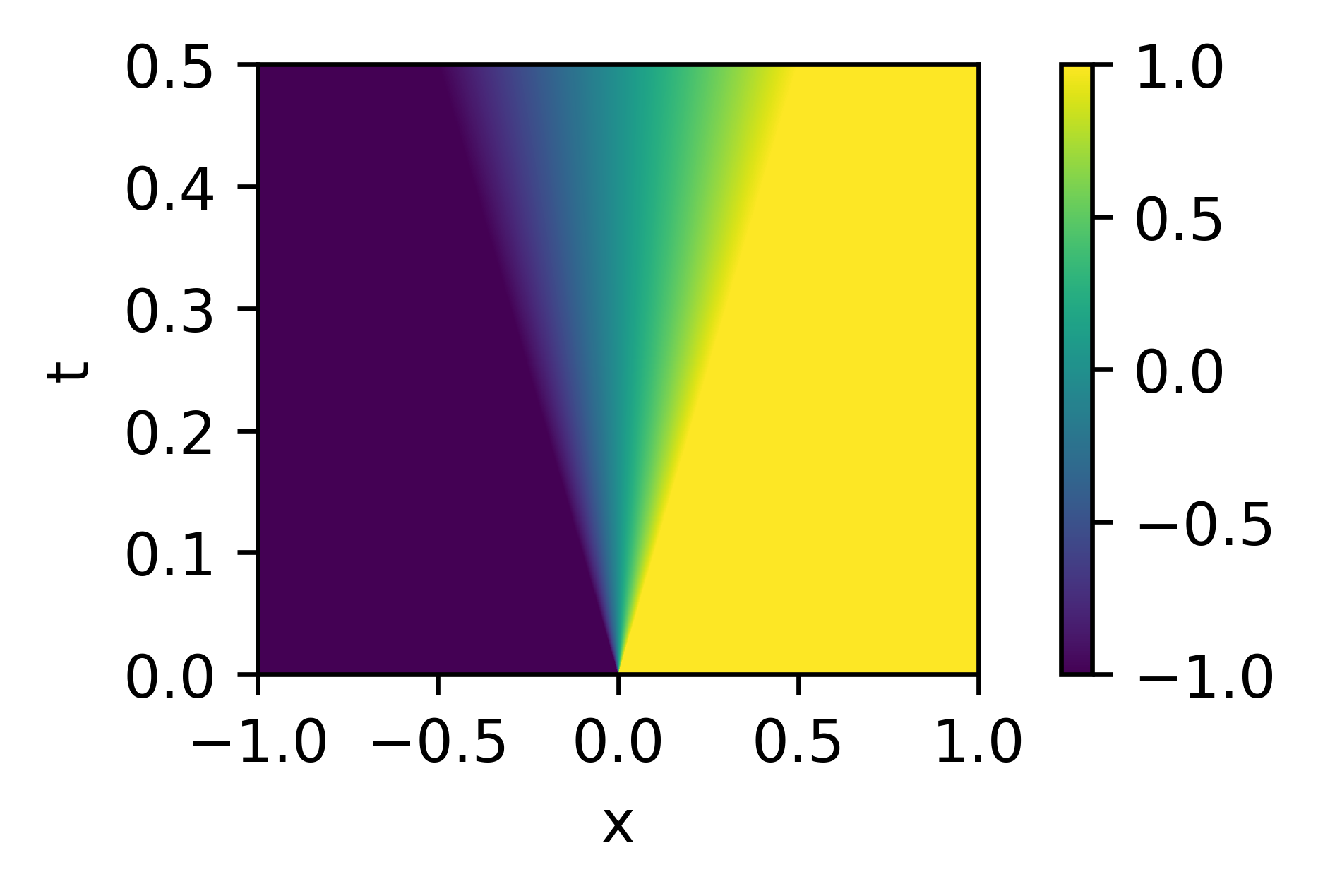}
		\caption{Reference solution.}
		\label{App:fig:2DRarefaction_ref}
	\end{subfigure}
	
	\begin{subfigure}[b]{0.49\textwidth}
		\centering
		\includegraphics[width=0.8\linewidth]{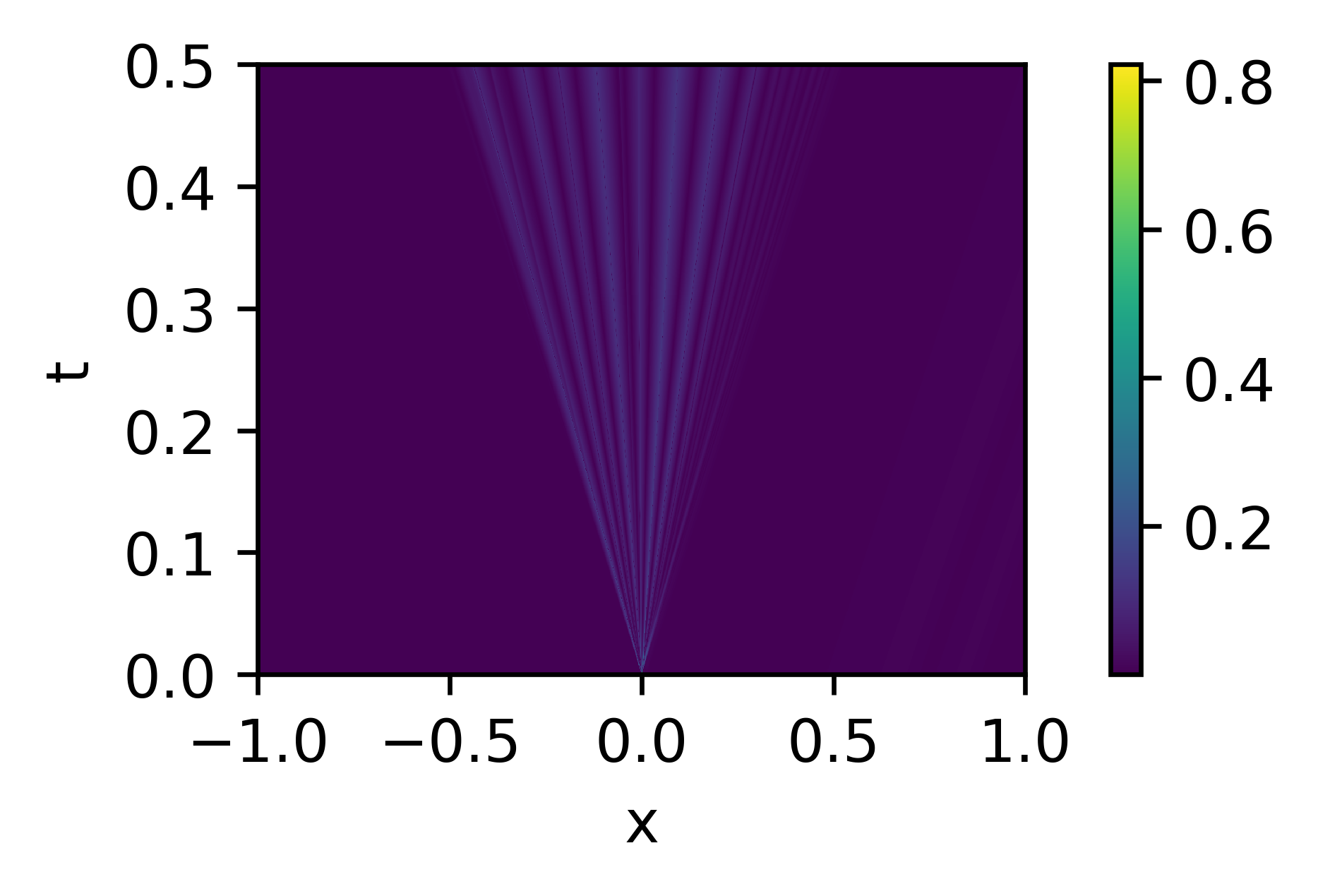}
		\caption{Error.}
		\label{App:fig:2DRarefaction_err}
	\end{subfigure}
	\begin{subfigure}[b]{0.49\textwidth}
		\centering
		\includegraphics[width=0.8\linewidth]{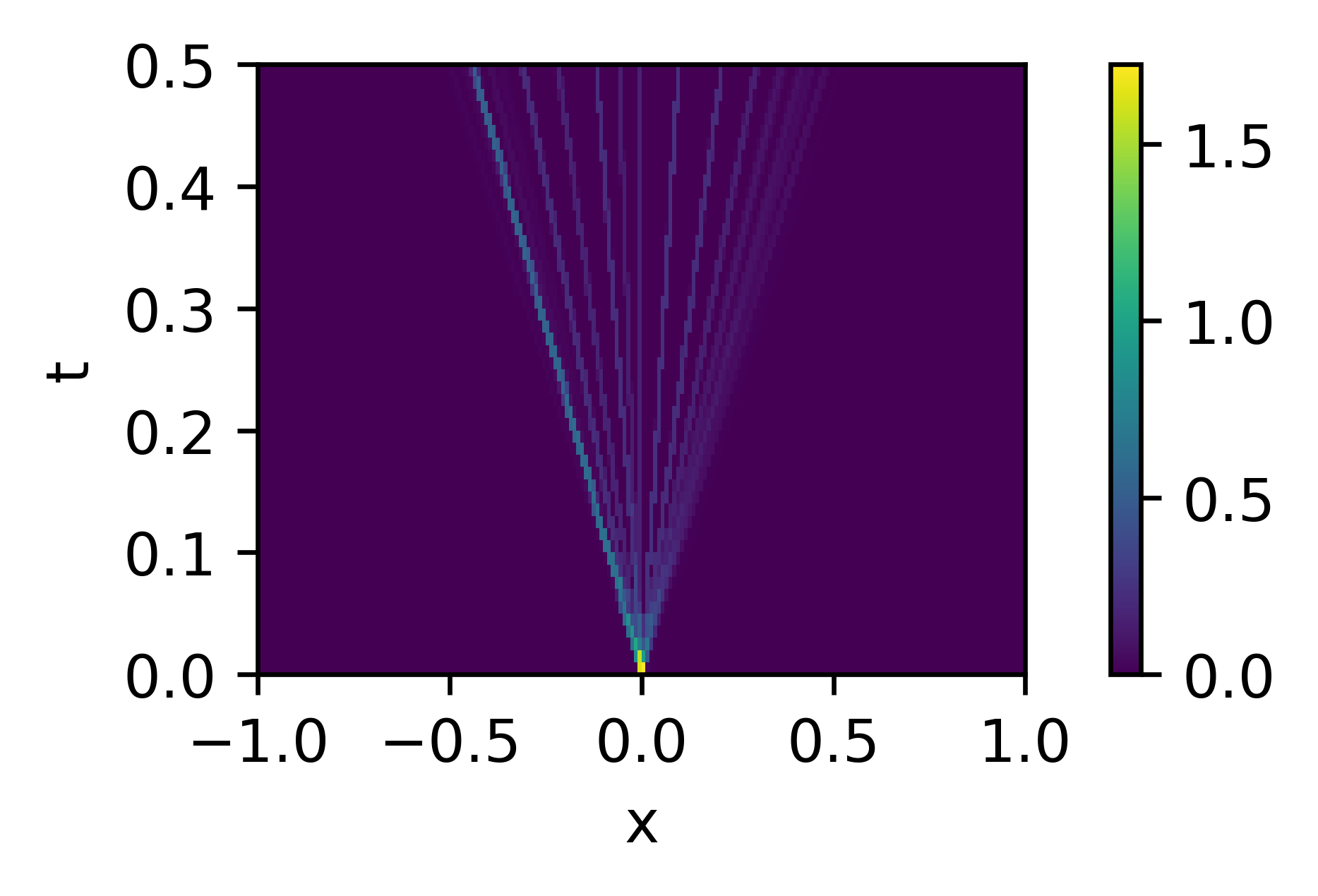}
		\caption{Estimated M-variation on a $100\times100$ grid.}
		\label{App:fig:2DRarefaction_var}
	\end{subfigure}
	\caption{The first network initialization trained for the rarefaction wave test case.}
	\label{App:fig:2DRarefaction}
\end{figure}

\begin{figure}
	\begin{subfigure}[b]{0.49\textwidth}
		\centering
		\includegraphics[width=0.8\linewidth]{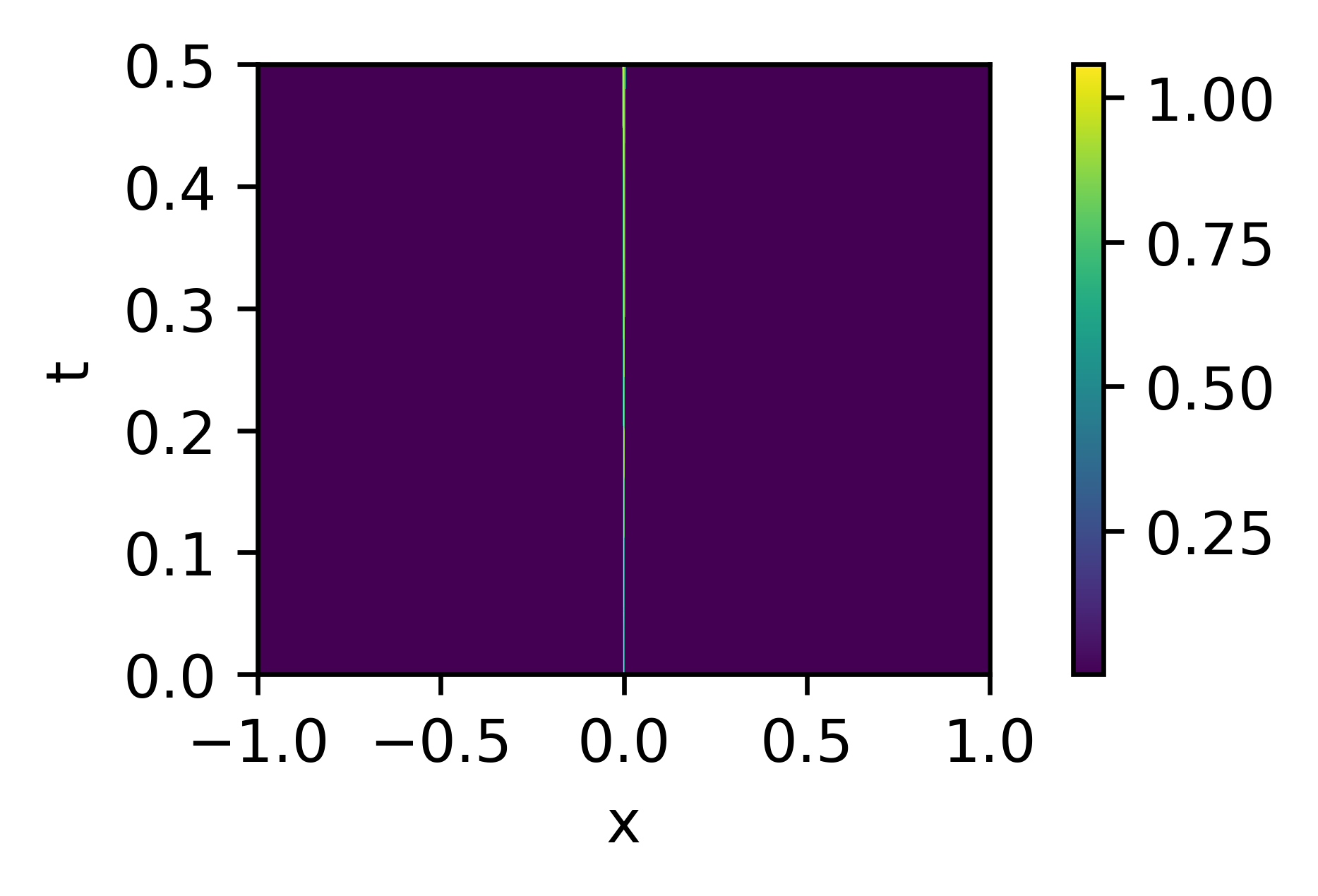}
		\caption{Standing shock.}
		\label{App:fig:StandingShockStd}
	\end{subfigure}
	\begin{subfigure}[b]{0.49\textwidth}
		\centering
		\includegraphics[width=0.8\linewidth]{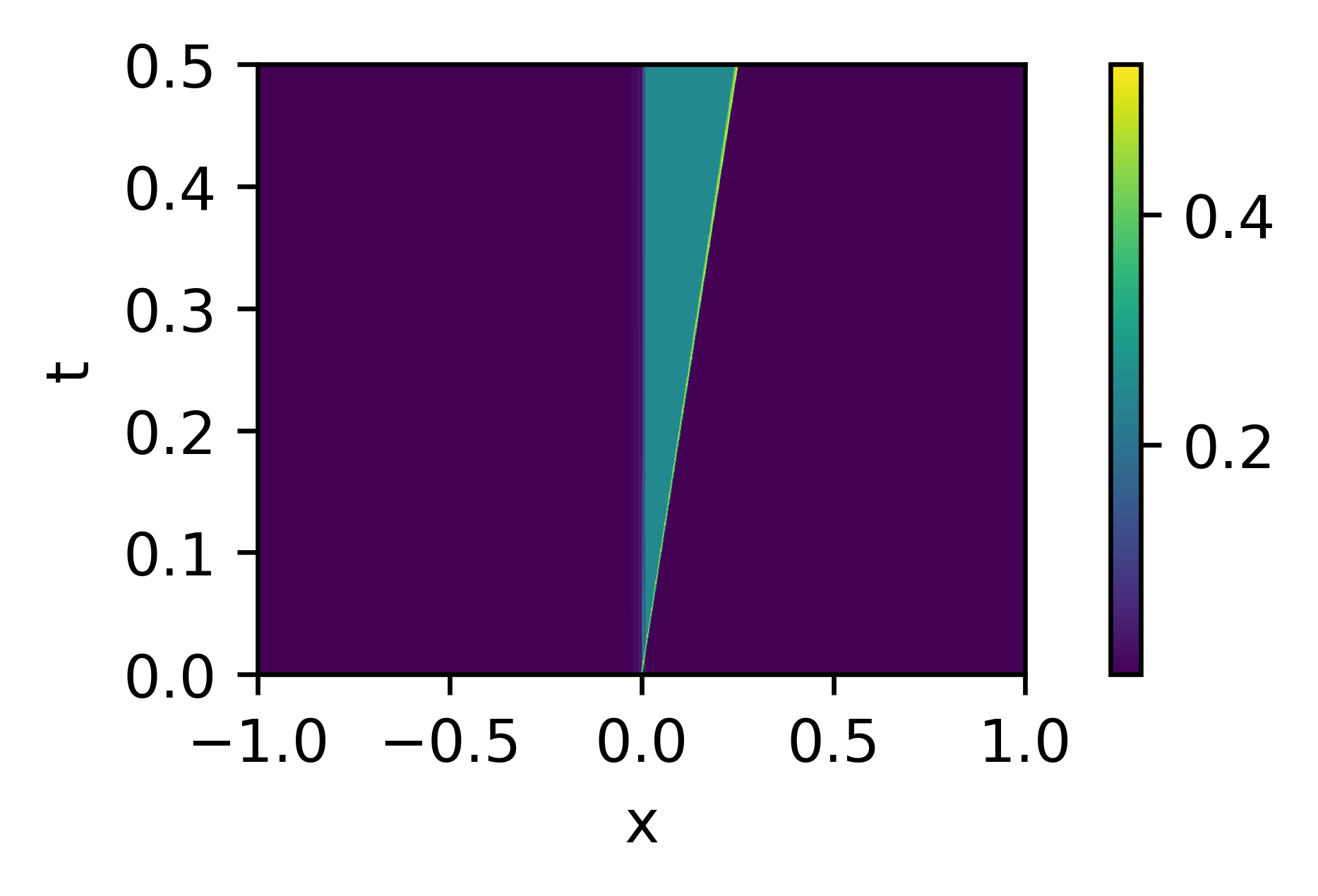}
		\caption{Moving shock.}
		\label{App:fig:MovingShockStd}
	\end{subfigure}
	
	\begin{subfigure}[b]{0.49\textwidth}
		\centering
		\includegraphics[width=0.8\linewidth]{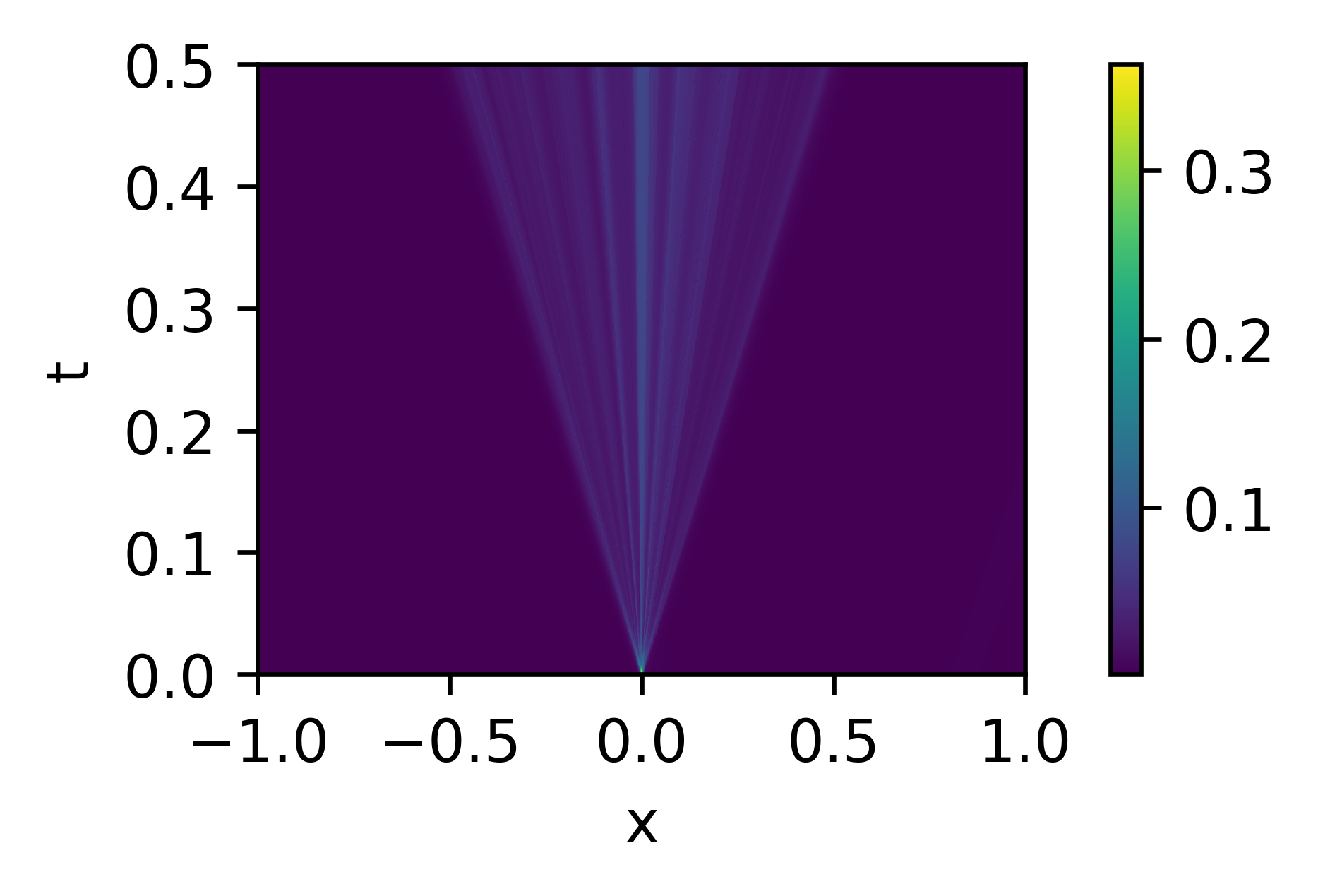}
		\caption{Rarefaction wave.}
		\label{App:fig:RarefactionStd}
	\end{subfigure}
	\begin{subfigure}[b]{0.49\textwidth}
		\centering
		\includegraphics[width=0.8\linewidth]{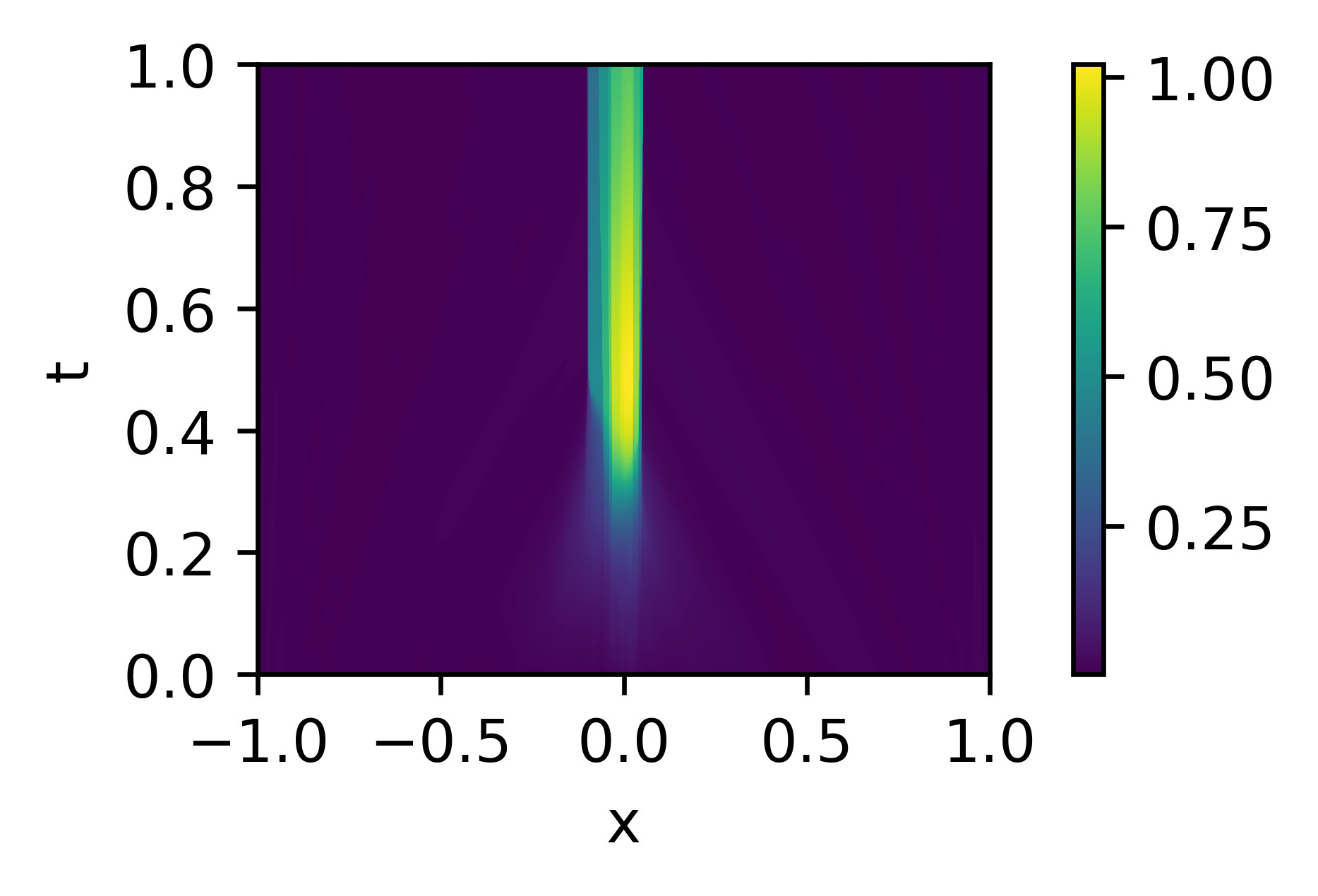}
		\caption{Sine initial data.}
		\label{App:fig:SineInitialDataStd}
	\end{subfigure}
	\caption{Pointwise empirical standard deviation of the 16 network initializations after training.}
	\label{App:fig:stdPlots}
\end{figure}
\end{document}